\documentclass[12pt,reqno]{amsart}

\usepackage{graphicx}
\usepackage{comment}
\usepackage{color}
\usepackage{hyperref}
\usepackage{amsmath, amsthm, amssymb, amsfonts}
\usepackage[mathscr]{eucal}
\usepackage[margin=2.9cm]{geometry}
\usepackage{verbatim}
\usepackage{enumitem}
\usepackage{chngcntr}
\newtheorem{theorem}{Theorem}[section]
\newtheorem{lemma}[theorem]{Lemma}

\theoremstyle{definition}

\numberwithin{subcase}{case}

\theoremstyle{definition}

\newtheorem{remark}[]{\textbf{Remark}}

\numberwithin{equation}{section}

\begin{document}
\title{Subconvexity for $GL(3)\times GL(2)$ and $GL(3)$ $L$-functions in $GL(3)$ spectral aspect}
\author{Prahlad Sharma \vspace{-1cm}}
\address{School of Mathematics, Tata Institute of Fundamental Research, Mumbai}
\email{prahlad@math.tifr.res.in}
\maketitle	
\begin{abstract}
Let $f$ be a $SL(2,\mathbb{Z})$ holomorphic cusp form or the Eisenstien series $E(z,1/2)$  and $\pi$ be a $SL(3,\mathbb{Z})$ Hecke-Maass cusp form with its Langlands parameter $\mu$ in generic position i.e. away from Weyl chamber walls and away from self dual forms. We study an amplified second moment $\sum_{j} A(\pi_j)|L(1/2,\pi_j\times f)|^2$ and deduce the subconvexity bound
\begin{equation*}
L(1/2,\pi\times f)\ll_{f,\epsilon} \|\mu\|^{3/2-1/2022+\epsilon}.
\end{equation*} 
As a corollary, when $f=E(z,1/2)$, we also obtain the subconvexity bound
\begin{equation*}
L(1/2,\pi)\ll_{\epsilon} \|\mu\|^{3/4-1/4044+\epsilon}.
\end{equation*}
\end{abstract}
\vspace{8mm}
\section{Introduction}
Let $f$ be a $SL(2,\mathbb{Z})$ holomorphic cusp form with the $n$-th fourier coefficient $\lambda_f(n)$. For an $SL(3,\mathbb{Z})$ Hecke Maass cusp form $\pi$, we denote its Langlands parameter by $\mu=\mu_{\pi}=(\mu_1,\mu_2,\mu_3)$. This triplet satisfies $\mu_1+\mu_2+\mu_3=0$, normalised such that Ramanujan predicts $\mu\in(i\mathbb{R})^3$. Let $\pi_0$ be such that its Langlands parameter $\mu_0=(\mu_{0,1},\mu_{0,2},\mu_{0,3})$ is in generic position, i.e. there exists constants $C>c>0$ such that
\begin{equation}
c\leq\frac{|\mu_{0,j}|}{\|\mu_0\|}\leq C\,\,\,(1\leq j\leq 3),\,\,\,\hbox{and}\,\,\, c\leq\frac{|\mu_{0,j}-\mu_{0,i}|}{\|\mu_0\|}\leq C\,\,\,(1\leq j\leq 3),
\end{equation}
Note that for a suitable choice of $c$ and $C$, this set cover $99\% $ of Maass forms. The Rankin-Selberg product of $\pi_0$ with $f$ is given by
\begin{equation}
L(s,\pi_0\times f)= \mathop{\sum\sum}_{r,n=1}^{\infty}\frac{A_{\pi_0}(r,n)\lambda_{f}(n)}{(nr^2)^s}\,,
\end{equation}
which converges absolutely for $\Re(s)>1$. This series extends to a entire function and satisfies functional equation of the Reimman type and has conductor $\|\mu_0\|^{6}$. Consequently, the Phragmen-Lindelof principle yeilds the convexity bound
\begin{equation*}
L(1/2,\pi_0\times f)\ll_{\pi,f,\epsilon}\|\mu_0\|^{\frac{3}{2}+\epsilon}.
\end{equation*} The Lindel\"{o}f hypothesis asserts that the exponent $3/2+\epsilon$ can be replaced by any positive number. In this paper we prove the following subconvex bound.
\begin{theorem}\label{subcon}
Let $\pi_0$ be an $SL(3,\mathbb{Z})$ Hecke-Maass cusp form with Langlands parameter $\mu_0$ in generic position. Let $f$ be a $SL(2,\mathbb{Z})$ holomorphic Hecke cusp form or the Eisenstein series $E(z,1/2)$. Then
\begin{equation}
L(1/2,\pi_0\times f)\ll_{f,\epsilon}\|\mu_0\|^{3/2-1/1960+\epsilon}\,.
\end{equation}
\end{theorem}Note that when $f(z)=E(z,1/2)$,
\begin{equation}
L(s,\pi_0\times f)=\mathop{\sum\sum}_{r,n=1}^{\infty}\frac{A_{\pi_0}(r,n)d(n)}{(nr^2)^s}=\left(\sum_{m=1}^{\infty}\frac{A_{\pi_0}(1,m)}{m^s}\right)^2=(L(s,\pi_0))^2.
\end{equation}Hence, as corollary we obtain following, which an improvement over the subconvex bound obtained in Blomer-Buttcane \cite{Bl1}.
\begin{theorem}\label{•}
Let $\pi_0$ be an $SL(3,\mathbb{Z})$ Hecke-Maass cusp form with Langlands parameter $\mu_0$ in generic position. Then
\begin{equation}
L(1/2,\pi_0)\ll_{\epsilon}\|\mu_0\|^{3/4-1/3920+\epsilon}.
\end{equation}
\end{theorem}

Subconvexity for degree three $L$-functions was first obtained by X. Li \cite{Li2} for a fixed self dual Maass form. In his series of papers  \cite{munshi2},\cite{munshi7},\cite{munshi8},\cite{munshi9},\cite{munshi10}, Munshi introduced a different approach to subconvexity through which he obtained subconvex bounds for more general degree three $L$-functions. In \cite{munshi1}, Munshi adapted his new apporach to obtain $t$-aspect subconvexity  for $GL(3)\times GL(2)$ $L$-functions, where the $GL(3)$ form is any Hecke-Maass cusp form for $SL(3, \mathbb{Z})$. Using Munshi's approach, Kumar-Mallesham-Singh \cite{kms} obtained subconvexity  for $GL(3)\times GL (2)$ in the $GL(3)$ spectral aspect, where the Langlands parameter lie in a region complementary to ours. Similar results are available for twists by Dirichlet characters \cite{Bltwist, sharma1}. In a recent breakthrough preprint, P. Nelson \cite{nelson}  established subconvexity for all `standard' $GL(n)$ $L$-functions in the spectral aspect having `uniform parameter growth'. His work is motivated by the fundamental work of Michel and Venkatesh \cite{venkatesh}.

Subconvexity using the spectral theory of $GL(3)$ automorphic forms was first obtained by Blomer-Buttcane \cite{Bl1}. We apply this spectral theory to the Ranking--Selberg product of a $GL(3)$ and a $GL(2)$ form. The main tool is the $GL(3)$ Kuznetsov formula as developed in \cite{Bl2} applied to the second moment of the Ranking--Selberg product $L(1/2, \pi \times f )$. Note that Theorem \ref{subcon} considerably improves upon the bound of Blomer-Buttcane and further generalizes to all $GL(3)\times GL(2)$ $L$-functions in a single shot. The main technical difficulty in this paper (as was in \cite{Bl1}) is the analysis of the multi-dimensional oscillatory integral arising from the Fourier transform of the integral kernel \eqref{58}. In \cite{Bl1}, with the help of the Mellin-Barnes representation for the integral kernels,  this analysis was reduced to obtaining non-trivial bounds for certain two-dimensional oscillatory integrals to which they apply, among other things, Morse theory in the form of a theorem of Milnor and Thom. Here we take a different route and proceed using the Bessel function representation (see \eqref{522}-\eqref{525}) for the integral kernels, which seems to fit well with the $GL(2)$ coefficients. We are then led to studying Fourier transforms of Bessel functions with a certain non-linear twist which we provide in the appendix. We finally reduce our problem to obtaining a non-trivial bound for a one-dimensional oscillatory integral with essentially a degree 13 polynomial as the phase function. At this stage, just a simple application of a Van der Corput type of lemma serves as the endgame. 

An important feature in the  $GL(3)\times GL(2)$ structure is that it allows us to achieve a crucial saving when the Kloosterman sums arising from the Kuznetsov formula get transformed into Ramanujan sum after the application of the $GL(2)$ Voronoi formula. However, this saving becomes ineffective when the corresponding moduli are unbalanced. Our analysis of the integral kernel takes care of this by gaining more in the ``analytic part" in this case. The handling of the long Weyl Kloosterman sum becomes more subtle in our case (as compared to \cite{Bl1}) due to the coprimality conditions in the $GL(2)$ Voronoi formula. We handle this by using a recent result of Kiral and Nakasuji \cite{kiral}, which gives a representation of the long Weyl Kloosterman sum as a finite sum of a product of two classical Kloosterman sums. Another feature of the $GL(3)\times GL(2)$ set-up is that off-diagonal after the Kuznetsov formula is reduced to only the long Weyl element, the others contributing negligibly small due to size considerations.

\section{Sketch}
We are interested in evaluating the second moment
\begin{equation}\label{2.1}
\sum_{\pi_j=\mu_0+O(1)}|L\left(1/2,\pi_j\times f\right)|^2,
\end{equation}which contains about $T^3$ terms. If we can show the off-diagonal term is $\ll T^{3-\delta}$ for some $\delta>0$, then an amplification to the above sum will establish subconvexity. For simplicity, we suppress the amplification part in the sketch. By the approximate functional equation, we roughly have 
\begin{equation}
|L\left(1/2,\pi_j\times f\right)|^2\approx T^{-3}\sum_{m\asymp T^3}\sum_{n\asymp T^3}\lambda_f(m)\overline{\lambda_f(n)} A_{\pi_j}(1,m)\overline{A_{\pi_j}(1,n)}\,.
\end{equation}Ignoring the positive contribution from the continuous spectrum, we  then apply the Kuznetsov formula to the $\mu_j$ sum in \eqref{2.1}. The only non-negligible contribution in the off diagonal comes from the long Weyl element in the Kuznetsov formula which is of the form
\begin{equation}\label{2.3}
T^{-3}\sum_{m,n\asymp T^3}\lambda_f(m)\overline{\lambda_f(n)}\sum_{D_1,D_2}\frac{S(m,1,1,n;D_1,D_2)}{D_1D_2}\Phi_{w_6}\left(\frac{mD_2}{D_1^2},\frac{nD_2}{D_1^2}\right)\,,
\end{equation}where $S(m_1,m_2,n_1,n_2;D_1,D_2)$ is a certain $SL_3$ Kloosterman sum and $\Phi_{w_6}$ is an integral transform of the form
\begin{equation}
\Phi_{w_6}(y_1,y_2)=\int_{\mu=\mu_0+O(1)}K(y_1,y_2;\mu)\hbox{spec}(\mu
)d\mu\,,
\end{equation}where $\hbox{spec}(\mu)d\mu\approx \|\mu\|^3d\mu $ is the spectral measure and $K$ is the kernel function of the $GL(3)$ Kuznetsov transform, an analogue of a Bessel $K_{2it}$ or $J_{2it}$ function. Lemma 5 of \cite{Bl1} gives a representation of this kernel as an integral over a product of two Bessel functions. This formula suggests that the typical size of $K$ is $T^{3/2}$; each Bessel function saves $T^{1/2}$, and the $u$ -integral saves $T^{1/2}$ by stationary phase. Hence the size of $\Phi_{w_6}(y_1,y_2)$ is roughly $T^{3/2}$ and Lemma \ref{5.2} suggests that  it oscillates like $e(y_1^{1/2})e(y_2^{1/2})$. Lemma \ref{5.2} also restricts $D_1,D_2\ll T$. We then execute the $m$ and $n$ sum using the Voronoi summation formula. Note that for $(D_1,D_2)=1$, we have 
\begin{equation}\label{2.4}
S(m,1,1,n;D_1,D_2)=S(m,D_2;D_1)S(n,D_1;D_2)\,,
\end{equation}and for the general case, Theorem \ref{Kiral} allows us to treat it in a similar fashion as we do for the coprime case \eqref{2.4}. Proceeding with \eqref{2.4}, we execute the $m$ sum using Voronoi summation which has conductor $D_1^2T^2$. Hence the dual sum after Voronoi is roughly bounded by $D_1T$. Note that the Klooserman sum $S(m,D_2;D_1)$ gets transformed into Ramanujan sum after the application of Voronoi summation and hence we save the whole modulus $D_1$ instead of just $D_1^{1/2}$ in the Kloosterman sum. This is a crucial point in the proof. The $n$ sum is handled similarly. Now taking the absolute values of the dual sums and executing the remaining sum trivially we see that \eqref{2.3} is bounded by $T^{-3}\times T^2\times T^2\times T^{3/2}=T^{(3-1/2)}$  and we win.  

We should however admit that the above sketch barely touches the heart of the matter, and much more happens while furnishing the details. Firstly, for the case $f=E(z,1/2)$, there is an additional main term of order $T^3$ in the off-diagonal of the Kuznetsov formula coming from the zero frequencies of the Voronoi summation for $d(n)$ and $d(m)$. We calculate the main term explicitly and save in the $L$-aspect of the amplifier. As observed in \cite{Bl1}, this follows from the existence of a zero in the Mellin transform of the Kuznetsov kernel, which becomes apparent only after some non-trivial manipulations. Lastly and most importantly, the desired square root savings in $\Phi_{w_6}$ by stationary phase are very hard to show due to the complexity of algebraic equations defining the stationary points. In Theorem \ref{average}, we manage to obtain an upper bound of $T^{2-1/140}$  on an average instead of our idealistic estimate $T^{3/2}$. With this change of bound, our bound for the off-diagonal contribution becomes $T^{3-1/140}$.

\section{Preliminaries}
We recall some basic results which we require in the upcoming sections.
\begin{lemma}[\textbf{GL(2) Voronoi summation formula.}]Let $\lambda_{f}(n)$ be the fourier coefficients of a holomorphic cusp form $f$ with weight $k_f$. Let $h$ be a compactly supported smooth function on the interval $(0,\infty)$. Let $q>0$ an integer and $a\in\mathbb{Z}$ be such that $(a,q)=1$. Then we have

\begin{equation}\label{•}
\sum_{n=1}^{\infty}\lambda_{f}(n)e\left(\frac{an}{q}\right)h(n)=\frac{1}{q}\sum_{n=1}^{\infty}\lambda_{f}(n)e\left(\frac{-\overline{a}n}{q}\right)H_{f}\left(\frac{n}{q^2}\right)\,,
\end{equation}
where
\begin{equation*}
\begin{aligned}
& H_f(y)=2\pi i^{k_f}\int_{0}^{\infty}h(x)J_{k_f-1}(4\pi\sqrt{xy})dx 
\end{aligned}
\end{equation*} 
and for $\lambda_f(n)=d(n)$ we have
\begin{equation}
\begin{aligned}
&\sum_{n=1}^{\infty}d(n)e\left(\frac{an}{q}\right)h(n)=\frac{2}{q}\int_{0}^{\infty}\left(\log \frac{\sqrt{x}}{q}+\gamma\right) h(x)\,dx\\
\frac{1}{q}\sum_{n=1}^{\infty}d(n)\int_{0}^{\infty}&\left(-2\pi e\left(\frac{-\overline{a}n}{q}\right)Y_0\left(\frac{4\pi\sqrt{nx}}{q}\right)+4e\left(\frac{\overline{a}n}{q}\right)K_0\left(\frac{4\pi\sqrt{nx}}{q}\right)\right)h(x)\,dx\,.
\end{aligned}
\end{equation}
\end{lemma}
\begin{proof}
See appendix A.4 of \cite{kmv}.
\end{proof}
We need the following asymptotics for the Bessel functions to extract the oscillation (see \cite{watson}, p. 206).
\begin{lemma}For $y>0$, the Bessel functions $J_{\nu}(y)$ and $K_{i\nu}(y)$, ($\nu\in \mathbb{R}$) satisfy the following oscillatory behaviour 
\begin{equation}\label{bes} 
J_{\nu}(y)=e^{iy}P_{ \nu}(y)+e^{-iy}Q_{ \nu}(y)\,\,\,\hbox{and}\,\,\,\, \left|y^kK_{i\nu}^{(k)}(y)\right|\ll_{k,\nu}\frac{e^{-y}}{\sqrt{y}}\,,
\end{equation}where the function $P_{ \nu}(y)$ (and similarly $Q_{ \nu}(y)$) satisfies 
\begin{equation}\label{besselweight}
y^jP_{\nu}^{(j)}(y)\ll_{j,\nu}\frac{1}{\sqrt{y}}.
\end{equation}
\end{lemma}
We will frequently encounter integrals of the form 
\[ I=\int w(t)e(h(t))dt\,,\]
where $w$ is a smooth function supported on $[a,b]$. The  next two lemma gives asymptotics of $I$ depending on the stationary point of $h(t)$ (see \cite{bky}, Section 8).
\begin{lemma}\label{s1}
Let $Y\geq 1, X,Q,U,R >0$, and suppose $w(t)$ and $h(t)$ satisfies 
\begin{equation}
w^{(j)}(t)\ll_{j}XU^{-j}
\end{equation}and
\begin{equation}
|h'(t)|\geq R \,\,\,\hbox{and}\,\,\,\,\, h^{(j)}(t)\ll_{j} YQ^{-j},\,\,\,\,\,\, j=2,3,4,...
\end{equation}then we have
\begin{equation}
I\ll_{A} (b-a)X\left( (QR/\sqrt{Y})^{-A}+(RU)^{-A}\right)\,.
\end{equation}
\end{lemma}	
\begin{lemma}\label{s2}
Let $0<\delta<1/10,X,Y,V,Q>0, Z:=X+Y+b-a+1$ and assume that 
\begin{equation*}
Y\geq Z^{3\delta},\,\,\, b-a\geq V\geq \frac{QZ^{\frac{\delta}{2}}}{\sqrt{Y}}.
\end{equation*}Assume that $w$ satisfies 
\begin{equation*}
w^{(j)}(t)\ll_{j} XV^{-j}
\end{equation*}for $j\in \mathbb{N}_0$. Suppose that there exist unique $t_0\in [a,b]$ such that $h'(t_0)=0$, and furthermore
\begin{equation*}
h''(t)\gg YQ^{-2},\,\,\,\,\,\, h^{(j)}(t)\ll_{j}YQ^{-j},\,\,\,\,\,\,\,\,\hbox{for}\,\,j=1,2,3...
\end{equation*}Then the integral $I$ has an asymptotic expansion 
\begin{equation}
I=\frac{e(h(t_0))}{\sqrt{h''(t_0)}}\sum_{n\leq 3\delta^{-1}A}p_n(t_0)+O_{\delta,A}(Z^{-A}),\,\,\,p_n(t_0)=\frac{\sqrt{2\pi}e^{i\pi/4}}{n!}\left(\frac{i}{2h''(t_0)}\right)^nG^{(2n)}(t_0),
\end{equation}where
\begin{equation}
G(t)=w(t)e^{iH(t)},\,\,\,\,\,\, H(t)=h(t)-h(t_0)-\frac{1}{2}h''(t_0)(t-t_0)^2.
\end{equation}Furthermore, each $p_n$ is a rational function in $h'',h''',....,$ satisfying
\begin{equation*}
\frac{d^j}{dt_0^j}p_n(t_0)\ll{j,n}X(V^{-j}+Q^{-j})\left((V^2Y/Q^2)^{-n}+Y^{-n/3}\right)\,.
\end{equation*}
\end{lemma}	
Using Lemma \ref{s1}, we can deduce the following analogue of the Van der Corput lemma for oscillatory integrals.
\begin{lemma}\label{s3}
Let $T>1$ and $Y\geq 1$, $X,Q,U,R>0$ be parameters bounded by $T^{O(1)}$. Let $F(t)$ a compactly supported smooth function satisfying
\begin{equation}\label{Fjder}
F^{(j)}(t)\ll_j XU^{-j},\,\,j\geq 1.
\end{equation}Suppose $\psi(t)$ is a smooth function such that
\begin{equation}\notag
\psi'(t)=A\cdot\frac{P(t)}{f(t)},
\end{equation}where $P(t)$ is a degree $d(\geq 1)$ monic polynomial and $f(t)$ is smooth function with $f(t)\gg 1$ for $t$ in the support of $F$. Furthermore, suppose
\begin{equation}\label{psidernew}
\psi^{(j)}(t)\ll_j YQ^{-j},\,\,j\geq 1.
\end{equation}Then
\begin{equation}\notag
\begin{aligned}
\int F(t)e(T\psi(t))\,dt&\ll_{\epsilon}T^{\epsilon}\sup(F)\left((AT)^{-1/(d+1)}+(ATU)^{-1/d}+(ATQ/Y^{1/2})^{-1/d}\right) \\
&\,\,\,\,+O_K(T^{-K}).
\end{aligned}
\end{equation}
\end{lemma}	
\begin{proof}
Fix a smooth function $w$ satisfying $w(t)=1$ if $t\in [-1/2, 1/2]$ and $w(t)=0$ if $|t|>1$. Denote
\begin{equation}\label{Cdef}
C:= T^{\epsilon}\left((AT)^{-1/(d+1)}+(ATU)^{-1/d}+(ATQ/Y^{1/2})^{-1/d}\right) ,
\end{equation}and define
\begin{equation}\label{defG}
G(t):= \prod_{i=1}^{d}(1-w(C^{-1}(t-\Re z_i))).
\end{equation}where $z_i, i=1,2,\cdots ,d$, are the roots of $P(t)$. Note that $t\in\text{supp}(G)$ implies $t-\Re z_i\gg C$ for all $1\leq i\leq d$. This in turn implies
\begin{equation}\label{1stderv}
\psi'(t)=A\cdot\frac{P(t)}{f(t)}\gg A C^d,\,\,t\in \text{supp}(G).
\end{equation}Now consider the integral
\begin{equation}\notag
\int G(t)F(t)e(T\psi(t))\,dt.
\end{equation}Then from \eqref{defG} and \eqref{Fjder}, we obtain
\begin{equation}\label{gfder}
\frac{\partial^j}{\partial t^j} G(t)F(t)\ll_j \max\{1,X\}(C+U)^{-j}.
\end{equation}Using \eqref{1stderv}, \eqref{gfder}, \eqref{psidernew} and invoking Lemma \ref{s1}, we obtain
\begin{equation}\label{par}
\int G(t)F(t)e(T\psi(t))\,dt\ll_K \max\{1,X\} \Big((TAC^d(C+U))^{-K}+(TAC^dQ/Y^{1/2})^{-K}\Big).
\end{equation}From the definition of $C$ in \eqref{Cdef}, it follows
\begin{equation}\notag
TAC^{d+1}\gg TA \cdot T^{\epsilon} (AT)^{-1}= T^{\epsilon},\,\,\,\,\,\text{and}\,\,\,\,TAC^dU\gg TA\cdot T^{\epsilon}(ATU)^{-1}\cdot U=T^{\epsilon},
\end{equation}Hence 
\begin{equation}\notag
TAC^d(C+U)\gg T^{\epsilon}.
\end{equation}Similarly, we obtain
\begin{equation}\notag
TAC^dQ/Y^{1/2}\gg T^{\epsilon}TAQ/Y^{1/2}\cdot (ATQ/Y^{1/2})^{-1}=T^{\epsilon}.
\end{equation}Substituting the last two obtained lower bounds into \eqref{par}, we obtain
\begin{equation}\notag
\int G(t)F(t)e(T\psi(t))\,dt\ll_K T^{-K},
\end{equation}and consequently
\begin{equation}\notag
\begin{aligned}
\int F(t)e(T\psi(t))\,dt&=\int G(t)F(t)e(T\psi(t))\,dt+\int (1-G(t))F(t)e(T\psi(t))\,dt\\
&=\int (1-G(t))F(t)e(T\psi(t))\,dt+O_K(T^{-K}).
\end{aligned}
\end{equation}The lemma follows after executing the last integral trivially.
\end{proof}The next lemma generalises the chain rule to higher derivatives. This is will required to estimate derivatives of the weight functions obtained from stationary phase analysis.
\begin{lemma}[Fa\'{a} di Bruno's formula]\label{bruno}
Let $f,g : \mathbb{R}\to \mathbb{R}$ be two smooth functions. Then
\begin{equation}\notag
\frac{d^n}{dx^n}f(g(x))= \sum\frac{n!}{m_1!1!^{m_1}m_2!2!^{m_2}\cdots m_n!n!^{m_n}}\cdot f^{(m_1+m_2+\cdots+m_n)}(g(x))\cdot\prod_{j=1}^{n}\left(g^{(j)}(x)\right)^{m_j},
\end{equation}where the sum is over all $n$-tuples of non-negative integers $(m_1,\cdots,m_n)$ satisfying the constraint
\begin{equation}\notag
1\cdot m_1+2\cdot m_2+\cdots+n\cdot m_n=n.
\end{equation}
\end{lemma}	

\section{The set up}	
For $0\leq c\leq \infty $ let
\begin{equation}\label{31}
 \Lambda_{c}:=\{\mu\in \mathbb{C}^3\,\,|\,\, \mu_1+\mu_2+\mu_3=0,\,\,|\Re\mu_j|\leq c\}
 \end{equation} and
 \begin{equation}
 \Lambda'_{c}:=\{\mu\in \Lambda_c\,\,|\,\,\{-\mu_1,-\mu_2,-\mu_3\}=\{\bar{\mu_1},\bar{\mu_2}, \bar{\mu_3}\}\}.
 \end{equation}
 In the Lie algebra $\Lambda_{\infty}$ we will simultaneously use $\mu$ and $\nu=(\nu_1,\nu_2,\nu_3)$ coordinates, defined by
 \begin{equation}
 \nu_1=\frac{1}{3}(\mu_1-\mu_2),\,\,\,\nu_2=\frac{1}{3}(\mu_2-\mu_3),\,\,\,\frac{1}{3}(\mu_3-\mu_1).
 \end{equation}Let $\mathcal{W}$ denote the Weyl group,
 \[\mathcal{W}:=\left\{I_3, w_2=\left(\begin{smallmatrix}1 & 0 & 0\\0 & 0 & 1\\0 & 1 & 0 \end{smallmatrix}\right), w_3=\left(\begin{smallmatrix}0 & 1 & 0\\1 & 0 & 0\\0 & 0 & 1 \end{smallmatrix}\right), w_4= \left(\begin{smallmatrix}0 & 1 & 0\\0 & 0 & 1\\1 & 0 & 0 \end{smallmatrix}\right), w_5=\left(\begin{smallmatrix}0 & 0 & 1\\1 & 0 & 0\\0 & 1 & 0 \end{smallmatrix}\right), w_6=\left(\begin{smallmatrix}0 & 0 & 1\\0 & 1 & 0\\1 & 0 & 0 \end{smallmatrix}\right) \right\},\]which acts on $\mu=(\mu_1,\mu_2,\mu_3)$ by permutations. Let $\pi_0$ be our preferred $SL(3,\mathbb{Z})$ Maass cusp from with with Hecke eigenvalues $A_{\pi_0}(1,n)$ and Langlands parameter $(\mu_{0,1} \mu_{0,2},\mu_{0,3})\in \Lambda_0$, and assume that 
 \begin{equation}
  |\mu_{0,j}|\asymp |\nu_{0,j}|\asymp T\,\,\,(j=1,2,3)\,.
  \end{equation} 
  $L(1/2,\pi_0\times f)$ has analytic conductor $T^3$. Hence, from the approximate functional equation (see Th. 5.3 and Prop. 5.4 in \cite{iwaniec}) we essentially get
\begin{equation}\label{FE}
L\left(\frac{1}{2},\pi_0\times f\right)\ll \left|\sum_{m=1}^{\infty}\frac{\lambda_{\pi_0\times f}(m)}{m^{1/2}}V\left({\frac{m}{T^{3}}}\right)\right|,
\end{equation}where $$\lambda_{\pi_0\times f}(m)=\mathop{\sum\sum}_{nr^2=m}A_{\pi_0}(r,n)\lambda_{f}(n),$$and the smooth function $V$ satisfies
\begin{equation}\notag
x^jV^{(j)}(x)=O_{A}(1+|x|)^{-A}.
\end{equation}


\begin{lemma}\label{trunc}
\begin{equation}\label{dyd}
L\left(\frac{1}{2},\pi_0\times f\right)\ll T^{\epsilon}\sup_{r\leq T^{\theta}}\sup_{\frac{T^{3-\theta}}{r^2}\leq N\leq \frac{T^{3+\epsilon}}{r^2}}\frac{|S_r(N)|}{N^{1/2}}+T^{(3-\theta)/2}\,,
\end{equation}
where
 \begin{equation*}
S_r(N)=\sum_{n=1}^{\infty}A_{\pi_0}(r,n)\lambda_f(n)V\left(\frac{n}{N}\right)\,,
\end{equation*} where $V$ is a smooth function supported in $[1,2]$ and satisfies $V^{(j)}(x)\ll_{j} 1$.

\end{lemma}	
\begin{proof}From the approximate  functional equation \eqref{FE}, one has
\begin{equation}\label{rr}
\begin{aligned}
L\left(\frac{1}{2},\pi_0\times f\right)&\ll\left|\mathop{\sum\sum}_{n,r=1}^{\infty}\frac{A_{\pi_0}(r,n)\lambda_{f}(n)}{(nr^2)^{1/2}}V\left(\frac{nr^2}{T^{3}}\right)\right|\\
&\ll\Bigg|\mathop{\sum\sum}_{nr^2\leq T^{3+\epsilon}}\frac{A_{\pi_0}(r,n)\lambda_{f}(n)}{(nr^2)^{1/2}}V\left(\frac{nr^2}{T^{3}}\right)\Bigg|+T^{-2019}\\
&=\Bigg|\sum_{r\leq T^{(3+\epsilon)/2}}\frac{1}{r}\sum_{n\leq T^{3+\epsilon}/r^2}\frac{A_{\pi_0}(r,n)\lambda_{f}(n)}{n^{1/2}}V\left(\frac{nr^2}{T^{3}}\right)\Bigg|+T^{-2019}\,.
\end{aligned} 
\end{equation}
We divide the range in the last summation as
\begin{equation*}
\sum_{r\leq T^{(3+\epsilon)/2}}\sum_{n\leq\frac{ T^{3+\epsilon}}{r^2}}=\sum_{r\leq T^{\theta}}\sum_{\frac{T^{3-\theta}}{r^2}\leq n\leq \frac{T^{3+\epsilon}}{r^2}}+\sum_{r\leq T^{\theta}}\sum_{n< \frac{T^{3-\theta}}{r^2}}+\sum_{r>T^{\theta}}\sum_{n\leq\frac{ T^{3+\epsilon}}{r^2}}
\end{equation*}where an optimal $\theta>0$ will be chosen later. Using the Ramanujan bound on average
\begin{equation*}
\mathop{\sum\sum}_{n_1^2n_2\leq x}|A(n_1,n_2)|^2\ll x^{1+\epsilon}
\end{equation*} one sees that the last two ranges contributes at most $T^{(3-\theta)/2}$ to \eqref{rr}. Hence we have
\begin{equation*}
L\left(\frac{1}{2},\pi_0\times f\right)\ll\Bigg|\sum_{r\leq T^{\theta}}\frac{1}{r}\sum_{\frac{T^{3-\theta}}{r^2}\leq n\leq \frac{T^{3+\epsilon}}{r^2}}\frac{A_{\pi_0}(r,n)\lambda_{f}(n)}{n^{1/2}}V\left(\frac{nr^2}{T^{3}}\right)\Bigg|+T^{(3-\theta)/2}.
\end{equation*} Using a smooth dyadic partition of unity $W$, we see that inner sum above is at most
\begin{equation*}
 \sup_{\frac{T^{3-\theta}}{r^2}\leq N\leq \frac{T^{3+\epsilon}}{r^2}}\Bigg|\sum_{n=1}^{\infty}\frac{A_{\pi_0}(r,n)\lambda_{f}(n)}{n^{1/2}}W\left(\frac{n}{N}\right)V\left(\frac{nr^2}{T^{3}}\right)\Bigg|.
\end{equation*} Using partial summation one sees that this is essentially bounded by
 \begin{equation*}
\sup_{\frac{T^{3-\theta}}{r^2}\leq N\leq \frac{T^{3+\epsilon}}{r^2}}\frac{|S_r(N)|}{N^{1/2}}\,
\end{equation*}where 
\begin{equation*}
S_r(N)=\sum_{n=1}^{\infty}A_{\pi_0}(r,n)\lambda_f(n)V_{r,N}\left(\frac{n}{N}\right)\,,
\end{equation*} and $V_{r,N}(x)=W(x)V(Nr^2x/T^{3})$. Note that $V_{r,N}(x)$ is supported on $[1,2]$ and satisfies $V_{r,N}^{(j)}\ll_{j} 1$ (bounds independent of $r,N$). Henceforth we ignore the dependence on $r,N$ and assume $V_{r,N}$ is same function for all $r,N$ and call it $V(x)$ (abusing notation). The claim follows.
\end{proof}
As seen from the sketch above, the unamplified moment $\sum |L(1/2,\pi\times f|^2$ fails to beat the convexity at the diagonal. To overcome this, we use the following amplifier introduced in \cite{Bl1}.
\subsection{The amplifier}
 Let the sequence $\{\pi_j\}$ of Maass cusp forms be a orthonormal basis for  $\mathcal{L}^2\left(\mathfrak{h}^3/\Gamma\right)$. Consider the amplifier 
\[
A(\pi)=\sum_{j=1}^{3}\left|\sum_{\substack{ l\sim L \\ l\,prime}}A_\pi(1,l^j)\overline{x(l^j)}\right|^2\,,
\] where 
\[x(n) :=\hbox{sgn} (A_{\pi_0}(1.n))\in S^1\cup \{0\}\,.\]
Note that form the Hecke relation
\begin{equation}\notag
A_{\pi}(1,l)A_{\pi}(1,l^2)=A_{\pi}(1,l^3)+A_{\pi}(1,l)A_{\pi}(l,1)-1
\end{equation}it follows $\max\{|A_{\pi}(1,l)|, |A_{\pi}(1,l^2)|, |A_{\pi}(1,l^3)| \}\gg 1$ and hence by the Cauchy-Schwarz inequality 
\begin{equation}\label{amp}
A(\pi_0)\gg \left(\sum_{j=1}^{3}\sum_{\substack{l\sim L\\l\,prime}}|A_{\pi_0}(1,l^j)|\right)^2\gg \Big(\sum_{\substack{l\sim L\\ l\,prime}}1\Big)^2\gg L^{2-\epsilon}.
\end{equation}

 We now insert the amplifiers $A(\pi_j)$ into Lemma \ref{trunc} to get,
\begin{equation}\label{3.11}
|L(1/2,\pi_0\times f)|^2\ll T^{\epsilon}\sup_{r\leq T^{\theta}}\sup_{\frac{T^{3-\theta}}{r^2}\leq N\leq \frac{T^{3+\epsilon}}{r^2}}\frac{\sum_{\mu_j=\mu_0+O(T^{\epsilon})}A(\pi_j)|S_{r,j}(N)|^2}{NL^2}+T^{(3-\theta)}\,.
\end{equation}We will make the spectral average $\sum_{\mu_j=\mu_0+O(T^{\epsilon})}$ more precise in the upcoming sections. Our main object of study now becomes
\begin{equation}\label{2.6}
\sum_{\mu_j=\mu_0+O(T^{\epsilon})}A(\pi_j)|S_{r,j}(N)|^2,
\end{equation}for which we will eventually establish
\begin{theorem}\label{kutz}
\begin{equation}\label{2.7}
\begin{aligned}
\sum_{\mu_j=\mu_0+O(T^{\epsilon})}A(\pi_j)|S_{r,j}(N)|^2 \ll NT^3L+\frac{N^2r^2L^8}{T^{1/280}} +\sum_{k=1}^3\sum_{\substack{(k_0,k_1,k_2)\\k_0+k_1+k_2=k\\L^{k_1}\ll r}}\frac{NT^3L^2}{L^{k_2+k_0}}.
\end{aligned}
\end{equation}
\end{theorem}
\begin{proof}[\textbf{Proof of Theorem \ref{subcon}}]
Substituting \eqref{2.7} in \eqref{3.11} we get
\begin{equation}
\begin{aligned}
|L(1/2,\pi_0\times f|^2 \ll \frac{T^3}{L}+T^{(3-1/280)}L^6+T^{3-\theta} + \sup_{\substack{\frac{T^{3-\theta}}{r^2}\leq N\leq \frac{T^{3+\epsilon}}{r^2}\\r\leq T^{\theta}}}\sum_{k=1}^3\sum_{\substack{(k_0,k_1,k_2)\\k_0+k_1+k_2=k\\L^{k_1}\ll r}}\frac{T^3}{L^{k_2+k_0}}.
\end{aligned}
\end{equation} Equating the first two terms, we choose $L=T^{1/1960}$. We now choose $\theta=1/1960-\epsilon$. With these choices, note that $k_1$ must equal zero in the third term and consequently $k_2+k_0\geq 1$. Hence the third term is dominated by the first term and we get
\begin{equation}
L(1/2,\pi_0\times f)\ll_{\epsilon} T^{(3/2-1/1960+\epsilon)}\,.
\end{equation}
\end{proof}
To begin with the proof of Theorem \ref{kutz}, we open the absolute value square and apply  the $GL(3)$ Kuznetsov formula. To bring it in the proper set-up, we need to rearrange \eqref{2.6} further. From the Hecke relations, we get
\begin{equation}
A_{\pi_j}(1,l^k)A_{\pi_j}(r,n)=\sum_{\substack{d_0d_1d_2=l^k\\d_1|r\\d_2|n}}A_{\pi_j}\left(\frac{rd_2}{d_1},\frac{nd_0}{d_2}\right)=\sum_{\substack{(k_0,k_1,k_2)\\k_0+k_1+k_2=k\\l^{k_1}|r,\,l^{k_2}|n}}A_{\pi_j}(rl^{k_2-k_1},nl^{k_0-k_2})\,.
\end{equation}Hence by an application of the Cauchy-Schwarz inequality, \eqref{2.6} becomes
\begin{equation}\label{CS}
\sum_{\mu_j=\mu_0+O(T^{\epsilon})}A(\pi_j)|S_{r,j}(N)|^2\ll\sum_{k=1}^{3}\sum_{\substack{\bar{k}=(k_0,k_1,k_2)\\k_0+k_1+k_2=k}}\sum_{\mu_j=\mu_0+O(T^{\epsilon})}|\tilde{S}_{\bar{k},r,j}(N)|^2\,,
\end{equation} where
\begin{equation*}
\tilde{S}_{\bar{k},r,j}(N)=\sum_{l\asymp L}\sum_{\substack{n\asymp N\\l^{k_2}|n}}A_{\pi_j}(rl^{k_2-k_1},nl^{k_0-k_2})\lambda (n)V\left(\frac{n}{N}\right)\,.
\end{equation*} Opening the absolute value square we have

\begin{equation}\label{2.10}
\begin{aligned}
\sum_{\mu_j=\mu_0+O(T^{\epsilon})}|\tilde{S}_{\bar{k},r,j}(N)|^2
=&\sum_{l_1,l_2\asymp L}\sum_{n}\sum_{m}\lambda(ml_1^{k_2})\overline{\lambda(nl_2^{k_2})}V\left(\frac{ml_1^{k_2}}{N}\right)V\left(\frac{nl_2^{k_2}}{N}\right)
\\ &\times\sum_{\mu_j=\mu_0+O(T^{\epsilon})}A_{\pi_j}(rl_1^{k_2-k_1},ml_1^{k_0})\overline{A_{\pi_j}(rl_2^{k_2-k_1},nl_2^{k_0})}\,.
\end{aligned}
\end{equation}
We are now in position to apply the Kuznetsov formula to the spectral sum above.

\section{The $GL(3)$ Kutznetsov}We first introduce some notations. 
\subsection{Kloosterman sums.} For $n_1,n_2,m_1,m_2,D_1,D_2\in\mathbb{N}$, we have the following two types of Kloosterman sums
\begin{equation}
\tilde{S}(n_1,n_2,m_1;D_1,D_2):=\sum_{\substack{C_1(\bmod D_1),C_2(\bmod D_2)\\(C_1,D_1)=(C_2,D_2/D_1)=1}}e\left(n_2\frac{\bar{C_1}C_2}{D_1}+m_1\frac{\bar{C_2}}{D_2/D_1}+n_1\frac{C_1}{D_1}\right)
\end{equation}for $D_1\,|\,D_2$, and 
\begin{equation}\label{4.2}
\begin{aligned}
&S(n_1,m_2,m_1,n_2;D_1,D_2)\\
&=\sum_{\substack{B_1,C_1(\bmod D_1)\\B_2,C_2(\bmod D_2)\\D_1C_2+B_1B_2+D_2C_1\equiv 0(\bmod D_1D_2)\\(B_j,C_j,D_j)=1}}e\left(\frac{n_1B_1+m_1(Y_1D_2-Z_1B_2)}{D_1}+\frac{m_2B_2+n_2(Y_2D_1-Z_2B_1)}{D_2}\right),
\end{aligned}
\end{equation}where $Y_jB_j+Z_jC_j\equiv 1(\bmod D_j)$ for $j=1,2$. \\

We have the following result of E.M. Kiral and M. Nakasuji \cite{kiral} which gives a decomposition of $S(n_1,m_2,m_1,n_2;D_1,D_2)$ as a sum of product of two classical $SL_2$ Kloosterman sums.
\begin{theorem}[Kiral, Nakasuji]\label{Kiral} We have
\begin{equation}\label{kiral}
\begin{aligned}
&S(m_2,m_1,n_1,n_2;D_1,D_2)\\
&\quad\quad\quad\quad\quad=\sum_{f|(D_1,D_2)}f\sum_{\substack{y(\bmod f)\\n_1\frac{D_2}{f}+m_1\frac{D_1}{f}y\equiv 0(\bmod f)}}S\left(m_2,M_f(y);\frac{D_1}{f}\right)S\left(n_2,\tilde{M}_f(y);\frac{D_2}{f}\right)\,,
\end{aligned}
\end{equation}where
\[M_f(y)=\frac{n_1D_2+m_1D_1y}{f^2}\,\,\,\,\,\,\hbox{and}\,\,\,\,\,\,\,\tilde{M}_f(y)=\frac{n_1D_2\bar{y}+m_1D_1}{f^2}\,.
\]
\end{theorem}
\subsection{Integral kernels.} Following  \cite[Theorem 2 \& 3]{butt}, we define the following integral kernels in terms of Mellin-Barnes representations. For $s\in\mathbb{C}, \mu\in\Lambda_{\infty}$, define the meromorphic function 
\begin{equation}
\tilde{G}^{\pm}(s,\mu):=\frac{\pi^{-3s}}{12288\pi^{7/2}}\left(\prod_{j=1}^{3}\frac{\Gamma(\frac{1}{2}(s-\mu_j))}{\Gamma(\frac{1}{2}(1-s+\mu_j))}\pm i\prod_{j=1}^{3}\frac{\Gamma(\frac{1}{2}(1+s-\mu_j))}{\Gamma(\frac{1}{2}(2-s+\mu_j))}\right)\,,
\end{equation} and for $s=(s_1,s_2)\in\mathbb{C}^2,\mu\in\Lambda_{\infty}$, define the meromorphic function
\begin{equation}\label{55}
G(s,\mu):=\frac{1}{\Gamma(s_1+s_2)}\prod_{j=1}^{3}\Gamma(s_1-\mu_j)\Gamma(s_2+\mu_j).
\end{equation} The latter is essentially the double Mellin transform of the $GL(3)$ Whittaker function. We also define the following trigonometric functions :
\begin{equation}\label{56}
\begin{aligned}
&S^{++}(s,\mu):=\frac{1}{24\pi^2}\prod_{j=1}^{3}\cos \left(\frac{3}{2}\pi\nu_j\right),\\
&S^{+-}(s,\mu):=-\frac{1}{32}\frac{\cos(\frac{3}{2}\pi\nu_2)\sin(\pi(s_1-\mu_1))\sin(\pi(s_2+\mu_2))\sin(\pi(s_2+\mu_3))}{\sin(\frac{3}{2}\pi\nu_1)\sin(\frac{3}{2}\pi\nu_3)\sin(\pi(s_1+s_2)) },\\
&S^{-+}(s,\mu):=-\frac{1}{32}\frac{\cos(\frac{3}{2}\pi\nu_1)\sin(\pi(s_1-\mu_1))\sin(\pi(s_1-\mu_2))\sin(\pi(s_2+\mu_3))}{\sin(\frac{3}{2}\pi\nu_2)\sin(\frac{3}{2}\pi\nu_3)\sin(\pi(s_1+s_2)) },\\
&S^{--}(s,\mu):=-\frac{1}{32}\frac{\cos(\frac{3}{2}\pi\nu_3)\sin(\pi(s_1-\mu_2))\sin(\pi(s_2+\mu_2))}{\sin(\frac{3}{2}\pi\nu_2)\sin(\frac{3}{2}\pi\nu_1) }.\\
\end{aligned}
\end{equation}
For $y\in\mathbb{R}\setminus{0} $ with $\hbox{sgn}(y)=\epsilon $, define
\begin{equation}
K_{w_4}(y;\mu):=\int_{-i\infty}^{i\infty}|y|^{-s}\tilde{G}^{\epsilon}(y)(s,\mu )\frac{ds}{2\pi i}.
\end{equation} and for $y=(y_1,y_2)\in(\mathbb{R}\setminus {0})^2$ with $\hbox{sgn}(y_1)=\epsilon_1, \hbox{sgn}(y_2)=\epsilon_2$, define
\begin{equation}\label{58}
K_{w_6}^{\epsilon_1,\epsilon_2}(y;\mu)=\int_{-i\infty}^{i\infty}\int_{-i\infty}^{i\infty}|4\pi^2y_1|^{-s_1}|4\pi^2y_2|^{-s_2}G(s,\mu)S^{\epsilon_1,\epsilon_2}(s,\mu)\frac{ds_1ds_2}{(2\pi i)^2}\,.
\end{equation}
\subsection{Normalising factors}
The following normalising appears in Kuznetsov formula :
\begin{equation}
N(\pi):= \|\phi\|^2\prod_{j=1}^{3}\cos \left(\frac{3}{2}\pi\nu_{\pi,j}\right)\,,
\end{equation}where $\phi$ is the arithmetically normalized Maass form generating $\mu$. That is, 
\begin{equation}
\phi(z)=\sum_{\gamma\in U\setminus SL_2(\mathbb{Z})}\sum_{m_1=1}^{\infty}\sum_{m_2\neq 0}\frac{A_{\pi}(m_1,m_2)}{|m_1m_2|}\mathcal{W}_{\nu}^{\hbox{sgn}(m_2)}\left(\left(\begin{smallmatrix}
|m_1m_2| & &  \\
& m_1 & \\
& & 1
\end{smallmatrix}\right)
\begin{smallmatrix}
\gamma & \\
& 1
\end{smallmatrix} z
\right),
\end{equation} where $ U=\left\{\left(\begin{smallmatrix} 1 & * \\ & 1 \end{smallmatrix}\right)\in SL_2(\mathbb{Z})\right\} $ and $ \mathcal{W}_{\nu}^{\pm}(z)=e(x_1\pm x_2)W_{\nu }^{*}(y_1,y_2) $, where $W_{\nu}^{*}$ is the standard completed Whittaker function as in [\cite{goldfeld}, Def. 5.9.2] , and $A_{\pi }(1,1)=1$.  By Rankin-Selberg theory in combination with Stade’s formula (see e.g. \cite[Section 4]{Bl2}) and \cite[Theorem 2]{Li1}), it follows that
\begin{equation}
N(\pi)\asymp \hbox{res}_{s=1}L(s,\pi\times \tilde{\pi})\ll_{\epsilon}\|\mu_{\pi}\|^{\epsilon}\,.
\end{equation}
\\
The spectral measure is defined by 
\begin{equation}
\hbox{spec}(\mu)d\mu, \,\,\,\,\,\,\,\,\hbox{spec}(\mu)d\mu :=\prod_{i=1}^{3}\left(3\nu_j\tan \left(\frac{3\pi}{2}\nu_j\right)\right)\,,
\end{equation}where $d\mu=d\mu_1d\mu_2=d\mu_2d\mu_3=d\mu_1d\mu_3.$

With the above notations, we now state the Kuznetsov formula in the version of (\cite{butt}, Theorem 2,3 \& 4). 
\subsection{The Kuznetsov formula.} Let $n_1,n_2,m_1,m_2\in\mathbb{N}$ and let $h$ be a function that is holomorphic on $\Lambda_{1/2+\delta}$ for some $\delta>0$, symmetric under the Weyl group, rapidly decreasing as $|\Im{\mu_j}|\to\infty$ and satisfies
\begin{equation}
h(3\nu_j\pm 1)=0,\,\,\,\,j=1,2,3.
\end{equation}Then we have
\begin{equation}\label{4.14}
\mathcal{C}+\mathcal{E}_{\min}+\mathcal{E}_{\max}=\textstyle\Delta+\sum_4+\sum_5+\sum_6\,,
\end{equation}
where
\begin{equation}\label{emin}
\begin{aligned}
&\mathcal{C}=\sum_{j}\frac{h(\pi_j)}{N_j}A_{j}(m_1,m_2)\overline{A_{j}(n_1,n_2)}\,,\\
& \mathcal{E}_{\min}=\frac{1}{24(2\pi i)^2}\int\int_{\Re(\mu)=0}\frac{h(\mu)}{N_{\mu}}A_{\mu}(m_1,m_2)\overline{A_{\mu}(n_1,n_2)}d\mu_1 d\mu_2,\\
& \mathcal{E}_{\max}=\frac{1}{2\pi i}\sum_{g}\int_{\Re(\mu)=0}\frac{h(\mu+\mu_g,\mu-\mu_g,-2\mu)}{N_{\mu,g}}B_{\mu,g}(m_1,m_2)\overline{B_{\mu,g}(n_1,n_2)}d\mu,
\end{aligned}
\end{equation}

and 
\[
\begin{aligned}
&\Delta=\delta_{\substack{m_1=n_1\\m_2=n_2}}\int_{\Re(\mu)=0}h(\mu)\hbox{spec}(\mu)d\mu\,\,,\\
&\hbox{$\sum_4$}=\sum_{\epsilon=\pm 1}\sum_{\substack{D_2|D_1\\n_2D_1=m_1D_2^2}}\frac{\tilde{S}\left(-\epsilon m_2,n_2,n_1:D_1,D_2\right)}{D_1D_2}\Phi_{w_4}\left(\frac{\epsilon n_1n_2m_2}{D_1D_2}\right)\,\,,\\
&\hbox{$\sum_5$}=\sum_{\epsilon=\pm 1}\sum_{\substack{D_1|D_2\\n_1D_2=m_2D_1^2}}\frac{\tilde{S}\left(\epsilon m_1,n_1,n_2:D_1,D_2\right)}{D_1D_2}\Phi_{w_5}\left(\frac{\epsilon m_1n_1n_2}{D_1D_2}\right)\,\,,\\
&\hbox{$\sum_6$}=\sum_{\epsilon _1,\epsilon _2=\pm 1}\sum_{D_1,D_2}\frac{S\left(\epsilon _2m_2,\epsilon _1m_1,n_1,n_2;D_1,D_2\right)}{D_1D_2}\Phi_{w_6}\left(-\frac{\epsilon _2n_1m_2D_2}{D_1^2},-\frac{\epsilon _1n_2m_1D_1}{D_2^2}\right)\,,
\end{aligned}
\]and
\begin{equation}\label{defphi}
\begin{aligned}
&\Phi_{w_4}(y)=\int_{\Re\mu=0}h(\mu)K_{w_4}(y;\mu )\hbox{spec}(\mu )d\mu,\\
&\Phi_{w_5}(y)=\int_{\Re\mu=0}h(\mu)K_{w_4}(-y;-\mu )\hbox{spec}(\mu )d\mu,\\
&\Phi_{w_6}(y_1,y_2)=\int_{\Re\mu =0}h(\mu )K_{w_6}^{\text{sgn}(y_1),\text{sgn}(y_2)}((y_1,y_2);\mu)\text{spec}(\mu)d\mu\,.
\end{aligned}
\end{equation}Here $A_{\mu}(m_1,m_2)$ denotes the Fourier coefficients of the minimal Eisenstein series $E(z,\mu)$ and $B_{\mu,g}(m_1,m_2)$ denotes the Fourier coefficient of $E_{P_{2,1}}(z,\mu,g)$, the maximal Eisenstien series twisted by Maass form $g$ (see Goldfeld's text \cite{goldfeld}, eqn. (10.4.1), (10.11.1)). We ignore the positive contribution of $\mathcal{E}_{\min}$ and $\mathcal{E}_{\max}$.


\subsection{Integral representations}\label{intrep}We quote the following alternate expressions for the kernel functions given in definition \eqref{58} in terms of the Bessel functions. See section 5 of \cite{Bl1} for details. The stationary phase analysis seems to work better for these representations (at least in our set-up), and we will mainly work with these.

For $y_1,y_2\in \mathbb{R}\setminus \{0\}$ and $\mu\in \Lambda_0$, define
\begin{equation}\label{521}
\begin{aligned}
&\mathcal{J}_1^{\pm}(y,\mu) =\left|\frac{y_1}{y_2}\right|^{\mu_2/2}\int_{0}^{\infty} J^{\pm}_{3\nu_3}\left(|y_1|^{1/2}\sqrt{1+u^2}\right)J^{\pm}_{3\nu_3}\left(|y_2|^{1/2}\sqrt{1+u^{-2}}\right)u^{3\mu_2}\,\frac{du}{u},\\
&\mathcal{J}_2(y,\mu) =\left|\frac{y_1}{y_2}\right|^{\mu_2/2}\int_{1}^{\infty} J^{-}_{3\nu_3}\left(|y_1|^{1/2}\sqrt{u^2-1}\right)J^{-}_{3\nu_3}\left(|y_2|^{1/2}\sqrt{1-u^{-2}}\right)u^{3\mu_2}\,\frac{du}{u},\\
&\mathcal{J}_3(y,\mu) =\left|\frac{y_1}{y_2}\right|^{\mu_2/2}\int_0^{\infty} \tilde{K}_{3\nu_3}\left(|y_1|^{1/2}\sqrt{1+u^2}\right)J^{-}_{3\nu_3}\left(|y_2|^{1/2}\sqrt{1+u^{-2}}\right)u^{3\mu_2}\,\frac{du}{u},\\
&\mathcal{J}_4(y,\mu) =\left|\frac{y_1}{y_2}\right|^{\mu_2/2}\int_{0}^{1} \tilde{K}_{3\nu_3}\left(|y_1|^{1/2}\sqrt{1-u^2}\right)\tilde{K}_{3\nu_3}\left(|y_2|^{1/2}\sqrt{u^{-2}-1}\right)u^{3\mu_2}\,\frac{du}{u},\\
&\mathcal{J}_5(y,\mu) =\left|\frac{y_1}{y_2}\right|^{\mu_2/2}\int_0^{\infty} \tilde{K}_{3\nu_3}\left(|y_1|^{1/2}\sqrt{1+u^2}\right)\tilde{K}_{3\nu_3}\left(|y_2|^{1/2}\sqrt{1+u^{-2}}\right)u^{3\mu_2}\,\frac{du}{u},
\end{aligned}
\end{equation}
 For $y_1,y_2>0$ we have
\begin{equation}\label{522}
K_{w_6}^{++}(y;\mu)=\frac{1}{12\pi^2}\frac{\cos \left(\frac{3}{2}\pi\nu_1\right)\cos \left(\frac{3}{2}\pi\nu_2\right)}{\cos \left(\frac{3}{2}\pi\nu_3\right)}\mathcal{J}_5(y,\mu);
\end{equation}for $y_1>0>y_2$ we have
\begin{equation}\label{523}
\sum_{w\in\{I,w_4,w_5\}}K_{w_6}^{+-}(y;w(\mu))=\frac{1}{24\pi^2}\sum_{w\in\{I,w_4,w_5\}}\left(\mathcal{J}_2(y;w(\mu))+\mathcal{J}_3(y;w(\mu))+\mathcal{J}_4(y;w(\mu))\right);
\end{equation}for $y_2>0>y_1$ we have
\begin{equation}\label{524}
K_{w_6}^{-+}((y_1,y_2);\mu)=K_{w_6}^{+-}((y_2,y_1);w_4(\mu));
\end{equation}and for $y_1,y_2<0$ we have
\begin{equation}\label{525}
\sum_{w\in\{I,w_4,w_5\}}K_{w_6}^{--}(y;w(\mu))=\frac{1}{48\pi^2}\sum_{w\in\{I,w_4,w_5\}}\left(4\mathcal{J}_1^{-}(y;w(\mu))+4\mathcal{J}_1^{+}(y;w(\mu)))\right).
\end{equation}

\section{Analytic properties of the integral transforms}
The first two lemmas, which we directly quote from \cite[Section 7]{Bl1}, will be used to truncare various dual sums.
\begin{lemma}\label{5.1}
Let $0< y\leq T^{3-\epsilon}$. Then for any constant $B\geq 0$ one has
\begin{equation}\label{phi5negsmall}
\Phi_{w_4}(y)\ll_{\epsilon,B} T^{-B}.
\end{equation}
If $T^{3-\epsilon}<|y|$, then 
\begin{equation}
|y|^j\Phi_{w_4}^{(j)}(y)\ll _{j,\epsilon}T^{3+2\epsilon}(T+|y|^{1/3})^j,
\end{equation}for any $j\in \mathbb{N}_0$.
\end{lemma}	
\begin{lemma}\label{5.2}
Let $\Upsilon :=\min (|y_1|^{1/3}|y_2|^{1/2},|y_1|^{1/6}|y_2|^{1/3})$. If $\Upsilon\leq T^{1-\epsilon}$, then
\begin{equation}\label{negsmalllong}
\Phi_{w_6}(y_1,y_2)\ll_{B,\epsilon} T^{-B},
\end{equation}for any fixed constant $B\geq 0$. If $\Upsilon\geq T^{1-\epsilon}$, then
\begin{equation}\label{longweightder}
\begin{aligned}
|y_1|^{j_1}|y_2|^{j_2}&\frac{\partial^{j_1}}{\partial y_1^{j_1}}\frac{\partial^{j_2}}{\partial y_2^{j_2}}\Phi_{w_6}(y_1,y_2)\\
&\ll_{j_1,j_2,\epsilon} T^{3+\epsilon}\left(T+|y_1|^{1/2}+|y_1|^{1/3}|y_2|^{1/6}\right)^{j_1}\left(T+|y_2|^{1/2}+|y_2|^{1/3}|y_1|^{1/6}\right)^{j_2}.
\end{aligned}
\end{equation}
\end{lemma}for all $j_1,j_2\in \mathbb{N}_0$.\\
\\
The main effort of the paper goes into obtaining non-trivial estimates for integral transforms of the form
\begin{equation}\notag
\begin{aligned}
\mathcal{K}(\theta_1,\theta_2, U,V):=\sum_{\epsilon_1,\epsilon_2=\pm 1}\int_{\mathbb{R}}\int_{\mathbb{R}} \mathcal{U}(x)\mathcal{V}(y)\Phi_{w_6}(\epsilon_1\theta_1x^2,\epsilon_2\theta_2y^2)e(Ux+Vy)\,dx\,dy.
\end{aligned}
\end{equation}By using the integral representations in term of the Bessel functions from \eqref{intrep}, we reduce our problem to estimating each
\begin{equation}
\begin{aligned}
\mathcal{K}_i(\theta_1, \theta_2, U, V)=\int h(\mu)\text{spec}(\mu)\int_{\mathbb{R}}\int_{\mathbb{R}} \mathcal{U}(x)\mathcal{V}(y)\mathcal{J}_{i}(\epsilon_1\theta_1x^2,\epsilon_2\theta_2y^2;\mu)e(Ux+Vy)\,\,dx\,dy\,d\mu,
\end{aligned}
\end{equation}where $\mathcal{J}_i$ as defined in \eqref{521}. In this paper, we will restrict our attention to $\mathcal{K}_4(\theta_1,\theta_2, U,V)$ and provide a detailed analysis for the same. The arguments are robust and can be easily adapted to the remaining ones which provide us with the same (or smaller) contributions.
\begin{theorem}\label{average}
Let $\mu=(\mu_1,\mu_2,\mu_3)$ be in generic position with $||\mu||\asymp T$. For $U,V\in \mathbb{R},\,\theta_1,\theta_2 >0$, denote
\begin{equation}\notag
c_1:= \frac{\pi |U|}{\theta_1^{1/2}},\,\,c_2:=\frac{\pi |V|}{\theta_2^{1/2}},\, A:=|c_1-1|.
\end{equation}
Suppose 
\begin{equation}\label{uvup}
\min\{\theta_1^{1/3}\theta_2^{1/6},\theta_1^{1/6}\theta_2^{1/3}\}\gg T^{1-\epsilon},\,\, U\ll \theta_1^{1/2}+\theta_1^{1/3}\theta_2^{1/6},\,\,\,V\ll \theta_2^{1/2}+\theta_2^{1/3}\theta_1^{1/6}.
\end{equation}Then $\mathcal{K}_4(\theta_1, \theta_2, U, V)$ is negligibly small unless 
\begin{equation}\label{t1>t2}
\theta_1\gg \theta_2,
\end{equation} in which case,
\begin{equation}\label{mainest}
\begin{aligned}
(\theta_1/\theta_2)^{1/6}\mathcal{K}_4(\theta_1,\theta_2, U,V)\ll T^{1-1/140}+T\delta_{\min\{c_1,c_2\}\ll T^{-1/140}}+T\delta_{A\ll T^{-1/140}}
\end{aligned}
\end{equation}Furthermore, if one of $U$ or $V$ is zero, we have
\begin{equation}\label{v=0}
\begin{aligned}
(\theta_1/\theta_2)^{1/6}\mathcal{K}_4(\theta_1,\theta_2, U,0)\ll T^{1-1/140}+T\delta_{c_1\ll T^{-1/140}}+T\delta_{A\ll T^{-1/140}},
\end{aligned}
\end{equation}and
\begin{equation}\notag
(\theta_1/\theta_2)^{1/6}\mathcal{K}_4(\theta_1,\theta_2, 0, V)\ll T^{1-1/140}+T\delta_{c_2\ll T^{-1/140}}.
\end{equation}
\end{theorem}
\begin{proof}
See section \ref{pfaverage}.
\end{proof}
 If both $U=V=0$, we have the following estimate for our original integral transform $ \mathcal{K}(\theta_1,\theta_2, U,V)$.
\begin{theorem}\label{dia}
We have
\begin{equation}\label{68}
\sum_{\substack{f\geq 1\\(f,F)=1}}\frac{1}{f}\mathcal{K}\left(\frac{\theta_1}{f},\frac{\theta_2}{f},0,0\right)\ll_{\epsilon}T^{3+4\epsilon}(\theta_1\theta_2)^{-1/2}.
\end{equation}for any positive integer $F$.
\end{theorem}
•

\begin{proof}
See section \ref{pfdia}
\end{proof}
\section{Applying the Kuznetsov formula}
Before applying the Kuznetsov formula \eqref{4.14} to the spectral sum in \eqref{2.10}, we need to specify the test function used to detect the spectral average.

\subsection*{The test function $h(\mu)$}
We use the same test function specified in \cite{Bl1}. In particular, this test function satisfies the required properties for the Kuznetsov formula, is non-negative on $\Lambda'_{1/2}$, satisfies $h(\mu_0)\gg 1$ and is negligibly small outside $O(T^{\epsilon})$-balls about $w(\mu_0)$ for $w\in \mathcal{W}$, where $\mathcal{W}$ is the Weyl group which acts on $\mu$ by permutations. Firstly let us fix
\begin{equation}
\psi(\mu)=\exp(\mu_1^2+\mu_2^2+\mu_3^2)\,,
\end{equation} and let
\begin{equation}
P(\mu):=\prod_{0\leq n\leq A}\prod_{j=1}^{3}\frac{(\nu_j-\frac{1}{3}(1+2n))(\nu_j+\frac{1}{3}(1+2n))}{|\nu_{0,j}|^2}\,,
\end{equation}for some large fixed constant $A$. This polynomial has zeroes at the poles of the spectral measure, which turns out to be convenient for contour shifts. Now we choose
\begin{equation}
h(\mu):=P(\mu)^2\left(\sum_{w\in\mathcal{W}}\psi\left(\frac{w(\mu)-\mu_0}{T^{\epsilon}}\right)\right)^2\,.
\end{equation}The $T^{\epsilon}$-radius provides us some extra space that will be convenient is later estimations. In particular, we have
\begin{equation}\label{519}
\mathcal{D}_jh(\mu)\ll_j T^{-j\epsilon}\,,
\end{equation} for any differential operator $\mathcal{D}_j$ of order $j$, which we will frequently use while integrating by parts. Moreover, we have
\begin{equation}
\int_{\Re\mu=0} h(\mu)\hbox{spec}(\mu)d\mu\ll T^{3+\epsilon}.
\end{equation}

Applying the $GL(3)$ Kuznetsov formula \eqref{4.14} with test function above, the $\mu_j$ sum in \eqref{2.10} becomes
\begin{equation}\label{kuznetsovdual}
  \sum_{j}\frac{h(\pi_j)}{N_j}A_{\pi_j}(rl_1^{k_2-k_1},ml_1^{k_0})\overline{A_{\pi_j}(rl_2^{k_2-k_1},nl_2^{k_0})} =\textstyle\Delta+\sum_4+\sum_5+\sum_6-\mathcal{E}_{\min}-\mathcal{E}_{\max}\,
\end{equation}where, by abuse of notations, the terms in the right hand side are as defined in \eqref{4.14} with the new variables variables $(m_1,m_2)=(rl_1^{k_2-k_1}, ml_1^{k_0})$, $(n_1,n_2)=(rl_2^{k_2-k_1}, nl_2^{k_0})$. Substituting \eqref{kuznetsovdual} into \eqref{2.10}, we obtain
\begin{equation}\label{s_i}
    \sum_{j}\frac{h(\pi_j)}{N_j}|\tilde{S}_{\bar{k},r,j}(N)|^2=S_0+S_4+S_5+S_6-S_{max}-S_{min},
\end{equation}where $S_0$ denotes the contribution corresponding to $\Delta$, $S_4$ corresponds to $\sum_{4}$, and so on. It remains to estimate each $S_i$. 

\section{The diagonal contribution $S_0$}
We have \[S_0=\sum_{l\asymp L}\sum_{m\asymp N}|\lambda(ml^{k_2})|^2\int_{\Re\mu=0}h(\mu)\text{spec}(\mu) d\mu\]
Substituting the poinwise bound $\lambda(n)\ll n^{\epsilon}$ and using the bound $\text{spec}(\mu)\asymp T^3$ for the spectral measure, we obtain
\begin{equation}\notag
    S_0\ll NT^3L.
\end{equation}
\section{The contribution of $S_6$}
We have
\begin{equation}
\begin{aligned}
S_6=\sum_{\epsilon_1,\epsilon_2=\pm 1}\sum_{l_1,l_2\asymp L}\sum_{D_1,D_2}\frac{1}{D_1D_2}\sum_{ m,n\geq 1}\lambda(ml_1^{k_2})\overline{\lambda(nl_2^{k_2})}S\left(\epsilon_2 ml_1^{k_0},\epsilon_1 rl_1^{k_2-k_1},rl_2^{k_2-k_1},nl_2^{k_0};D_1,D_2\right)\\
 V\left(\frac{ml_1^{k_2}}{N}\right)V\left(\frac{nl_2^{k_2}}{N}\right) \Phi_{w_6}\left(-\frac{\epsilon_2 mrl_1^{k_0}l_2^{k_2-k_1}D_2}{D_1^2},\,-\frac{\epsilon_1 nrl_1^{k_2-k_1}l_2^{k_0}D_1}{D_2^2}\right)
\end{aligned}
\end{equation}We apply Lemma \ref{5.2} with $y_1=ml_1^{k_0}rl_2^{k_2-k_1}D_2/D_1^2$ and $y_2=nl_2^{k_0}rl_1^{k_2-k_1}D_1/D_2^2$ to see that $\Phi_{w_6}(y_1,y_2)$ is negligibly small unless 
\begin{equation}
\frac{1}{D_1^{\frac{1}{2}}}\gg \frac{T}{N^{\frac{1}{2}}r^{\frac{1}{2}}L^{\frac{k_0+k_2-k_1}{2}}}\,\,\,\,\hbox{and}\,\,\,\, \frac{1}{D_2^{\frac{1}{2}}}\gg \frac{T}{N^{\frac{1}{2}}r^{\frac{1}{2}}L^{\frac{k_0+k_2-k_1}{2}}}
\end{equation}i.e.
\begin{equation}\label{modulus}
D_1,D_2\ll \frac{NL^{k_0+k_2-k_1}r}{T^2}=:D_0.
\end{equation}We now substitute the decomposition \eqref{kiral} of the long Weyl Kloosterman sum. After substituting and writing $D_1=fd_1, D_2=fd_2$ and dividing the $d_1,d_2$ sum into dyadic blocks $d_1\asymp H_1\ll D_0/f$ and $ d_2\asymp H_2\ll D_0/f$ , we see that each dyadic box contributes at most
\begin{equation}\label{6.3}
\begin{aligned}
\frac{1}{H_1H_2}\sum_{f}\frac{1}{f}\sum_{\substack{d_1\asymp H_1\\d_2\asymp H_2}}&\sum_{\epsilon_1,\epsilon_2=\pm 1}\sum_{l_1,l_2\asymp L}\sum_{\substack{y(\bmod f)\\f|(rl_2^{k_2-k_1}d_2+\epsilon_1 rl_1^{k_2-k_1}d_1y)}}\\
&\times\sum_{\substack{m,n\asymp N/L^{k_2}}}\lambda(ml_1^{k_2})\overline{\lambda(nl_2^{k_2})}S(\epsilon_2 ml_1^{k_0},M_f(y);d_1)S(\epsilon_1 nl_2^{k_0},\tilde{M}_f(y);d_2)\\
&\times V\left(\frac{ml_1^{k_2}}{N}\right)V\left(\frac{nl_2^{k_2}}{N}\right)\Phi_{w_6}\left(-\frac{\epsilon_2 mrl_1^{k_0}l_2^{k_2-k_1}d_2}{d_1^2f},\,-\frac{\epsilon_1 nrl_1^{k_2-k_1}l_2^{k_0}d_1}{d_2^2f}\right).
\end{aligned}
\end{equation}
To apply the $GL(2)$ Voronoi summation to the $m$ and $n$ sums, we  pull out the $l_1,l_2$ variables using the Hecke relation
\begin{equation}\label{hecke}
\lambda(ml_1^{k_2})=\lambda(m)\lambda(l_1^{k_2})-\left(\sum_{d=1}^{k_2}\lambda\left(\frac{ml_1^{k_2}}{l_1^{2d}}\right)\right)
\end{equation}where the $r$-th term in the bigger parenthesis occurs only if $l_1^{r}|m$. For simplicity, we only estimate the contribution of  the first term. We remark that using \eqref{hecke} recursively the other terms can be handled similarly and provide us with smaller contributions. Hence the $m,n$ sums in \eqref{6.3} is essentially
\begin{equation}\label{m,nsum}
\begin{aligned}
\lambda(l_1^{k_2})\overline{\lambda(l_2^{k_2})}&\sum_{\substack{m,n\asymp N/L^{k_2}}}\lambda(m)\overline{\lambda(n)}S(\epsilon_2 ml_1^{k_0},M_f(y);d_1)S(\epsilon_1 nl_2^{k_0},\tilde{M}_f(y);d_2)\\
&\times V\left(\frac{ml_1^{k_2}}{N}\right)V\left(\frac{nl_2^{k_2}}{N}\right)\Phi_{w_6}\left(-\frac{\epsilon_2 mrl_1^{k_0}l_2^{k_2-k_1}d_2}{d_1^2f},\,-\frac{\epsilon_1 nrl_1^{k_2-k_1}l_2^{k_0}d_1}{d_2^2f}\right).
\end{aligned}
\end{equation}We now apply $GL(2)$ Voronoi to the $m$ and $n$ sum. We carry out the details for $f$ cuspidal and $(d_1, l_1)=(d_2, l_2)=1$ in this section. The case with $l_1|d_1$ or $l_2|d_2$ has smaller contribution due to lowering of the conductor and can be carried out in a similar fashion. The case $\lambda(n)=d(n)$ contains an additional main term coming from the summation formula for $d(n)$ and is addressed in the next section. Note that the Kloosterman sums get transformed into Ramanujan sums, saving the whole modulus $d_1$ and $d_2$. This is a crucial point in the proof. The dual sum is essentially of the form
\begin{equation}\label{6.6}
\frac{N^2}{L^{2k_2}d_2d_1}\sum_{\substack{c_2|d_2\\c_1|d_1}}c_1c_2\sum_{\substack{\tilde{m}\\c_2|(\tilde{m}+l_1^{k_0}M_f(y))}}\,\,\sum_{\substack{\tilde{n}\\c_1|(\tilde{n}+l_2^{k_0}\tilde{M}_f(y))}}\lambda(\tilde{m})\overline{\lambda(\tilde{n})}I(\tilde{m},\tilde{n})\,,
\end{equation}where 
\begin{equation}
I(\tilde{m},\tilde{n})=\sum_{\epsilon_1,\epsilon_2=\pm 1}\int\int V(x)V(y)\Phi_{w_6}(-\epsilon_2\theta_1x,-\epsilon_1\theta_2y)J_{k_f-1}\left(U\sqrt{x}\right)J_{k_f-1}\left(V\sqrt{y}\right)dx\,dy,
\end{equation}where
\begin{equation}\label{7.9}
\begin{aligned}
\theta_1=\frac{Nrl_1^{k_0-k_2}l_2^{k_2-k_1}d_2}{d_1^2f} ,\quad\quad\quad\quad U=\frac{4\pi\sqrt{N\tilde{m}}}{l_1^{k_2/2}d_1},  \\
\theta_2=\frac{Nrl_1^{k_2-k_1}l_2^{k_0-k_2}d_1}{d_2^2f},\quad\quad\quad\quad V=\frac{4\pi\sqrt{N\tilde{n}}}{l_2^{k_2/2}d_2}.
\end{aligned}
\end{equation} 
Extracting the oscillations from the Bessel functions using \eqref{bes} we see that $I(\tilde{m},\tilde{n})$ is essentially a sum of four terms of the form
\begin{equation}\label{extracttion}
|UV|^{-1/2}\sum_{\epsilon_1,\epsilon_2=\pm 1}\int\int W(x)W(y)\Phi_{w_6}(-\epsilon_2\theta_1x,-\epsilon_1\theta_2y)e(\pm U\sqrt{x}\pm V\sqrt{y})dx\,dy,
\end{equation}where $W(x)$ is the updated smooth weight function which by \eqref{besselweight} satisfies $W^{(j)}(x)\ll_{j} 1$. By  \eqref{longweightder} and repeated integration by parts in the $x,y$ integral, it follows that $I(\tilde{m},\tilde{n})$ is negligibly small unless
\begin{equation}\notag
\begin{aligned}
U\ll T+\theta_1^{1/2}+\theta_1^{1/3}\theta_2^{1/6},\,\,V\ll T+\theta_2^{1/2}+\theta_1^{1/6}\theta_2^{1/3},
\end{aligned}
\end{equation}which by \eqref{negsmalllong} implies
\begin{equation}\notag
U,V\ll \theta_1^{1/2}+\theta_2^{1/2}.
\end{equation}Substituting the values \eqref{7.9}, we obtain
\begin{equation}\label{7.8}
\tilde{m},\tilde{n}\ll \frac{L^{k_0+k_2-k_1}r(d_1+d_2)}{f}:=M_0.
\end{equation}Substituting the definition \eqref{defphi} into \eqref{extracttion} we obtain
\begin{equation}\label{7.10}
I(\tilde{m},\tilde{n})\ll |UV|^{-1/2}|\mathcal{K}\left(\theta_1,\theta_2,\pm U,\pm V\right)|
\end{equation}where
\begin{equation}\label{inte}
\begin{aligned}
\mathcal{K}(\theta_1,\theta_2, U,V)&=\sum_{\epsilon_1,\epsilon_2\in \pm 1}\int h(\mu)\text{spec}(\mu)\\
&\,\,\,\,\,\,\,\,\int_{\mathbb{R}}\int_{\mathbb{R}} 2xW(x^2)2yW(y^2)K_{w_6}^{\epsilon_1,\epsilon_2}(\epsilon_1\theta_1x^2,\epsilon_2\theta_2y^2;\mu)e( Ux+ Vy)\,\,\,dx\,dy\,d\mu.
\end{aligned}
\end{equation}Now by the integral representations given in section \ref{intrep}, we get
\begin{equation}\notag
\mathcal{K}(\theta_1,\theta_2, U,V)\ll\sum_{i=1}^{5}|\mathcal{K}_i(\theta_1, \theta_2, U, V)|
\end{equation}where
\begin{equation}\label{ji}
\mathcal{K}_i(\theta_1, \theta_2, U, V)=\int h(\mu)\text{spec}(\mu)\int_{\mathbb{R}}\int_{\mathbb{R}} \mathcal{U}(x)\mathcal{U}(y)\mathcal{J}_{i}(\epsilon_1\theta_1x^2,\epsilon_2\theta_2y^2;\mu)e(Ux+Vy)\,\,dx\,dy\,d\mu,
\end{equation}where $\mathcal{U}(x)=2xW(x^2)$ and $\mathcal{J}_i$ as defined in \eqref{521}. It remains to estimate the contribution of each $\mathcal{K}_i, i=1,\cdots, 5$. As remarked earlier, we only estimate the contribution of $\mathcal{K}_4$, which is given by \eqref{mainest} of Theorem \ref{average}. The estimation of the rest is very similar and is just a matter of formality. In \eqref{mainest}, it is enough to estimate the contribution of the last term on the right-hand side since the estimation for the rest are very similar and provide us with a smaller contribution. Substituting the last term of \eqref{mainest} and the expression for $U,V$ from \eqref{7.9} into \eqref{7.10}, and trivially executing the $\mu$ integral we obtain
\begin{equation}
I(\tilde{m},\tilde{n})\ll \frac{TL^{k_2/2}}{N^{1/2}}\cdot\frac{d_1}{{\tilde{m}}^{1/4}{\tilde{n}}^{1/4}},
\end{equation}with the constraint 
\begin{equation}\label{mres}
\left|\frac{2\pi U}{\theta_1^{1/2}}-1\right|=\left|\frac{\sqrt{\tilde{m}}}{bd_2}-1\right|\ll T^{-1/140},
\end{equation}where $b:=(8\pi^2)^{-1}(rl_1^{k_0}l_2^{k_2-k_1}/f)$.
Hence taking absolute values, we see that \eqref{6.6} is dominated by
\begin{equation}\label{6.15}
\frac{TN^{3/2}L^{k_2/2}}{L^{2k_2}d_2}\sum_{\substack{c_2|d_2\\c_1|d_1}}c_1c_2\sideset{}{'}\sum_{\substack{\tilde{m}\ll M_0\\c_2|(\tilde{m}+l_1^{k_0}M_f(y))}}\,\,\sum_{\substack{\tilde{n}\ll M_0\\c_1|(\tilde{n}+l_2^{k_0}\tilde{M}_f(y))}}\frac{1}{{\tilde{m}}^{1/4}{\tilde{n}}^{1/4}},
\end{equation}where the prime over the $\tilde{m}$ sum denotes the restriction \eqref{mres}. Note that have replaced the absolute values of Fourier coefficients by their pointwise bound $T^{\epsilon}$. Substituting \eqref{6.15} in \eqref{6.3}, and further dividing  $\tilde{m}\asymp Z_1\ll M_0$ and $\tilde{n}\asymp Z_2\ll M_0$ into dyadic blocks, we arrive at
\begin{equation}\label{6.16}
\begin{aligned}
&\frac{TN^{3/2}L^{k_2/2}}{L^{2k_2}}\sum_{f}\frac{1}{f}\sum_{l_1,l_2\asymp L}\\
&\times\frac{1}{H_1H_2^{2}Z_1^{1/4}Z_2^{1/4}}}\sum_{\substack{d_1\asymp H_1\\d_2\asymp H_2}}\,\,\sum_{\substack{y(\bmod f)\\f|(rl_2^{k_2-k_1}d_2+rl_1^{k_2-k_1}d_1y)}}\,\,\,\sum_{\substack{c_1|d_1\\c_2|d_2}}c_1c_2\,\,\,\sideset{}{'}\sum_{\substack{\tilde{m}\asymp Z_1\\c_1|(\tilde{m}+l_1^{k_0}M_f(y))}}\,\,\,\sum_{\substack{\tilde{n}\asymp Z_2\\c_2|(\tilde{n}+l_2^{k_0}\tilde{M}_f(y))}}1{.
\end{aligned}
\end{equation}
\\
We now count the number of points given by the last five set of summations. The last two condition on $\tilde{m}$ and $\tilde{n}$ implies
\begin{equation}\label{setofcong}
\tilde{m}f=c_1k_1+rl_2^{k_2-k_1}d_2\,\,\,\,\text{and}\,\,\,\,\,\,\tilde{n}f=c_2k_2+rl_1^{k_2-k_1}d_1.
\end{equation}We first fix $\tilde{m}, c_2, d_2$ . Note that this determines $c_1$ (upto a divisor function). For this $c_1$, the number of $d_1$ is $\ll H_1/c_1$, and for each of these $d_1$, the no. of $k_2$ is $\ll 1+fZ_2/c_2$. Also, the number of $y (\bmod f)$ is $\ll (rl_2^{k_2-k_1}d_2,f)\leq (rl_2^{k_2-k_1},f)(d_2,f)$. Hence \eqref{6.16} is dominated by
\begin{equation}\label{6.18}
\begin{aligned}
\frac{TN^{3/2}L^{k_2/2}}{L^{2k_2}}&\sum_{l_1,l_2\asymp L}\sum_{f}\frac{(rl_2^{k_2-k_1},f)}{f}\\
&\times\frac{1}{H_2^{2}Z_1^{1/4}Z_2^{1/4}}\sum_{\substack{\\d_2\asymp H_2}}(d_2,f)\,\,\sum_{\substack{\\c_2|d_2}}c_2\,\,\,\left(1+\frac{fZ_2}{c_2}\right)\sideset{}{'}\sum_{\substack{\tilde{m}\asymp Z_1\\}} 1.
\end{aligned}
\end{equation}Executing the $\tilde{m}$ sum with the constraint \eqref{mres}, and the remaining sum trivially, we see that the second line in \eqref{6.18} is dominated by
\begin{equation}\label{small}
\begin{aligned}
&\frac{T^{-1/140}Z_1^{3/4}}{H_2^{2}Z_2^{1/4}}\left(H_2^2+fZ_2H_2\right)= \frac{T^{-1/140}Z_1^{1/4}}{Z_2^{1/4}}\left(Z_1^{1/2}+\frac{fZ_1^{1/2}Z_2}{H_2}\right)\\
&\ll\frac{T^{-1/140}Z_1^{1/4}}{Z_2^{1/4}}\cdot\frac{f(L^{k_0+k_2-k_1}r)^{3/2}D_0^{1/2}}{f^2}\ll (Z_1/Z_2)^{1/4}T^{-1/140}\frac{N^{1/2}(L^{k_0+k_2-k_1}r)^{2}}{fT}.
\end{aligned}
\end{equation}The last bound is fine for us if $(Z_1/Z_2)\ll 1$. When $Z_1/Z_2$ is large, we count the set of congruences \eqref{setofcong} by first fixing $\tilde{n}, c_1, d_1$. Following the rest of the arguments from earlier, we will arrive at the bound
\begin{equation}\label{large}
(Z_2/Z_1)^{1/4}\frac{N^{1/2}(L^{k_0+k_2-k_1}r)^{2}}{fT}
\end{equation}for the second line of \eqref{6.18}. Combining \eqref{small} and \eqref{large}, we see that the second line of \eqref{6.18} can be dominated by
\begin{equation}\notag
T^{-1/280}\frac{N^{1/2}(L^{k_0+k_2-k_1}r)^{2}}{fT}.
\end{equation}Substituting the last bound and executing the remaining sum trivially, we see that \eqref{6.18}, and consequently 
\begin{equation}\label{mainod}
S_6\ll \frac{N^2r^2L^{2(1+k_0-k_1)+k_2/2}}{T^{1/280}}\ll \frac{N^2r^2L^8}{T^{1/280}}.
\end{equation}

\section{Computations for $f(z)=E(z,1/2)$}
For the case $\lambda(m)=d(m)$, the $m$-sum (similarly the $n$-sum) in \eqref{m,nsum} becomes
\begin{equation}\label{eisenstienvor}
\begin{aligned}
&\sideset{}{^*}\sum_{x (d_1)}e\left(\frac{xM_f(y)}{d_1}\right)\sum_{m}d(m)e\left(\frac{\epsilon_2\bar{x}ml_1^{k_0}}{d_1}\right)V\left(\frac{ml_1^{k_2}}{N}\right)\\&\hspace{70mm}\Phi_{w_6}\left(-\frac{\epsilon_2 mrl_1^{k_0}l_2^{k_2-k_1}d_2}{d_1^2f},\,-\frac{\epsilon_1 nrl_1^{k_2-k_1}l_2^{k_0}d_1}{d_2^2f}\right)\\
&=\frac{N(d_1,l_1^{k_0})}{l_1^{k_2}d_1}\sideset{}{^*}\sum_{x (d_1)}e\left(\frac{xM_f(y)}{d_1}\right)\Bigg(2\int_{\mathbb{R}_{>0}}\left(\log ((d_1,l_1^{k_0})\sqrt{x}/d_1)+\gamma\right)h(x)dx\\
& +\sum_{\tilde{m}\geq 1}d(\tilde{m})e\left(\frac{-\epsilon_2x\overline{(l_1^{k_0}/(d_1,l_1^{k_0}))}\tilde{m}}{d_1/(d_1,l_1^{k_0})}\right)\int_{0}^{\infty}\left(-2\pi Y_0(U'\sqrt{x})+4K_0 (U'\sqrt{x})\right)h(x)dx\Bigg),
\end{aligned}
 \end{equation} where $h(x)=V(x)\Phi_{w_6}(-\epsilon_2\theta_1 x, -\epsilon_1nrl_1^{k_2-k_1}l_2^{k_0}d_1/d_2^2f)$ and $U'=(d_1,l_1^{k_0})U$, where $U,\theta_1$ as in \eqref{7.9}. 
 The contribution of $Y_0(U'\sqrt{x})$ part will be the same as in \eqref{6.6} since the asymptotics for the $Y_0$ Bessel function and the $J$ Bessel function are the same (for fixed orders). For the $K_0(U'\sqrt{x})$ part, note that from \eqref{5.2} and \eqref{modulus} , $U'\gg U\gg fN^{1/2}/(L_1^{k_2/2}D_0)\gg (T^2/rN^{1/2})L^{-2(k_0+k_2+k_1)}\gg T^{1/2}L^{-6}\gg T^{\epsilon}$, by our choice of $L$. Hence from the exponential decay \eqref{besselweight}, the contribution of the $K_0(U'\sqrt{x})$ part is negligibly small.

Thus for $f=E(z,1/2)$, we just have to take care of the additional contribution  $A_0+A_1+A_2$ towards \eqref{6.3}, where
\begin{equation}
\begin{aligned}
A_0=\sum_{l_1,l_2\asymp L}\sum_{d_1,d_2\geq 1}\frac{N^2(d_1,l_1^{k_0})(d_2, l_2^{k_0})}{l_1^{k_2}l_2^{k_2}d^2_1d^2_2}&\sum_{f\geq 1}\frac{1}{f}\sum_{\substack{y(\bmod f)\\f|\left(rl_2^{k_2-k_1}d_2+rl_1^{k_2-k_1}d_1y\right)}}\\
&\times\sum_{\substack{c_1|(d_1, M_f(y))\\ c_2|(d_2 ,\tilde{M}(y)))}}c_1\mu\left(\frac{d_1}{c_1}\right)c_2\mu\left(\frac{d_2}{c_2}\right)\,I_{0},
\end{aligned}
\end{equation}
\begin{equation}
\begin{aligned}
A_1\ll \sum_{l_1,l_2\asymp L}\sum_{f\geq 1}\frac{1}{f}\sum_{d_1,d_2\ll D_0/f}\frac{N^2}{l_1^{k_2}l_2^{k_2}d^2_1d^2_2}\sum_{\substack{y(\bmod f)\\f|\left(rl_2^{k_2-k_1}d_2+rl_1^{k_2-k_1}d_1y\right)}}\\
\sum_{\substack{c_1|(d_1,M_f(y))\\c_2|d_2}}c_1c_2\sum_{\substack{\tilde{n}\\c_2|(\tilde{n}+l_2^{k_0}\tilde{M}_f(y))}}d(\tilde{n})\,I_1(\tilde{m}),
\end{aligned}
\end{equation}
\begin{equation}
\begin{aligned}
 A_2\ll \sum_{l_1,l_2\asymp L}\sum_{f\geq 1}\frac{1}{f}\sum_{d_1,d_2\ll D_0/f}\frac{N^2}{l_1^{k_2}l_2^{k_2}d^2_1d^2_2}\sum_{\substack{y(\bmod f)\\f|\left(rl_2^{k_2-k_1}d_2+rl_1^{k_2-k_1}d_1y\right)}}\\
\sum_{\substack{c_1|d_1\\c_2|(d_2,\tilde{M}_f(y))}}c_1c_2\sum_{\substack{\tilde{m}\\c_1|(\tilde{m}+l_1^{k_0}M_f(y))}}d(\tilde{m})\,I_2(\tilde{m}),
\end{aligned}
\end{equation} where
\begin{equation}
I_0=\sum_{\epsilon_1,\epsilon_2=\pm 1}\int\int V(x)V(y)\Phi_{w_6}(-\epsilon_2\theta_1x,-\epsilon_1\theta_2y)dx\,dy=\mathcal{K}(\theta_1,\theta_2, 0, 0),
\end{equation}
\begin{equation}
I_1(\tilde{m})=\sum_{\epsilon_1,\epsilon_2=\pm 1}\int\int V(x)V(y)\Phi_{w_6}(-\epsilon_2\theta_1x,-\epsilon_1\theta_2y)Y_0\left(U\sqrt{x}\right)dx\,dy,
\end{equation}and
\begin{equation}
I_2(\tilde{n})=\sum_{\epsilon_1,\epsilon_2=\pm 1}\int\int V(x)V(y)\Phi_{w_6}(-\epsilon_2\theta_1x,-\epsilon_1\theta_2y)Y_0\left(V\sqrt{y}\right)dx\,dy,
\end{equation}where $\mathcal{K}$ as in \ref{inte} and  $U,V,\theta_1,\theta_2$ as in \eqref{7.9}.
By abuse of notation, here $V(x)$ and $V(y)$ denotes the updated weight functions which also includes the $\log$ factors coming from \eqref{eisenstienvor}, wherever occurring. Also, note that the expressions for $S_1$ and $S_2$ are under the assumptions $(d_1,l_1)=(d_2,l_2)=1$ as earlier.

As we will see in a moment, $S_0$ contains a main term which we need to evaluate explicitly and save in the $L$ variable. For this reason all the sums are kept intact for this part.
\subsection{Estimating $A_0$.}
Denote $a_1=rl_1^{k_2-k_1}, a_2=rl_2^{k_2-k_1}$. Recall $M_f(y)=(a_2d_2+a_1d_1y)/f$ so that 
\begin{equation}\notag
c_1|(d_1, M_f(y))\Rightarrow c_1|(d_1, a_2d_2+a_1d_1y)\Rightarrow c_1|(d_1,a_2d_2).
\end{equation}One similarly has $c_2|(a_1d_1,d_2)$. Hence $S_0$ can be written as
\begin{equation}\notag
\begin{aligned}
S_0=\sum_{l_1,l_2\asymp L}\sum_{d_1,d_2}\frac{N^2(d_1,l_1^{k_0})(d_2, l_2^{k_0})}{l_1^{k_2}l_2^{k_2}d^2_1d^2_2}\sum_{\substack{c_1|(d_1,a_2d_2)\\ c_2|(a_1d_1,d_2)}}c_1\mu\left(\frac{d_1}{c_1}\right)c_2\mu\left(\frac{d_2}{c_2}\right)\sum_{f}\frac{1}{f}\cdot I_0\sideset{}{^*}\sum_{\substack{y(\bmod f)\\f|\left(a_2d_2+a_1d_1y\right)\\c_1|M_f(y)\\c_2|\tilde{M}_f(y)}}1,
\end{aligned}
\end{equation}We now count the number of solutions $y\bmod f$ to the last set of congruence conditions. Write  $(a_1d_1, a_2d_2)=\lambda_0, \lambda_1=\lambda_0/c_1,\lambda_2=\lambda_0/c_2$ and further write $(\lambda_1,f)=a, (\lambda_2, f)=b$. Now
\begin{equation}\notag
\begin{aligned}
f|(a_2d_2+a_1d_1y)\,\,\,\,\text{and}\,\,\,\,\, c_1|M_f(y)\Leftrightarrow c_1f|(a_2d_2+a_1d_1y)\\
\Leftrightarrow (a_2d_2/\lambda_0)+(a_1d_2/\lambda_0)y=0 (\bmod f/a),
\end{aligned}
\end{equation}and similarly
\begin{equation}\notag
f|(a_2d_2+a_1d_1y)\,\,\,\,\text{and}\,\,\,\,\, c_2|\tilde{M}_f(y)\Leftrightarrow  (a_2d_2/\lambda_0)+(a_1d_2/\lambda_0)y=0 (\bmod f/b).
\end{equation}Since $[f/a, f/b]=f/(a,b)$, we thus obtain
\begin{equation}\notag
\sideset{}{^*}\sum_{\substack{y(\bmod f)\\f|\left(a_2d_2+a_1d_1y\right)\\c_1|M_f(y)\\c_2|\tilde{M}_f(y)}}1=\sideset{}{^*}\sum_{\substack{y(\bmod f)\\(a_2d_2/\lambda_0)+(a_1d_1/\lambda_0)y=0 (f/(a,b))}}1
\end{equation}Since $(a_2d_2/\lambda_0, a_1d_1/\lambda_0)=1$, the last congruence condition uniquely determines $y$ modulo $f/(a,b)$ which we call $y_{f/(a,b)}$. Hence we get $y=\lambda\cdot f/(a,b)+y_{f/(a,b)}, 1\leq \lambda\leq (a,b)$. Note that $(y_{f/(a,b)}, f/(a,b))=1$ by definition. Hence our counting problem boils down to counting $1\leq \lambda\leq (a,b)$ such that $(\lambda\cdot f/(a,b)+y_{f/(a,b)}, (a,b))=1$ which is easily seen to be
\begin{equation}\notag
\sideset{}{^*}\sum_{\substack{y(\bmod f)\\f|\left(a_2d_2+a_1d_1y\right)\\c_1|M_f(y)\\c_2|\tilde{M}_f(y)}}1=\sideset{}{^*}\sum_{\substack{y(\bmod f)\\(a_2d_2/\lambda_0)+(a_1d_1/\lambda_0)y=0 (f/(a,b))}}1=\left(\frac{f}{(a,b)}, (a,b)\right)\phi\left(\frac{(a,b)}{\left(\frac{f}{(a,b)}, (a,b)\right)}\right).
\end{equation}Substituting we have
\begin{equation}\label{s0reduced}
\begin{aligned}
S_0=\sum_{l_1,l_2\asymp L}\sum_{d_1,d_2}&\frac{N^2(d_1,l_1^{k_0})(d_2, l_2^{k_0})}{l_1^{k_2}l_2^{k_2}d^2_1d^2_2}\sum_{\substack{c_1|(d_1,a_2d_2)\\ c_2|(a_1d_1,d_2)}}c_1\mu\left(\frac{d_1}{c_1}\right)c_2\mu\left(\frac{d_2}{c_2}\right)\\
&\sum_{a|\lambda_1}\sum_{b|\lambda_2}\sum_{\substack{f\geq 1\\ (f,\lambda_1)=a\\(f,\lambda_2)=b}}\frac{1}{f}\cdot\left(\frac{f}{(a,b)}, (a,b)\right)\phi\left(\frac{(a,b)}{\left(\frac{f}{(a,b)}, (a,b)\right)}\right) \cdot I_0.
\end{aligned}
\end{equation}The second line of \eqref{s0reduced} equals
\begin{equation}\label{red2}
\sum_{a|\lambda_1}\sum_{b|\lambda_2}\sum_{c|(a,b)}c\,\phi\left(\frac{(a,b)}{c}\right)\sum_{\substack{f\geq 1\\ (f,\lambda_1)=a\\(f,\lambda_2)=b\\(f/(a,b), (a,b))=c}}\frac{1}{f}\cdot I_0.
\end{equation}For notational ease, we work the details only for $c=1$. The general case is easily worked out in the same way provided extra space and notations. Now $(f,\lambda_1)=a$ implies $f=af', (f',\lambda_1/a)=1$. Substituting this value of $f$ in $(f,\lambda_2)=b$ we obtain $(\lambda_2,a)|b, f'=(b/(\lambda_2,a))f''$ where $(f'', \lambda_2/b)=1$. Combining these two we get $f=abf''/(\lambda_2,a)$ where $(f'', \lambda_1\lambda_2/ab)=1$ and $(b/(\lambda_2,a), \lambda_1/a)=1$. Finally, the last condition $(f/(a,b), (a,b))=1$ translates to $(f'',(a,b))=1$ and $(ab/((\lambda_2,a)(a,b)), (a, b))=1$.  Hence the contribution of $c=1$ to \eqref{red2} becomes
\begin{equation}\label{red3}
\sum_{a|\lambda_1}\sum_{\substack{(\lambda_2,a)|b|\lambda_2\\(ab/((\lambda_2,a)(a,b)), (a, b))=1}}\frac{(\lambda_2,a)\phi((a, b))}{ab}\sum_{\substack{f''\geq 1\\ (f'', \lambda_1\lambda_2(a,b)/ab)=1}}\frac{1}{f''}\mathcal{K}(\theta_1,\theta_2, 0, 0)
\end{equation} where from \eqref{7.9}
\begin{equation}\notag
\theta_1=\frac{Nrl_1^{k_0-k_2}l_2^{k_2-k_1}d_2}{{d_1}^2f}=\left(\frac{Nrl_1^{k_0-k_2}l_2^{k_2-k_1}d_2(\lambda_2,a)}{{d_1}^2ab}\right)\frac{1}{f''},
\end{equation}and
\begin{equation}\notag
\theta_2=\frac{Nrl_1^{k_2-k_1}l_2^{k_0-k_2}d_1}{{d_2}^2f}=\left(\frac{Nrl_1^{k_2-k_1}l_2^{k_0-k_2}d_1(\lambda_2,a)}{{d_2}^2ab}\right)\frac{1}{f''}.
\end{equation}Executing the $f''$-sum using Theorem \ref{dia} we get
\begin{equation}\notag
\sum_{\substack{f''\geq 1\\ (f'', \lambda_1\lambda_2(a,b)/ab)=1}}\frac{1}{f''}\mathcal{K}(\theta_1,\theta_2, 0, 0)\ll T^{3} \left(\frac{d_1^{1/2}d_2^{1/2}ab}{N(rL^{-k_1})L^{k_0}(\lambda_2,a)}\right).
\end{equation}Substituting in \eqref{red3} and executing the remaining sum we see that \eqref{red3} is bounded by
\begin{equation}\notag
\frac{T^3d_1^{1/2}d_2^{1/2}}{N(rL^{-k_1})L^{k_0}}\sum_{a|\lambda_1}\sum_{b|\lambda_2}\phi((a,b))\ll \frac{T^3d_1^{1/2}d_2^{1/2}}{N(rL^{-k_1})L^{k_0}} (\lambda_1,\lambda_2)=\frac{T^3d_1^{1/2}d_2^{1/2}}{N(rL^{-k_1})L^{k_0}}\cdot \frac{\lambda_0 (c_1,c_2)}{c_1,c_2}.
\end{equation}Substituting the above bound with the value $\lambda_0=(rl_1^{k_2-k_1}d_1,rl_2^{k_2-k_1}d_2)$ for the second line of \eqref{s0reduced} we obtain
\begin{equation}\label{A_0}
\begin{aligned}
A_0&\ll \frac{NT^3}{(rL^{-k_1})L^{2k_2+k_0}}\sum_{l_1,l_2\asymp L}\sum_{d_1,d_2}\frac{(d_1,l_1^{k_0})(d_2, l_2^{k_0})(rl_1^{k_2-k_1}d_1, rl_2^{k_2-k_1}d_2)}{d_1^{3/2}d_2^{3/2}}\sum_{\substack{c_1|(d_1,a_2d_2)\\ c_2|(a_1d_1,d_2)}} (c_1,c_2)\\
&\ll  \frac{NT^3}{(rL^{-k_1})L^{2k_2+k_0}}\sum_{l_1,l_2\asymp L}\sum_{d_1,d_2}\frac{(d_1,l_1^{k_0})(d_2, l_2^{k_0})(rl_1^{k_2-k_1}d_1, rl_2^{k_2-k_1}d_2)}{d_1^{3/2}d_2^{3/2}}(d_1,d_2)\\
&\ll \frac{NT^3rL^{-k_1}}{(rL^{-k_1})L^{2k_2+k_0}}\sum_{l_1,l_2\asymp L}\sum_{d_1,d_2}\frac{(d_1,l_1^{k_0})(d_2, l_2^{k_0})(l_1^{k_2}d_1, l_2^{k_2}d_2) (d_1,d_2)}{d_1^{3/2}d_2^{3/2}}\\
&\ll \frac{NT^3}{L^{2k_2+k_0}}\sum_{l_1,l_2\asymp L}\left(1+L^{k_2+k_0}\delta_{l_1=l_2}\right)\ll NT^3L^2\left(\frac{1}{L^{2k_2+k_0}}+\frac{1}{L}\right).
\end{aligned}
\end{equation}
\subsection{Estimating $A_1$ and $A_2$.} By symmetry, it is enough to consider $S_1$. By extracting the oscillations from the Bessel functions using \eqref{bes}, and substituting the definition of $\Phi_{w_6}$, we obtain
\begin{equation}\label{tildem}
I_1(\tilde{m})\ll |U|^{-1/2}\mathcal{K}\left(\theta_1,\theta_2,\pm U,0\right),
\end{equation}where $\mathcal{K}\left(\theta_1,\theta_2,U,V\right)$ as defined in  \eqref{inte}. From the representations given in section \ref{intrep}, we have
\begin{equation}\notag
\mathcal{K}(\theta_1,\theta_2, U,0)\ll\sum_{i=1}^{5}|\mathcal{K}_i(\theta_1, \theta_2, U, 0)|,
\end{equation}where $\mathcal{K}_i(\theta_1,\theta_2, U,V)$ as defined in \eqref{ji}. As earlier, we consider the contribution of only $\mathcal{K}_4(\theta_1,\theta_2, U,0)$, towards which we have the estimate \eqref{v=0}. It is enough to provide the calculation for the last term in the right hand side of \eqref{v=0}. Substituting the bound into \eqref{tildem}, we obtain
\begin{equation}
I_1(\tilde{m})\ll \frac{TL^{k_2/4}}{N^{1/4}}\cdot\frac{d_1}{d_2^{1/2}\tilde{m}^{1/4}}
\end{equation}with the constraint 
\begin{equation}\label{mres2}
\left|\frac{2\pi U}{\theta_1^{1/2}}-1\right|=\left|\frac{\sqrt{\tilde{m}}}{bd_2}-1\right|\ll T^{-1/140},
\end{equation}where $b:=(8\pi^2)^{-1}(rl_1^{k_0}l_2^{k_2-k_1}/f)$. Dividing into dyadic blocks $d_1\asymp H_1\ll D_0/f , d_2\asymp H_2\ll D_0/f$ and $\tilde{m}\asymp Z\ll M_0$ we see that the contribution of this dyadic block towards $A_1$  at most
\begin{equation}\label{917}
\begin{aligned}
&\frac{N^{7/4}L^{k_2/4}}{L^{2k_2}}\sum_{f}\frac{1}{f}\sum_{l_1,l_2\asymp L}\lambda(l_1^{k_2})\overline{\lambda(l_2^{k_2})}\\
&\times\frac{1}{H_1H_2^{5/2}Z^{1/4}}}\sum_{\substack{d_1\asymp H_1\\d_2\asymp H_2}}\,\,\sum_{\substack{y(\bmod f)\\f|(rl_2^{k_2-k_1}d_2+rl_1^{k_2-k_1}d_1y)}}\,\,\,\sum_{\substack{c_1|d_1\\c_2|(d_2, \tilde{M}_f(y))}}c_1c_2\,\,\,\sideset{}{'}\sum_{\substack{\tilde{m}\asymp Z\\c_1|(\tilde{m}+l_1^{k_0}M_f(y))}}1{,
\end{aligned}
\end{equation}where the prime over the $\tilde{m}$ sum denotes the restriction \eqref{mres2}. Note that $c_2|(d_2,\tilde{M}_f(y))$ implies $c_2|(d_2,rl_1^{k_2-k_1}d_1)$ and $c_1|(\tilde{m}+l_1^{k_0}M_f(y))$ implies $c_1|(\tilde{m}f+rl_1^{k_0}l_2^{k_2-k_1}d_2)$. Executing the $c_2$ sum and rearranging, we get that the second line of \ref{917} is bounded by 
\begin{equation}\label{918}
\frac{1}{H_1H_2^{5/2}Z^{1/4}}\sideset{}{'}\sum_{\tilde{m}\asymp Z}\sum_{d_2\asymp H_2}(f,rl_2^{k_2-k_1}d_2)\sum_{c_1|(\tilde{m}f+rl_1^{k_0}l_2^{k_2-k_1}d_2)}c_1\sum_{\substack{d_1\asymp H_1\\c_1|d_1}}(d_2,rl_1^{k_2-k_1}d_1).
\end{equation}The last sum in \ref{918} is bounded by $(H_1/c_1)(d_2,rl_1^{k_2-k_1}c_1)$, which by the congruence condition is at most $(H_1/c_1)(d_2,rl_1^{k_2-k_1}\tilde{m}f)$. Summing over $\tilde{m}$ with the restriction \eqref{mres2}, we get that \ref{918} is at most
\begin{equation}
\begin{aligned}
\frac{T^{-1/140}Z^{3/4}}{H_2^{5/2}}\sum_{d_2\asymp H_2}(f,rl_2^{k_2-k_1}d_2)(d_2,rl_1^{k_2-k_1}f) &\ll \frac{T^{-1/140}Z^{3/4}}{H_2^{5/2}}\sum_{d_2\asymp H_2}(f,d_2)^2\\
&\ll\frac{T^{-1/140}Z^{3/4}}{H_2^{1/2}}.
\end{aligned}
\end{equation}Substituting $Z$ from \eqref{7.8}, we see that \eqref{917} is bounded by 
\begin{equation}\label{920}
\begin{aligned}
&(rL^{k_0+k_2-k_1})^{3/4}\frac{T^{-1/140}N^{7/4}L^{k_2/4}L^2}{L^{2k_2}}\sum_{f\leq D_0}\frac{1}{f^{7/4}}\sup_{H_1,H_2\ll D_0/f} H_2^{1/4}\\
&\ll (rL^{k_0+k_2-k_1})^{3/4}\frac{T^{-1/140}N^{7/4}L^{k_2/4}L^2}{L^{2k_2}}D_0^{1/4}\\
&\ll (rL^{k_0+k_2-k_1})^{5/4}\frac{N^2L^{k_2/4}L^2}{T^{1/2+1/140}L^{2k_2}},
\end{aligned}
\end{equation}which is clearly smaller than \eqref{mainod}.

\section{Remaining contributions}
First let us consider $S_4$ which involves the transform $\Phi_{w_4}(n_1n_2m_2/(D_1D_2))$ with the constraints $D_2|D_1,\,\,n_2D_1=m_1D_2^2$, where $(m_1,m_2)=(rl_1^{k_2-k_1}, ml_1^{k_0})$ and $(n_1,n_2)=(rl_2^{k_2-k_1},nl_2^{k_0})$ in our case. From \eqref{phi5negsmall}, this transform negligibly small unless
\begin{equation}
\frac{n_1n_2m_2}{D_1D_2}\gg T^3
\end{equation}Writing $D_1=D_2q$, we get $D_2=n_2q/m_1$ and hence we must have
\begin{equation}\label{7.2}
\frac{n_1m_1^2m_2}{n_2q^2}\gg T^3.
\end{equation}Substituting the values $(m_1,m_2)=(rl_1^{k_2-k_1}, ml_1^{k_0})$ and $(n_1,n_2)=(rl_2^{k_2-k_1},nl_2^{k_0})$, the above gives
\begin{equation}\label{7.3}
\frac{(rL^{k_2-k_1})^3m}{nq^2}\gg T^3.
\end{equation}Since $m\asymp n\asymp N$ and $r,L$ are small powers of $T$, it is clear that \eqref{7.3} can never happen. Hence the contribution of $\sum_4$ is negligible. One can similarly show that the contribution of $S_5$ is also negligible.

Substituting $\mathcal{E}_{min}, \mathcal{E}_{\max}$ from \eqref{emin} into \eqref{s_i}, we have
\[S_{min}=\frac{1}{24(2\pi i)^2}\int\int_{\Re(\mu)=0}\frac{h(\mu)}{N_{\mu}}\left|\sum_{l\asymp L}\sum_{\substack{n\asymp N\\l^{k_2}|n}}\lambda_{\mu}(rl^{k_2-k_1},nl^{k_0-k_2})\lambda (n)V\left(\frac{n}{N}\right)\right|^2d\mu_1 d\mu_2\geq 0,\]
\begin{equation}\notag
\begin{aligned}
S_{max}=\frac{1}{2\pi i}\sum_{g}\int_{\Re(\mu)=0}&\frac{h(\mu+\mu_g,\mu-\mu_g,-2\mu)}{N_{\mu,g}}\\
&\times \left|\sum_{l\asymp L}\sum_{\substack{n\asymp N\\l^{k_2}|n}}\lambda_{\mu, g}(rl^{k_2-k_1},nl^{k_0-k_2})\lambda (n)V\left(\frac{n}{N}\right)\right|^2d\mu\geq 0.
\end{aligned}
\end{equation}
Hence we ignore $S_{min}$ and  $S_{max}$ due to their negative contribution towards \eqref{s_i}.

\begin{proof}[\textbf{Proof of Theorem \ref{kutz}}]
Follows after combining \eqref{CS} with the estimates  \eqref{920}, \eqref{A_0} and \eqref{mainod}.
\end{proof}
It remains to prove Theorem \ref{dia} and Theorem \ref{average} which we do in the last two sections below.
\section{Proof of Theorem \ref{dia}}\label{pfdia}
Shifting the contours in \ref{58} to $\Re{s_1}=\Re{s_2}=-1/2$, we have (upto a negligible error term)
\begin{equation}
\begin{aligned}
&\frac{1}{f}\mathcal{K}\left(\frac{\theta_1}{f},\frac{\theta_2}{f},0,0\right)=\frac{1}{6}\int_{\Re(\mu)=0}h(\mu)\text{spec}(\mu)\int_{\mathbb{R}}\int_{\mathbb{R}}\mathcal{U}(x)\mathcal{V}(y)\sum_{\epsilon_1,\epsilon_2=\pm 1}\sum_{w\in \mathcal{W}}\\
&\int_{(-1/2)}\int_{(-1/2)}f^{s_1+s_2-1}|4\pi^2\theta_1x^2|^{-s_1}|4\pi^2\theta_2y^2|^{-s_2}G(s,\mu)S^{\epsilon_1,\epsilon_2}(s,w(\mu))\frac{ds}{(2\pi i)^2}\,dx\,dy\,d\mu
\end{aligned}
\end{equation}Summing over $f$ with $(f,F )=1$, the inner $s=(s_1, s_2)$ integral becomes
\begin{equation}\label{112}
\begin{aligned}
\int_{(-1/2)}\int_{(-1/2)}\zeta (1-s_1-s_2)W_{F}(1-s_1-s_2)|4\pi^2\theta_1x^2|^{-s_1}|4\pi^2\theta_2y^2|^{-s_2}G(s,\mu)&S^{\epsilon_1,\epsilon_2}(s,w(\mu))\\&\quad\quad\quad\frac{ds}{(2\pi i)^2}\,,
\end{aligned}
\end{equation}where $W_F(s)=\prod_{p|F}(1-p^{-s})$. We now shift the contours in \eqref{112} to $\Re s_1=\Re s_2 = 1/2-\epsilon$, right past the possible pole at $s_1+s_2=0$. We claim that the residues corresponding to $\epsilon_1,\epsilon_2\in \{\pm 1\}^2,  w\in\mathcal{W}$ all add up to zero (upto a negligible error). Indeed, we first note that
\begin{equation}
G((s,-s),\mu)S^{++}((s,-s),\mu)=G((s,-s),\mu)S^{--}((s,-s),\mu)=0
\end{equation}by the definition in \eqref{55} and \eqref{56}. So, we are left with
\begin{equation}\label{116}
\sum_{w\in\mathcal{W}}\left(S^{+-}((s,-s),w(\mu))+S^{-+}((s,-s),w(\mu))\right)G((s,-s),\mu).
\end{equation}We have
\begin{equation}
\begin{aligned}
&S^{+-}((s,-s),w(\mu))G((s,-s),\mu)\\
&\quad\quad\quad\quad=\frac{\cos\left(\frac{\pi}{2}(\mu_2-\mu_3)\right)}{32\sin\left(\frac{\pi}{2}(\mu_1-\mu_2)\right)\sin\left(\frac{\pi}{2}(\mu_1-\mu_3)\right)(\mu_1-s)(s-\mu_2)(s-\mu_3)}.
\end{aligned}
\end{equation}This is negligibly small unless
\begin{equation}
\Im (\mu_2)<\Im (\mu_1)<\Im (\mu_3)\,\,\,\,\text{or}\,\,\,\,\Im (\mu_3)<\Im (\mu_1)<\Im (\mu_2),
\end{equation}in which case it equals, upto a negligible error term,
\begin{equation}
\frac{1}{16(\mu_1-s)(s-\mu_2)(s-\mu_3)}.
\end{equation}One can similarly show that
\begin{equation}
S^{-+}((s,-s),w(\mu))G((s,-s),\mu)= -\frac{1}{16(\mu_1-s)(s-\mu_2)(s-\mu_3)}+O_A(T^{-A}),
\end{equation}if 
\begin{equation}
\Im (\mu_1)<\Im (\mu_3)<\Im (\mu_2)\,\,\,\,\text{or}\,\,\,\,\Im (\mu_2)<\Im (\mu_3)<\Im (\mu_1),
 \end{equation} and is negligibly small otherwise. Adding the four relevant Weyl chambers with non-negligible contribution, it is now clear that \eqref{116} is essentially zero. Hence the theorem follows after shifting the contours to $\Re s_1=\Re s_2 = 1/2-\epsilon$, and then estimating trivially using the Stirling's formula.
 \\
 \\

\section{Proof of Theorem \ref{average}}\label{pfaverage}
Let us start by recalling 
\begin{equation}\label{1120}
\mathcal{J}_4(y,\mu) =\left|\frac{y_1}{y_2}\right|^{\mu_2/2}\int_0^{1} \tilde{K}_{3\nu_3}\left(|y_1|^{1/2}\sqrt{1-u^2}\right)\tilde{K}_{3\nu_3}\left(|y_2|^{1/2}\sqrt{u^{-2}-1}\right)u^{3\mu_2}\,\frac{du}{u},
\end{equation}and the object we want to bound in this section is
\begin{equation}
\begin{aligned}
&\mathcal{K}_4(\theta_1,\theta_2,U,V)\\
&=\int_{\Re\mu=0}h(\mu)\hbox{spec}(\mu){\int\int}\mathcal{U}(x)\mathcal{V}(y)\mathcal{J}_{4}(\theta_1x^2,\theta_2y^2;\mu)e(Ux+Vy)\,dx\,dy\,d\mu.
\end{aligned}
\end{equation}After substituting \eqref{1120}, we can write the above as
\begin{equation}\label{1121}
\mathcal{K}_4(\theta_1,\theta_2,U,V)=\int_{\Re\mu=0}h(\mu)\left(\frac{\theta_1}{\theta_2}\right)^{\mu_2/2}\hbox{spec}(\mu)\int_{0}^{1}I(U, \alpha(u)) \tilde{I}(V,\beta (u))u^{3\mu_2}\frac{du}{u},
\end{equation}where
\begin{equation}\notag
I(U, \alpha(u)) =\int_{\mathbb{R}}\mathcal{U}(x)x^{\mu_2}\tilde{K}_{3\nu_3}\left(\alpha(u)x\right)e(Ux)\,dx,\,\,\alpha(u)=\theta_1^{1/2}\sqrt{1-u^2},
\end{equation}
\begin{equation}\notag
\tilde{I}(V,\beta (u))=\int_{\mathbb{R}}\mathcal{V}(y)y^{-\mu_2}\tilde{K}_{3\nu_3}\left(\beta(u)y\right)e(Vy)\,dx,\,\,\beta(u)=\theta_2^{1/2}\sqrt{u^{-2}-1}.
\end{equation}Now, the $\mu_2$ integral in \eqref{1121} is given by
\begin{equation}
\int_{\Re\mu_2=0}h(\mu)\hbox{spec}(\mu)\left|\frac{u^3\theta_1^{1/2}x}{\theta_2^{1/2}y}\right|^{\mu_2}\,d\mu_2,
\end{equation}and by integration by parts and \ref{519}, this integral is negligible unless
\begin{equation}\label{usize}
u=\left|\frac{\theta_2^{1/2}y}{\theta_1^{1/2}x}\right|^{1/3}(1+O(T^{-\epsilon})).
\end{equation}Since $u\leq 1$, the first claim \eqref{t1>t2} follows.

The rest of the section is devoted to the estimation of the $u$-integral in \eqref{1121}. For convenience, we denote $T_1=\Im(3\nu_3)=\Im(\mu_3-\mu_1)$ and $T_2=\Im \mu_2$. Note that by our hypothesis
\begin{equation}
T_1\asymp T_2\asymp T.
\end{equation} Denote
\begin{equation}\label{eq11.39}
c_1:=\left(\frac{U\pi}{\theta_1^{1/2}}\right)\left(\frac{\epsilon_2T_1}{\epsilon_1(\epsilon_2T_1-T_2)}\right),\,\,\,c_2:=\left(\frac{V\pi}{\theta_2^{1/2}}\right)\left(\frac{\epsilon'_2T_1}{\epsilon'_1(\epsilon'_2T_1-T_2)}\right)
\end{equation}and
\begin{equation}\notag
c_0:=\frac{\epsilon_2T_1+T_2}{\epsilon_2T_2-T_2},\,\,\,\,\tilde{c}_0:=\frac{\epsilon'_2T_1+T_2}{\epsilon'_2T_1-T_2}.
\end{equation}
 We introduce several dyadic partition of unity and insert localising factors
\begin{equation}\label{g}
G(u):=F_1(T^{\delta_1}(1-u^2))F_2(T^{\delta_2}u)F_3\left(T^{\delta_3}\left(\frac{c^2_1}{1-u^2}-4c_0\right)\right)F_4\left(T^{\delta_4}\left(\frac{c^2_2}{u^{-2}-1}-4\tilde{c}_0\right)\right)
\end{equation}
 to the $u$-integral in \eqref{1121}, so that
 \begin{equation}\label{K4}
\mathcal{K}_4(\theta_1,\theta_2,U,V)\ll T^{3}\sup_{\substack{\mu=\mu_0+O(T^{\epsilon})\\\delta_1,\delta_2\geq 0\\\delta_3,\delta_4\in\mathbb{R}}}\left|\int_{0}^{1}\frac{G(u)}{u}I(U, \alpha(u)) \tilde{I}(V,\beta (u))u^{3\mu_2}\,du\right|.
 \end{equation}To estimate the last $u$-integral, we substitute of asymptotics \eqref{eq1121} for $I(U,\alpha(u))$ and $\tilde{I}(V,\beta(u))$ from Lemma \ref{11.1} to get
\begin{equation}\label{eq1138}
\begin{aligned}
&\int_{0}^{1}\frac{G(u)}{u}I(U, \alpha(u)) \tilde{I}(V,\beta (u))u^{3\mu_2}\,du\\
& \sim\sum_{R_1, R_2}\frac{1}{T}\int_{0}^{1}\frac{G(u)}{u} I_{R_1}^{\epsilon_1,\epsilon_2}(U,\alpha(u)) \tilde{I}_{R_2}^{\epsilon'_1,\epsilon'_2}(V,\beta(u)) u^{3iT_2}\,\frac{du}{u},
\end{aligned}
\end{equation}where $I_{R_1}^{\epsilon_1,\epsilon_2}$ and $\tilde{I}_{R_2}^{\epsilon'_1,\epsilon'_2}$ as in \eqref{Inepsilon}.
Our object of study is now
\begin{equation}\label{os}
\frac{1}{T}\int_{0}^{1}\frac{G(u)}{u}I_{R_1}^{\epsilon_1,\epsilon_2}(U,\alpha(u)) I_{R_2}^{\epsilon'_1,\epsilon'_2}(V,\beta(u)) u^{3iT_2}\,du.
\end{equation}Note that from \eqref{usize}, we can assume
\begin{equation}\label{delta2theta}
u^{-1}\asymp T^{\delta_2}\asymp (\theta_1/\theta_2)^{1/6}.
\end{equation}
\begin{lemma}\label{Gder}
We have
\begin{equation}\label{1140}
G^{j}(u)\ll ((1+T^{\delta_3}+T^{\delta_4})T^{\delta_1+\delta_2})^j,\,\,\,j\geq 0.
\end{equation}
 \end{lemma}	
 \begin{proof}
We first note that since $1-u^2\asymp T^{-\delta_1}$, from Fa\'a di Bruno's formula \ref{bruno}, we get
\begin{equation}
\begin{aligned}
&\frac{\partial^j}{\partial u^j}(T^{\delta_3}(c_1^2(1-u^2)^{-1}-4c_0))\ll_jT^{\delta_3}c_1^2\sum_{\substack{j_1,j_2\\j_1+2j_2=j}}|1-u^2|^{-1-j_1-j_2}|u|^{j_1}\\
&\ll T^{\delta_3}c_1^2 \sum_{j_2\leq j/2}(T^{-\delta_1})^{-1-j+j_2}(T^{-\delta_2})^{j-2j_2}\ll \sum_{j_2\leq j/2} T^{\delta_3}c_1^2T^{\delta_1}T^{j(\delta_1-\delta_2)}(1+T^{-\delta_1+2\delta_2})^{j/2}\\
&\ll T^{\delta_3}c_1^2T^{\delta_1}(T^{\delta_1-\delta_2}+T^{\delta_1/2})^j\ll (1+T^{\delta_3})T^{j\delta_1}.
\end{aligned}
\end{equation}Hence
\begin{equation}\label{F3}
\frac{\partial^j}{\partial u^j}F_3(T^{\delta_3}(c_1^2(1-u^2)^{-1}-4c_0))\ll_j T^{j\delta_1}(1+T^{\delta_3})^j.
\end{equation}Similarly, using the Fa\'a di Bruno's formula we get
\begin{equation}\notag
\begin{aligned}
&\frac{\partial^j}{\partial u^j}(T^{\delta_4}(c_2^2(u^{-2}-1)^{-1}-4\tilde{c}_0))\ll_jT^{\delta_4}c_2^2\sum_{\substack{j_1,j_2,\cdots\\\sum kj_k=j}}|u^{-2}-1|^{-1-\sum j_k}|u|^{-2\sum j_k-j}\\
&\ll T^{\delta_4}c_2^2\sum_{\sum kj_k=j}(T^{-\delta_1+2\delta_2})^{-1-\sum j_k}(T^{-\delta_2})^{-2\sum j_k-j}\\
&\ll T^{\delta_4}c_2^2 T^{-\delta_1+2\delta_2} T^{j\delta_2}\sum_{\sum kj_k=j}T^{\delta_1\sum j_k}\ll T^{\delta_4}c_2^2 T^{-\delta_1+2\delta_2}  T^{(\delta_1+\delta_2)j} \\
&\ll (1+T^{\delta_4})T^{(\delta_1+\delta_2)j} ,
\end{aligned}
\end{equation}and consequently
\begin{equation}\label{F4}
\frac{\partial^j}{\partial u^j}F_4(T^{\delta_4}(c_2^2(u^{-2}-1)^{-1}-4\tilde{c}_0))\ll_jT^{(\delta_1+\delta_2)j}(1+T^{\delta_4})^j.
\end{equation}Using \eqref{F3}, \eqref{F4} and the simple inequalities $F_1^{(j)}(T^{\delta_1}(1-u^2))\ll_j T^{j\delta_1}$ and $F_2^{(j)}(T^{\delta_2}u)\ll_j T^{j\delta_2}$ we arrive at 
\begin{equation}\notag
G^{j}(u)\ll ((1+T^{\delta_3}+T^{\delta_4})T^{\delta_1+\delta_2})^j,\,\,\,j\geq 0.
\end{equation}
 \end{proof}
\subsubsection*{\bf{\textit{Case 1 : The generic case}}} $R_1^{-1} \ll  T^{\beta}$, $R_2^{-1} \ll T^{\beta}$,
\begin{equation}\label{eq11.40}
T^{-\delta_3/2}\gg T^{-1/3+2\epsilon}\left(\frac{T}{\alpha(u)}+1\right)
\end{equation}and
\begin{equation}\label{eq11.41}
T^{-\delta_4/2}\gg T^{-1/3+2\epsilon}\left(\frac{T}{\beta(u)}+1\right).
\end{equation}Note that  we have
\begin{equation}\label{delta2aux}
T^{\delta_2}\asymp \theta_1^{1/6}/\theta_2^{1/6}=(\theta_1^{1/2}/T)(T/\theta_1^{1/3}\theta_2^{1/6})\ll T^{\epsilon}T^{\delta_1/2}(\alpha(u)/T),
\end{equation}where we have used the first condition of \eqref{uvup}. Since $T/\alpha(u)\asymp R_1+1>1$, we conclude 
\begin{equation}\label{delta2small}
T^{\delta_2}\ll T^{\epsilon+\delta_1/2}\implies T^{\delta_2}\ll T^{\epsilon}.
\end{equation}

In this case we claim that the condition \eqref{eq11.27} holds for the corresponding roots of $I_{R_1}^{\epsilon_1,\epsilon_2}(U,\alpha(u))$ and $\tilde{I}_{R_2}^{\epsilon'_1.\epsilon'_2}(V,\beta(u))$. It is enough to prove the claim for $I_{R_1}^{\epsilon_1,\epsilon_2}(U,\alpha(u))$ since the argument for the other is identical. Recall that $z_1(\epsilon_1\alpha(u))$ and $ z_2(\epsilon_1\alpha(u))$ are the two roots of the equation 
 \begin{equation}\label{11.44}
\frac{\partial}{\partial v}\left(-T_2\log(\phi_R(\epsilon_1\alpha(u),v))+\epsilon_2T_1\log v\right)=0,
\end{equation}where
\begin{equation*}
\phi_R(\alpha,v)=\frac{-2\pi}{T_2}\left(\frac{\alpha}{2\pi}\left(Rv-\frac{1}{Rv}\right)+U\right).
\end{equation*}Expanding we get
\begin{equation*}
\begin{aligned}
&\frac{\partial}{\partial v}\left(-T_2\log(\phi_R(\epsilon_1\alpha(u),v))+\epsilon_2T_1\log v\right)\\
&=-\frac{T_2\epsilon_1\alpha(u)\left(v^2+R^{-2}\right)}{v(\epsilon_1\alpha(u)\left(v^2-R^{-2}\right)+R^{-1}2\pi U v)}+\frac{\epsilon_2T_1}{v}\\
&= \frac{\epsilon_1\alpha(u)(\epsilon_2T_1-T_2)v^2+R^{-1}(2\pi U)(\epsilon_2T_1)v-\epsilon_1\alpha(u)R^{-2}(\epsilon_2T_1+T_2)}{v(\epsilon_1\alpha(u)\left(v^2-R^{-2}\right)+R^{-1}2\pi U v)}
\end{aligned}
\end{equation*}Hence $z_1(\epsilon_1\alpha(u))$ and $ z_2(\epsilon_1\alpha(u))$ are the roots of the quadratic equation
\begin{equation*}
v^2+R^{-1}\left(\frac{2\pi U}{\alpha(u)}\right)\left(\frac{\epsilon_2T_1}{\epsilon_1(\epsilon_2T_1-T_2)}\right)v-R^{-2}\left(\frac{\epsilon_2T_1+T_2}{\epsilon_2T_1-T_2}\right).
\end{equation*}Solving we get
\begin{equation}\label{11.45}
z_1(\epsilon_1\alpha(u))=R^{-1}\frac{c_1}{\sqrt{u^2-1}}+R^{-1}\sqrt{\frac{c^2_1}{1-u^2}-c_0}
\end{equation}and
\begin{equation*}
z_2(\epsilon_1\alpha(u))=R^{-1}\frac{c_1}{\sqrt{u^2-1}}-R^{-1}\sqrt{\frac{c^2_1}{1-u^2}-c_0},
\end{equation*}where $c_1$ and $c_0$ as in \eqref{eq11.39}. Hence by the assumption \eqref{eq11.40}
\begin{equation*}
z_1(\epsilon_1\alpha(u))-z_2(\epsilon_1\alpha(u))\asymp R^{-1}\sqrt{\frac{c^2_1}{1-u^2}-c_0}\asymp R^{-1}T^{-\delta_3/2}\gg T^{2\epsilon}R^{-1}(T/\alpha(u)+1)T^{-1/3}.
\end{equation*}Since we are in the range $R+1\asymp T/\alpha(u)$, we get
\begin{equation}\label{rootdiff}
\begin{aligned}
z_1(\epsilon_1\alpha(u))-z_2(\epsilon_1\alpha(u))\gg T^{2\epsilon}R^{-1}(T/\alpha(u)+1)T^{-1/3}&\gg T^{2\epsilon}(R^3T/(R+1)^3)^{-1/3}\\
&\asymp  T^{2\epsilon}(R^3\alpha/(R+1)^2)^{-1/3}.
\end{aligned}
\end{equation}
Hence we can substitute asymptotic expansion \eqref{eq1123} in place of $I_{R_1}^{\epsilon_1,\epsilon_2}(U,\alpha(u))$ in \eqref{os}. By similar arguments, the asymptotic expansion holds for  $\tilde{I}_{R_2}^{\epsilon'_1,\epsilon'_2}(V,\beta(u))$ as well. Let $x_l(\epsilon'_1\beta(u))$ be the corresponding roots of the phase function of $\tilde{I}_{R_2}$. Substituting these approximations we get
\begin{equation}\label{uint}
\begin{aligned}
 &\frac{1}{T}\int_{0}^{1}\frac{G(u)}{u}I_{R_1}^{\epsilon_1,\epsilon_2}(U,\alpha(u)) \tilde{I}_{R_2}^{\epsilon'_1,\epsilon'_2}(V,\beta(u)) u^{3iT_2}\,du\\
 &\approx\frac{1}{T}\cdot\frac{T^{\delta_1/4+\delta_3/4}}{\theta_1^{1/4}}.\frac{T^{\delta_1/4+\delta_4/4-\delta_2/2}}{\theta_2^{1/4}}\sum_{k,l}\int_{0}^{1}H_{k,l}(u)e\left(T_1\cdot\frac{\phi_{k,l}(u)}{2\pi}\right)\,du,
 \end{aligned}
 \end{equation} where 
 \begin{equation}\label{phikl}
 \begin{aligned}
 \phi_{k,l}(u):=-(T_2/T_1)\log\left(\frac{\phi_{R_1}(\epsilon_1\alpha(u),z_k(\epsilon_1\alpha(u)))}{\psi_{R_2}(\epsilon'_1\beta(u),x_l(\epsilon'_1\beta(u)))}\right)+\epsilon_2 \log(z_k(\epsilon_1\alpha(u))) &+\epsilon_2'\log(x_l(\epsilon_1'\beta(u)))\\
 &+3(T_2/T_1)\log u,
 \end{aligned}
 \end{equation}where
 \begin{equation*}
\phi_{R_1}(\alpha,v)=\frac{-2\pi}{T_2}\left(\frac{\alpha}{2\pi}\left(R_1v-\frac{1}{R_1v}\right)+U\right),
 \end{equation*}
\begin{equation*}
\psi_{R_2}(\alpha,v)=\frac{-2\pi}{T_2}\left(\frac{\alpha}{2\pi}\left(R_2v-\frac{1}{R_2v}\right)+V\right),
 \end{equation*} 
 and
 \begin{equation}\label{h}
 H_{k,l}(u):=\frac{G(u)}{u}H_k(u)\tilde{H}_{l}(u),
 \end{equation}where $H_k(u)$ is the weight function coming from $I_{R_1}^{\epsilon_1,\epsilon_2}$ and $H_l(u)$ is the weight function coming from $\tilde{I}_{R_2}^{\epsilon'_1,\epsilon_2'}$, that is,
 \begin{equation}\label{hk}
 H_k(u)=(T^{\delta_1}(1-u^2))^{-1/4}\frac{\omega_{R_1}(\epsilon_1\alpha(u), z_k(\epsilon_1\alpha(u)))}{(RT^{\delta_3/2}|z_1(\epsilon_1\alpha(u))-z_2(\epsilon_1\alpha(u))|)^{1/2}}
 \end{equation}and
 \begin{equation}\label{hl}
 \tilde{H}_l(u)=(T^{\delta_1-2\delta_2}(u^{-2}-1))^{-1/4}\frac{\omega_{R_2}(\epsilon_1'\beta(u), x_l(\epsilon'_1\beta(u)))}{(RT^{\delta_4/2}|x_1(\epsilon'_1\beta(u))-x_2(\epsilon'_1\beta(u))|)^{1/2}},
 \end{equation}where $\omega_{R_1}, \omega_{R_2}$ as in \eqref{11.24}. From the derivative bounds \eqref{eq1126}, \eqref{1140} we deduce
 \begin{lemma}\label{w1}
 We have
 \begin{equation}
  \frac{d^j H_{k,l}(u)}{du^j}\ll_{j}T^{\delta_2}\cdot ((B_1+B_2)T^{\delta_1+\delta_2})^j,\,\,j\geq 1,
 \end{equation}where
 \begin{equation*}
 B_1:=(1+R_1^{-1}T^{-\delta_3/2})(1+T^{\delta_3})
  \end{equation*}
 \begin{equation*}
B_2:= (1+R_2^{-1}T^{-\delta_4/2})(1+T^{\delta_4}).
 \end{equation*}
\end{lemma}
\begin{proof}
We first note that since $1-u^2\asymp T^{-\delta_1}, u\asymp T^{-\delta_2}$, from Fa\'a di Bruno's formula \ref{bruno}, we get
\begin{equation}\label{alphader}
\begin{aligned}
\frac{\partial^j}{\partial u^j}\alpha (u)=\frac{\partial^j}{\partial u^j}(\theta_1^{1/2}(1-u^2)^{1/2})&\ll_j \theta_1^{1/2}\sum_{\substack{j_1,j_2\\j_1+2j_2=j}}|1-u^2|^{1/2-j_1-j_2}|u|^{j_1}\\
&\ll \theta_1^{1/2} \sum_{j_2\leq j/2}(T^{-\delta_1})^{1/2-j+j_2}(T^{-\delta_2})^{j-2j_2}\\
&\ll \theta_1^{1/2}T^{-\delta_1/2}T^{j(\delta_1-\delta_2)}(1+T^{-\delta_1+2\delta_2})^{j/2}\\
&\ll\theta_1^{1/2}T^{-\delta_1/2}(T^{\delta_1-\delta_2}+T^{\delta_1/2})^j.
\end{aligned}
\end{equation}Using \eqref{alphader}, \eqref{eq1126} and invoking Fa\'a di Bruno's formula \ref{bruno} once again  we obtain
\begin{equation}\label{uweight1}
\begin{aligned}
&\frac{\partial^j}{\partial u^j}\frac{\omega_{R_1}(\epsilon_1\alpha(u), z_k(\epsilon_1\alpha(u)))}{(RT^{\delta_3/2}|z_1(\epsilon_1\alpha(u))-z_2(\epsilon_1\alpha(u))|)^{1/2}}\\
&\ll_j \sum_{\substack{j_1,j_2,\cdots\\\sum kj_k=j}}(|\alpha(u)|^{-1}B_1)^{\sum j_k}\prod_{k}(\theta_1^{1/2}T^{-\delta_1/2}(T^{\delta_1-\delta_2}+T^{\delta_1/2})^k)^{ j_k}\\
&\ll (T^{\delta_1-\delta_2}+T^{\delta_1/2})^j \sum_{\substack{j_1,j_2,\cdots\\\sum kj_k=j}}B_1^{\sum j_k}\ll (B_1(T^{\delta_1-\delta_2}+T^{\delta_1/2}))^j\ll (B_1T^{\delta_1+\delta_2})^j.
\end{aligned}
\end{equation}Using the Fa\'a di Bruno's formula \ref{bruno}, one similarly obtains
\begin{equation}\label{uweight2}
\begin{aligned}
\frac{\partial^j}{\partial u^j}(T^{\delta_1}(1-u^2))^{-1/4}\ll_j \sum_{\substack{j_1,j_2\\j_1+2j_2=j}}(T^{\delta_1}u)^{j_1}(T^{\delta_1})^{j_2}&\ll \sum_{j_2\leq j/2}(T^{\delta_1-\delta_2})^{j-2j_2}(T^{\delta_1})^{j_2}\\
&\ll (T^{\delta_1-\delta_2}+T^{\delta_1/2})^j\\
&\ll T^{j(\delta_1+\delta_2)}.
\end{aligned}
\end{equation}Combining \eqref{uweight1} and \eqref{uweight2}, we obtain
\begin{equation}\label{H_k}
\frac{\partial^j}{\partial u^j}H_k(u)\ll_j (B_1T^{\delta_1+\delta_2})^j.
\end{equation}For $\beta(u)=\theta_2^{1/2}(u^{-2}-1)^{1/2}=\theta_2^{1/2}u^{-1}(1-u^2)^{1/2}$, using \eqref{alphader} and the fact that $(\partial^j/\partial u^j)u^{-1}\ll_j T^{\delta_2}T^{j\delta_2}$, we obtain that for $j\geq 1$,
\begin{equation}\notag
\begin{aligned}
\frac{\partial ^j}{\partial u^j} \beta(u)&\ll_j \theta_2^{1/2}T^{-\delta_1/2+\delta_2}(T^{\delta_1-\delta_2}+T^{\delta_1/2}+T^{\delta_2})^j\ll \theta_2^{1/2}T^{-\delta_1/2+\delta_2} T^{j(\delta_1+\delta_2)}.
\end{aligned}
\end{equation}Doing a similar calculation as in \eqref{uweight1} and \eqref{uweight2}, we arrive at
\begin{equation}\notag
\frac{\partial ^j}{\partial u^j}\frac{\omega_{R_2}(\epsilon_1'\beta(u), x_l(\epsilon'_1\beta(u)))}{(R_2T^{\delta_4/2}|x_1(\epsilon'_1\beta(u))-x_2(\epsilon'_1\beta(u))|)^{1/2}}\ll_j (B_2T^{\delta_1+\delta_2})^j,
\end{equation}and
\begin{equation}\notag
\frac{\partial^j}{\partial u^j}(T^{\delta_1-2\delta_2}(u^{-2}-1))^{-1/4}\ll_j T^{j(\delta_1+\delta_2)},
\end{equation}from which it follows
\begin{equation}\label{H_l}
\frac{\partial^j}{\partial u^j}H_l(u)\ll_j (B_2T^{\delta_1+\delta_2})^j.
\end{equation}Using Lemma \ref{Gder}, it is easy to see that
\begin{equation}\label{G}
\frac{\partial^j}{\partial u^j}\frac{G(u)}{u}\ll_j T^{\delta_2}((1+T^{\delta_3}+T^{\delta_4})T^{\delta_1+\delta_2})^j,\,\,\j\geq 1.
\end{equation}From \eqref{H_k}, \eqref{H_l} and \eqref{G}, it follows
\begin{equation}\notag
\frac{\partial^j}{\partial u^j} H_{k,l}(u)=\frac{\partial^j}{\partial u^j}\frac{G(u)}{u}H_k(u)H_{l}(u)\ll_j T^{\delta_2}((B_1+B_2)T^{\delta_1+\delta_2})^j.
\end{equation}The lemma follows.
\end{proof}	

\begin{lemma}\label{w2}
We have
\begin{equation}
 \frac{d^j \phi_{k,l}(u)}{du^j}\ll_{j} ((B_1+B_2)T^{\delta_1+\delta_2})^j,\,\,j\geq 1,
\end{equation}where $B_1, B_2$ as in Lemma \ref{w1}.
\end{lemma} 
\begin{proof}
Same as the proof of the previous lemma after using the derivative bounds \eqref{rootder} and \eqref{phiRder} for each individual term in \eqref{phikl}. Note the term \eqref{G} is not present in this case and hence the absence of the factor $T^{\delta_2}$ in the statement of the lemma.
\end{proof}
\begin{lemma}\label{sign}
Let $\delta_1,\delta_2,\delta_3$ and $\delta_4$ as in \eqref{g}. Then $T^{-\delta_3}, T^{-\delta_4}\ll T^{\delta_1}$.
\end{lemma}	
\begin{proof}
By definition, we have
\begin{equation}\label{defdelta34}
T^{-\delta_3}\asymp \frac{c_1^2}{1-u^2}-4c_0,\,\,\,T^{-\delta_4}\asymp \frac{c_2^2}{u^{-2}-1}-4\tilde{c}_0
\end{equation}Now
\begin{equation}\notag
\frac{c_1^2}{1-u^2}\asymp T^{\delta_1}c_1^2=T^{\delta_1} U^2/\theta_1\ll T^{\delta_1}(1+(\theta_2/\theta_1)^{1/3})\asymp T^{\delta_1}(1+T^{-2\delta_2}),
\end{equation}and
\begin{equation}\notag
\frac{c_2^2}{u^{-2}-1}\asymp T^{\delta_1-2\delta_2}V^2/\theta_2\asymp T^{\delta_1} V^2/(\theta_1^{1/3}\theta_2^{2/3})\ll T^{\delta_1}(1+(\theta_2/\theta_1)^{1/3})\asymp T^{\delta_1}(1+T^{-2\delta_2}),
\end{equation}where we have used \eqref{uvup} and \eqref{delta2theta} in the last two inequalities. The claim follows after substituting the above bounds into \eqref{defdelta34} and observing that $\delta_2\geq 0$.
\end{proof}
\begin{lemma}\label{•}
Define the quantities
\begin{equation}\label{quan}
A:= \left|4c_1^2/(c_0-1)^2-1\right|=\left|(\pi U/\theta_1^{1/2})^2-1\right|,
\end{equation}
\begin{equation*}
B_3:=(B_1+B_2)T^{\delta_1-\delta_2}((1+T^{-\delta_1/2-\delta_4/2}).
\end{equation*}
We can write
\begin{equation}\label{reduc}
\int_{0}^{1}H_{k,l}(u)e\left(\frac{T_1\phi_{k,l}(u)}{2\pi}\right)\,du=\int_{\mathbb{R}}F(t)e(T_1\psi(t))\,dt,
\end{equation}where $F(t)$ is a compactly supported smooth function and the phase function $\psi(t)$ is such that
\begin{equation*}
\psi'(t)=K_0\frac{Q(t)}{h(t)},\,\,\,\,\,K_0:=AT^{\delta_1-12\delta_2},
\end{equation*}where $Q(t)$ is a degree 13 monic polynomial in $t$ and $h(t)$ is a smooth bounded function for $t\in\text{supp}(F)$. Furthermore, for $t\in \text{supp}(F)$ we have the following derivative bounds.
\begin{equation*}
F^{(j)}(t)\ll_j T^{-\delta_2}(1+T^{-\delta_4/2-\delta_1/2})B_3^{j},\,\,j\geq 1,
\end{equation*}
\begin{equation*}
F(t)\ll T^{-\delta_2-\delta_1/2-\delta_4/2},
\end{equation*}and
\begin{equation*}
\psi^{(j)}(t)\ll_j B_3^{j},\,\,j\geq 1.
\end{equation*}
\end{lemma}	
\begin{proof}
It follows from \eqref{11.44} and \eqref{phikl} that
\begin{equation}\label{11.47} 
\begin{aligned}
\phi'_{k,l}(u)=&-(T_2/T_1)\epsilon_1\alpha'(u)\left(\frac{1}{\phi_{R_1}}\frac{\partial \phi_{R_1}(\alpha,v)}{\partial\alpha}\right)\Biggr|_{(\alpha,v)=(\epsilon_1\alpha(u),z_k(\epsilon_1\alpha(u)))}\\
&+(T_2/T_1)\epsilon'_1\beta'(u)\left(\frac{1}{\psi_{R_2}}\frac{\partial \psi_{R_2}(\alpha,v)}{\partial\alpha}\right)\Biggr|_{(\alpha,v)=(\epsilon'_1\beta(u),x_l(\epsilon'_1\beta(u)))}\\
&+\frac{3(T_2/T_1)}{u}.
\end{aligned}
\end{equation}We carefully evaluate each of the first two terms in the r.h.s. of \eqref{11.47} in the following two lemma. Denote 
\begin{equation}\label{11.48}
t:=\sqrt{c^2_2-\tilde{c}_0(u^{-2}-1)}.
\end{equation}We begin with the second term of \eqref{11.47}.
\begin{lemma}\label{•}
We have
\begin{equation}\label{der1}
(T_2/T_1)\epsilon'_1\beta'(u)\left(\frac{1}{\psi_{R_2}}\frac{\partial \psi_{R_2}(\alpha,v)}{\partial\alpha}\right)\Biggr|_{(\alpha,v)=(\epsilon'_1\beta(u),x_l(\epsilon'_1\beta(u)))}=\frac{-\epsilon'_2(P_1(t))^{3/2}Q_1(t)}{P_2(t)Q_2(t)},
\end{equation}where
\begin{equation}\label{1150}
\begin{aligned}
&P_1(t)= -t^2/\tilde{c}_0+(1+c_2^2/\tilde{c}_0),\,\,\,P_2(t)=(c_2^2-t^2)/\tilde{c}_0,\\
&Q_1(t)=(-(\tilde{c}_0+1)/\tilde{c}_0)t^2\pm 2c_2t-(\tilde{c}_0-1)c_2^2/\tilde{c}_0,\\&Q_2(t)=(-(\tilde{c}_0-1)/\tilde{c}_0)t^2\pm 2c_2t-(\tilde{c}_0+1)c_2^2/\tilde{c}_0
\end{aligned}
\end{equation}where $t$ as in \eqref{11.48}. Furthermore, for $u\in \text{supp}(H_{k,l})$, the sizes of these polynomials are 
\begin{equation}\label{11.50}
\begin{aligned}
P_1(t)\asymp T^{2\delta_2},\,P_2(t)\asymp T^{-\delta_1+2\delta_2}\,\,\text{and}\,\,\,\, Q_1(t),Q_2(t)\ll T^{2\delta_2+\epsilon}.
\end{aligned}
\end{equation}
\end{lemma}	
\begin{proof}
Using the fact that $x_l:=x_l(\epsilon'_1\beta(u))$ satisfies \eqref{11.44}, we get
\begin{equation}\notag
\frac{1}{\psi_{R_2}(\epsilon'_1\beta(u), x_l(\epsilon'_1\beta(u)))}=-\frac{\epsilon'_2T_1}{\beta(u) (R_2x_l+(R_2x_l)^{-1})},
\end{equation}and hence
\begin{equation}\label{2ndterm}
T_2\epsilon'_1\beta'(u)\left(\frac{1}{\psi_{R_2}}\frac{\partial \psi_{R_2}(\alpha,v)}{\partial\alpha}\right)\Biggr|_{(\alpha,v)=(\epsilon'_1\beta(u),x_l(\epsilon'_1\beta(u)))}=-\frac{\epsilon_2'T_1(x_l^2-R_2^{-2})}{u(1-u^2)(x_l^2+R_2^{-2})}.
\end{equation}Substituting the roots (see \eqref{11.45} in case of $z_k$) ,we get
\begin{equation}\label{xl}
\begin{aligned}
\frac{x_l^2-R_2^{-2}}{x_l^2+R_2^{-2}}&=\frac{(\tilde{c}_0+1)(u^{-2}-1)-2c_2^2\pm 2c_2\sqrt{c_2^2-\tilde{c}_0(u^{-2}-1)}}{(\tilde{c}_0-1)(u^{-2}-1)-2c_2^2\pm 2c_2\sqrt{c_2^2-\tilde{c}_0(u^{-2}-1)}}.
\end{aligned}
\end{equation}In terms of $t=\sqrt{c^2_2-\tilde{c}_0(u^{-2}-1)}$, the numerator becomes
\begin{equation*}
\begin{aligned}
(\tilde{c}_0+1)(u^{-2}-1)-2c_2^2\pm 2c_2\sqrt{c_2^2-\tilde{c}_0(u^{-2}-1)}&=-(\tilde{c}_0+1)t^2/\tilde{c}_0\pm 2c_2t - (\tilde{c}_0-1)c_2^2/\tilde{c_0}\\
&=Q_1(t),
\end{aligned}
\end{equation*}and the denominator
\begin{equation*}
\begin{aligned}
(\tilde{c}_0-1)(u^{-2}-1)-2c_2^2\pm 2c_2\sqrt{c_2^2-\tilde{c}_0(u^{-2}-1)}&=-(\tilde{c}_0-1)t^2/\tilde{c}_0\pm2 c_2t - (\tilde{c}_0+1)c_2^2/\tilde{c_0}\\
&=Q_2(t).
\end{aligned}
\end{equation*}We also get
\begin{equation}\notag
u^{-1}=\sqrt{-t^2/\tilde{c}_0+(1+c_2^2/\tilde{c}_0)}=\sqrt{P_1(t)},
\end{equation}
\begin{equation}\notag
u^{-2}-1=(c_2^2-t^2)/\tilde{c}_0=P_2(t).
\end{equation}
 The first part of the claim follows after substituting the above transformations into \eqref{2ndterm}. For the sizes of the polynomials, note from the definitions it follows $P_1(t)=u^{-2}\asymp T^{2\delta_2}$ and $P_2(t)= u^{-2}-1\asymp T^{-\delta_1+2\delta_2}$. For $Q_1, Q_2$, we have from \eqref{xl}, 
\begin{equation}\notag
 Q_1(t)=R_2^2(u^{-2}-1)(x_l^2-R_2^{-2}), Q_2(t)=R_2^2(u^{-2}-1)(x_l^2+R_2^{-2}).
 \end{equation} Now from the support of $\omega_{R_2}(\epsilon_1'\beta(u), x_l(\epsilon'_1\beta(u)))$ function  in \eqref{hl}, it follows $x_l\ll 1+R_2^{-1}$. Hence
 \begin{equation}\label{Q1Q2bd}
 Q_1(t),Q_2(t)\ll T^{-\delta_1+2\delta_2}(1+R_2)^2\asymp T^{-\delta_1+2\delta_2}(T/\beta(u))^2\ll T^{2\delta_2+\epsilon},
 \end{equation}since $\beta(u)\asymp T^{-\delta_1/2}T^{\delta_2}\theta_2^{1/2}\asymp  T^{-\delta_1/2} \theta_1^{1/6}\theta_2^{1/3}\gg T^{-\delta_1/2} T^{1-\epsilon}$. 
 
\end{proof}
\begin{lemma}\label{•}
We have
\begin{equation}\label{der2}
\begin{aligned}
&-(T_2/T_1)\epsilon_1\alpha'(u)\left(\frac{1}{\phi_{R_1}}\frac{\partial \phi_{R_1}(\alpha,v)}{\partial\alpha}\right)\Biggr|_{(\alpha,v)=(\epsilon_1\alpha(u),z_k(\epsilon_1\alpha(u)))}\\
&=\frac{\epsilon_2\sqrt{P_1(t)}(Q_3(t)\pm c_1\sqrt{P_1(t)P_3(t)})}{P_2(t)(Q_4(t)\pm c_1\sqrt{P_1(t)P_3(t)})},
\end{aligned}
\end{equation}where
\begin{equation*}
\begin{aligned}
&P_3(t)=(c_1^2-c_0)P_1(t)-c_0,\,\,\,\, Q_3(t)=(c_0+1)P_2(t)-2c_1^2P_1(t), \\
&Q_4(t)=(c_0-1)P_2(t)-2c_1^2P_1(t),
\end{aligned}
\end{equation*}where $P_1(t)$ and $P_2(t)$ as in \eqref{1150}. The sizes are given by
\begin{equation}\label{11.55}
\begin{aligned}
&(Q_3(t)\pm c_2\sqrt{P_1(t)P_3(t)}),\,\,(Q_4(t)\pm c_2\sqrt{P_1(t)P_3(t)})\ll T^{-\delta_1}\ll T^{\epsilon}.
\end{aligned}
\end{equation}
\end{lemma}	
\begin{proof}
Using the fact that $z_k:=z_k(\epsilon_1\alpha(u))$ satisfies \eqref{11.44} we get
\begin{equation}\notag
\frac{1}{\phi_{R_1}(\epsilon_1\alpha(u),z_k(\epsilon_1\alpha(u)))}=-\frac{\epsilon_2T_1}{\alpha(u)(R_1z_k+(R_1z_k)^{-1})}
\end{equation}and hence
\begin{equation}\label{11.51}
\begin{aligned}
-T_2\epsilon_1\alpha'(u)\left(\frac{1}{\phi_{R_1}}\frac{\partial \phi_{R_1}(\alpha,v)}{\partial\alpha}\right)\Biggr|_{(\alpha,v)=(\epsilon_1\alpha(u),z_k(\epsilon_1\alpha(u)))}&=\frac{\epsilon_2T_1u (z_k^2-R_1^{-2})}{(1-u^2)(z_k^2+R_1^{-2})}\\
&=\frac{\epsilon_2T_1 (z_k^2-{R_1}^{-2})}{u(u^{-2}-1)(z_k^2+{R_1}^{-2})}
\end{aligned}
\end{equation}Substituting the roots from \eqref{11.45} we get
\begin{equation*}
\begin{aligned}
\frac{ (z_k^2-R_1^{-2})}{(z_k^2+R_1^{-2})}&=\frac{(c_0+1)(1-u^2)-2c_1^2\pm 2c_1\sqrt{c_1^2-c_0(1-u^2)}}{(c_0-1)(1-u^2)-2c_1^2\pm 2c_1\sqrt{c_1^2-c_0(1-u^2)}}\\
&=\frac{(c_0+1)(u^{-2}-1)-2c_1^2u^{-2}\pm 2c_1u^{-1}\sqrt{c_1^2u^{-2}-c_0(u^{-2}-1)}}{(c_0-1)(u^{-2}-1)-2c_1^2u^{-2}\pm 2c_1u^{-1}\sqrt{c_1^2u^{-2}-c_0(u^{-2}-1)}}
\end{aligned}
\end{equation*}In terms of $t=\sqrt{c^2_2-\tilde{c}_0(u^{-2}-1)}$, the numerator becomes
\begin{equation*}
\begin{aligned}
&(c_0+1)(u^{-2}-1)-2c_1^2u^{-2}\pm 2c_1u^{-1}\sqrt{c_1^2u^{-2}-c_0(u^{-2}-1)}\\
&=(c_0+1)P_2(t)-2c_1^2P_1(t)\pm 2c_1\sqrt{P_1(t)}\sqrt{(c_1^2-c_0)P_1(t)+c_0}\\
&=Q_3(t)\pm 2c_1\sqrt{P_1(t)P_3(t)}
\end{aligned}
\end{equation*}and the denominator
\begin{equation*}
\begin{aligned}
&(c_0-1)(u^{-2}-1)-2c_1^2u^{-2}\pm 2c_1u^{-1}\sqrt{c_1^2u^{-2}-c_0(u^{-2}-1)}\\
&=(c_0-1)P_2(t)-2c_1^2P_1(t)\pm 2c_1\sqrt{P_1(t)}\sqrt{(c_1^2-c_0)P_1(t)+c_0}\\
&=Q_4(t)\pm 2c_1\sqrt{P_1(t)P_3(t)}.
\end{aligned}
\end{equation*}
The first part of claim follows after substituting these expressions in \eqref{11.51}. For the last part, following the calculation in \eqref{Q1Q2bd}, we obtain
\begin{equation}\notag
(Q_3(t)\pm c_2\sqrt{P_1(t)P_3(t)}),\,\,(Q_4(t)\pm c_2\sqrt{P_1(t)P_3(t)})\ll T^{-\delta_1}(1+R_1)^2\ll T^{\epsilon}.
\end{equation}
\end{proof}Finally the third term in \eqref{11.47} in terms of $t=\sqrt{c^2_2-\tilde{c}_0(u^{-2}-1)}$ becomes
\begin{equation}\label{der3}
\frac{3T_2}{u}=3T_2\sqrt{P_1(t)}
\end{equation}where $P_1(t)$ as in \eqref{1150}. Combining \eqref{11.47}, \eqref{der1}, \eqref{der2} and \eqref{der3} we obtain
\begin{equation}\label{11.58}
\phi_{k,l}'(u)=\frac{-\epsilon'_2T_1(P_1(t))^{3/2}Q_1(t)}{P_2(t)Q_2(t)}+\frac{\epsilon_2T_1\sqrt{P_1(t)}(Q_3(t)\pm 2c_1\sqrt{P_1(t)P_3(t)})}{P_2(t)(Q_4(t)\pm 2c_1\sqrt{P_1(t)P_3(t)})}+3T_2\sqrt{P_1(t)}.
\end{equation}We now rationalise the expression in \eqref{11.58} so that numerator becomes a polynomial in $t$. Let $P_i,Q_i$ denote the shorthand for $P_i(t),Q_i(t)$. Simplifying we get
\begin{equation}\label{11.59}
\begin{aligned}
\phi_{k,l}'(u)=\frac{T_1P_1^{1/2}\left(2c_1P_1^{1/2}P_3^{1/2}F_{\pm}+G\right)}{P_2Q_2(Q_4\pm 2c_1P_1^{1/2}P_3^{1/2})}\\
\end{aligned}
\end{equation}where $F_{\pm}:=\mp\epsilon'_2P_1Q_1\pm\epsilon_2Q_2\pm cP_2Q_2$ and $G:=-\epsilon'_2P_1Q_1Q_4+\epsilon_2Q_2Q_3+(3T_2/T_1)P_2Q_2Q_4$. Multiplying and dividing $2c_1P_1^{1/2}P_3^{1/2}F_{\pm}-G$, we get
\begin{equation}\label{11.60}
\begin{aligned}
\phi_{k,l}'(u)&=\frac{T_1P_1^{1/2}\left(4c_1^2P_1P_3F^2_{\pm}-G^2\right)}{P_2Q_2(Q_4\pm 2c_1P_1^{1/2}P_3^{1/2})(2c_1P_1^{1/2}P_3^{1/2}F_{\pm}-G)}.
\end{aligned}
\end{equation}Note that $4c_1^2P_1P_3F^2_{\pm}$ is a degree 12 polynomial in $t$ with leading coefficient 
\begin{equation*}
\frac{4c_1^2(c_1^2-c_0)(\epsilon'_2(\tilde{c}_0+1)-(3T_2/T_1)(\tilde{c}_0-1))^2}{\tilde{c}_0^6}=\left(\frac{4T_2}{\tilde{c}_0^3(T_2+\epsilon'_2T_1)}\right)^24c_1^2(c_1^2-c_0),
\end{equation*}and $G^2$ is also degree 12 polynomial with leading coefficient
\begin{equation*}
\frac{(2c_1^2-(c_0-1))^2(\epsilon'_2(\tilde{c}_0+1)-(3T_2/T_1)(\tilde{c}_0-1))^2}{\tilde{c}_0^6}=\left(\frac{4T_2}{\tilde{c}_0^3(T_2+\epsilon'_2T_1)}\right)^2(2c_1^2-(c_0-1))^2.
\end{equation*}Hence, the numerator $c_1^2P_1P_3F^2_{\pm}-G^2$ in \eqref{11.60} is a degree 12 polynomial in $t$ with leading coefficient 
\begin{equation}
\left(\frac{4T_2}{\tilde{c}_0^3(T_2+\epsilon'_2T_1)}\right)^2\left(4c_1^2(c_1^2-c_0)-(2c_1^2-(c_0-1))^2\right)\asymp \left(4c_1^2/(c_0-1)^2-1\right) =:A
\end{equation}Also, from the bounds in \eqref{11.50}, \eqref{11.55}, we get
\begin{equation}\label{11.62}
2c_1P_1^{1/2}P_3^{1/2}F_{\pm}\pm G \ll T^{6\delta_2+\epsilon}.
\end{equation}\eqref{11.62} combined with the bounds from \eqref{11.50} and \eqref{11.55} shows that the denominator in \eqref{11.60} is bounded by 
\begin{equation}
T^{10\delta_2-\delta_1+\epsilon}.
\end{equation}Summarising the above, we have
\begin{equation*}
\phi_{k,l}'(u)=T_1\left(A T^{\delta_1-9\delta_2-\epsilon}\right)\frac{P(t)}{f(t)},
\end{equation*}where $P(t)$ is a degree 12 monic polynomial in $t$ and $f(t)\ll 1$ is a smooth bounded function. Hence, after the change of variable $\sqrt{c_2^2-\tilde{c}_0(u^{-2}-1)}\mapsto t$, the above discussion boils down to the equality
\begin{equation*}
\int_{0}^{1}H_{k,l}(u)e\left(\frac{T_1\phi_{k,l}(u)}{2\pi}\right)\,du=\int_{\mathbb{R}}F(t)e(T_1\psi (t))\,dt,
\end{equation*}where
\begin{equation*}
F(t)=-\frac{P_1'(t)H_{k,l}((P_1(t))^{-1/2})}{2(P_1(t))^{3/2}},\,\,\,\,  \psi(t)=\frac{\phi_{k,l}((P_1(t))^{-1/2})}{2\pi}
\end{equation*}such that
\begin{equation}\label{11.65}
\psi'(t)\asymp K_0\cdot\frac{tP(t)}{T^{-3\delta_2}(P_1(t))^{3/2}f(t)},
\end{equation}where
\begin{equation*}
K_0:=AT^{\delta_1-12\delta_2}
\end{equation*}Here $P_1(t)$ is as in \eqref{11.50} with size $T^{2\delta_2}$, so that the denominator in \eqref{11.65} is a smooth bounded function, and the numerator, $tP(t)$, is a degree 13 monic polynomial. This completes the proof of first part of the lemma. For the second part, note that for $\alpha\in\mathbb{R}$ and $P_1(t)=-t^2/\tilde{c}_0+(1+c_2^2/\tilde{c}_0)\asymp T^{2\delta_2}, t\asymp T^{-\delta_4/2-\delta_1/2+\delta_2}$, using the Fa\'a di Bruno's formula \ref{bruno}, we obtain
\begin{equation}\label{P1}
\begin{aligned}
&\frac{d^j}{dt^j}P_1(t)^{\alpha}\ll_j \sum_{j_1+2j_2=j} (T^{2\delta_2})^{\alpha-j_1-j_2}T^{j_1(-\delta_4/2-\delta_1/2+\delta_2)}\\
&=\sum_{j_1+2j_2=j}T^{2\delta_2\alpha} (T^{2\delta_2})^{-j+j_2}T^{(j-2j_2)(-\delta_4/2-\delta_1/2+\delta_2)}\ll T^{2\delta_2\alpha}T^{-2\delta_2j}\sum_{j_2\leq j/2}T^{j_2(\delta_4+\delta_1)}\\
&\ll T^{2\delta_2\alpha}(T^{-2\delta_2}(1+T^{\delta_1/2+\delta_4/2}))^j,
\end{aligned}
\end{equation}for $j\geq 0$. Using the above with $\alpha=-1/2$, and using Lemma \ref{w1} and the Fa\'a di Bruno's formula \ref{bruno} we obtain
\begin{equation}\notag
\begin{aligned}
\frac{d^j}{dt^j}H_{k,l}((P_1(t))^{-1/2})&\ll_j T^{\delta_2}(T^{-2\delta_2}(1+T^{\delta_1/2+\delta_4/2}))^j\sum_{\sum kj_k=j}((B_1+B_2)T^{\delta_1+\delta_2})^{\sum j_k}T^{-\delta_2\sum j_k}\\
&\ll T^{\delta_2}((B_1+B_2)T^{\delta_1-2\delta_2}(1+T^{\delta_1/2+\delta_4/2}))^j.
\end{aligned}
\end{equation}Using the last inequality and \eqref{P1} with $\alpha=-3/2$, we deduce
\begin{equation}\notag
\frac{d^j}{dt^j}\frac{H_{k,l}((P_1(t))^{-1/2})}{P_1(t)^{3/2}}\ll_j T^{-2\delta_2}((B_1+B_2)T^{\delta_1-2\delta_2}(1+T^{\delta_1/2+\delta_4/2}))^j.
\end{equation}Since $P'_1(t)=-2t/\tilde{c}_0\asymp T^{-\delta_4/2-\delta_1/2+\delta_2} $, we finally obtain
\begin{equation}\notag
\begin{aligned}
F^{(j)}(t)&\ll_j|t|\frac{d^j}{dt^j}\frac{H_{k,l}((P_1(t))^{-1/2})}{P_1(t)^{3/2}}+\frac{d^{j-1}}{dt^{j-1}}\frac{H_{k,l}((P_1(t))^{-1/2})}{P_1(t)^{3/2}}\\
&\ll_j T^{-\delta_2}(1+T^{-\delta_4/2-\delta_1/2})((B_1+B_2)T^{\delta_1-\delta_2}((T^{-\delta_1/2-\delta_4/2}+1))^j\\
&=T^{-\delta_2}(1+T^{-\delta_4/2-\delta_1/2})B_3^{j}.
\end{aligned}
\end{equation}
Similarly, using  Lemma \ref{w2} one can deduce
\begin{equation*}
\begin{aligned}
\psi^{(j)}(t)&\ll_j \left((B_1+B_2)T^{\delta_1}(T^{-\delta_2-\delta_1/2-\delta_4/2}+T^{-\delta_2})\right)^j,\,\,j\geq 1\\
&=B_3^{j}.
\end{aligned},
\end{equation*}For the last part, note that for $u\in \text{supp} (G(u))$, where $G$ as in \eqref{g},  $$t=\sqrt{c_2^2-\tilde{c}_0(u^{-2}-1)}\asymp T^{-\delta_1/2-\delta_4/2+\delta_2}.$$ Similarly for $H_{k,l}$ as in \eqref{h} we have $H_{k,l}(u)\ll T^{\delta_2}$. Lastly from \eqref{11.50} we have $P_1(t)\asymp T^{2\delta_2}$. Hence we get
\begin{equation*}
F(t)=-\frac{P_1'(t)H_{k,l}((P_1(t))^{-1/2})}{2(P_1(t))^{3/2}}\ll T^{-\delta_2-\delta_1/2-\delta_4/2}.
\end{equation*}This completes the proof the the lemma.
\end{proof}Applying Lemma \ref{s3} to the right hand side of \eqref{reduc} we obtain
\begin{equation}\label{finalbd}
\begin{aligned}
&\int_{0}^{1}H_{k,l}(u)e\left(\frac{\phi_{k,l}(u)}{2\pi}\right)\,du\\
&\ll T^{-\delta_2-\delta_1/2-\delta_4/2}\Bigg((TK_0)^{1/14}+\left(\frac{TK_0}{B_3}\right)^{-1/13}+\left(\frac{T^{1/2}K_0}{B_3}\right)^{-1/13}\Bigg)\\
&\ll T^{-\delta_2-\delta_1/2-\delta_4/2}T^{-1/26}B_3^{1/13}(K_0^{-1/13}+K_0^{-1/14}).
\end{aligned}
\end{equation}Note that we also have the trivial bound
\begin{equation}\label{trivialbd}
\int_{0}^{1}H_{k,l}(u)e\left(\frac{\phi_{k,l}(u)}{2\pi}\right)\,du\ll \min\{T^{-\delta_1}, T^{-\delta_3+\delta_2}, T^{-\delta_4}\}
\end{equation}coming from the size of support of $H_{k,l}(u)$. We now substitute the above obtained bounds in \eqref{uint} according to the following subcases.\\
\\
\emph{Subcase 1.} 
\begin{equation}\label{sbsb1}
A> T^{-\gamma}.
\end{equation}Let us recall that
\begin{equation*}
B_3=(B_1+B_2)T^{\delta_1-\delta_2}(1+T^{-\delta_1/2-\delta_4/2}),
\end{equation*}where
\begin{equation}\notag
B_1=(1+T^{\delta_3})\left(1+R_1^{-1}T^{-\delta_3/2}\right),\,\,B_2=(1+T^{\delta_4})\left(1+R_2^{-1}T^{-\delta_4/2}\right)
\end{equation}and
\begin{equation*}
K_0=AT^{\delta_1-12\delta_2}.
\end{equation*}

Let $\delta_{max}=\max\{\delta_1,\delta_3,\delta_4\}$ and $\delta_{min}=\min\{\delta_1,\delta_3,\delta_4\}$. Suppose first that $\delta_{max}\leq a$, for some $a>0$ to be chosen in a moment. Note that from Lemma \ref{sign} we have $T^{-\delta_3}\ll T^{\delta_1}$ and $T^{-\delta_4}\ll T^{\delta_1}$. Consequently,
\begin{equation}
B_3\ll T^{-\delta_2+3\delta_{max}+\beta},\,\,K_0= AT^{\delta_1-12\delta_2}.
\end{equation}Hence from \eqref{finalbd} we obtain
\begin{equation}\notag
\begin{aligned}
&\int_{0}^{1}H_{k,l}(u)e\left(\frac{\phi_{k,l}(u)}{2\pi}\right)\,du\\
&\ll T^{-\delta_2-\delta_1/2-\delta_4/2}T^{-1/26-\delta_1/14+12\delta_2/13}T^{-\delta_2/13+3\delta_{\max}/13+\beta/13}(A^{-1/13}+A^{-1/14}).
\end{aligned}
\end{equation}Substituting, we see that the contribution of this case towards \eqref{uint} is bounded by
\begin{equation}\label{s1ss2}
\begin{aligned}
&\frac{T^{-\delta_2/2}}{T(\theta_1\theta_2)^{1/4}}\cdot T^{\delta_3/4+\delta_4/4+\delta_1/2}\cdot T^{-\delta_1/2-\delta_4/2} T^{-(1-2(\beta+\gamma))/26}T^{3\delta_{\max}/13}\\
&\ll \frac{T^{-\delta_2/2}}{T(\theta_1\theta_2)^{1/4}}\cdot T^{\delta_{max}/4}T^{-(1-2(\beta+\gamma))/26+3\delta_{max}/13}\\
&\ll \frac{T^{-\delta_2/2}}{T(\theta_1\theta_2)^{1/4}}\cdot T^{-(1-2(\beta+\gamma))/26+a/2}\asymp T^{-1}(\theta_1\theta_2)^{-1/4} (\theta_2/\theta_1)^{1/12}T^{-(1-2(\beta+\gamma))/26+a/2}\\
&=(\theta_2/\theta_1)^{1/6}T^{-1}(\theta_1^{1/6}\theta_2^{1/3})^{-1}T^{-(1-2(\beta+\gamma))/26+a/2}\\
&\ll (\theta_2/\theta_1)^{1/6}T^{-2}T^{-(1-2(\beta+\gamma))/26+a/2},
\end{aligned}
\end{equation}where we have used the fact that $T^{\delta_2}\asymp (\theta_2/\theta_1)^{1/6}$ and the assumption \eqref{uvup}. Now suppose $\delta_{max}>a>0$ and $\delta_{max}-\delta_{min}\geq \delta_{max}/2$.
Then, using the trivial bound \eqref{trivialbd}, we see that the contribution of this case towards \eqref{uint} is bounded by
\begin{equation}\notag
\begin{aligned}
&\frac{T^{-\delta_2/2}}{T(\theta_1\theta_2)^{1/4}}\cdot T^{\delta_3/4+\delta_4/4+\delta_1/2}\min\{T^{-\delta_1}, T^{-\delta_3+\delta_2},T^{-\delta_4}\}\\
&\ll\frac{T^{\epsilon}T^{-\delta_2/2}}{T(\theta_1\theta_2)^{1/4}}\cdot T^{\delta_3/4+\delta_4/4+\delta_1/2}T^{-\delta_{max}}\ll\frac{T^{\epsilon}T^{-\delta_2/2}}{T(\theta_1\theta_2)^{1/4}}\cdot T^{-(\delta_{max}-\delta_{min})/4}  \\
&\ll (\theta_2/\theta_1)^{1/6} T^{-2-a/8+\epsilon}.
\end{aligned}
\end{equation}
\\
Next suppose $\delta_{max}>a$ and $\delta_{max}-\delta_{min}<\delta_{max}/2$. This implies $\delta_{max}/2<\delta_i\leq\delta_{max}$ for each $i=1,3,4$. By the definitions of $T^{\delta_3}, T^{\delta_4}$ \eqref{defdelta34}, this forces $c^2_1T^{\delta_1}\ll 1, c^2_2T^{\delta_1}\ll 1$ and consequently,
\begin{equation}\notag
c^2_1\ll T^{-\delta_1}\ll T^{-a/2}\,\,\, \text{and} \,\,\,\,c^2_2\ll T^{-\delta_1}\ll T^{-a/2}.
\end{equation}In this case we use the trivial bound \ref{trivialbd} to obtain that the contribution of this case towards \eqref{uint} is bounded by
\begin{equation}\notag
\begin{aligned}
&\frac{T^{-\delta_2/2}}{T(\theta_1\theta_2)^{1/4}} T^{\delta_3/4+\delta_4/4+\delta_1/2} \min\{T^{-\delta_1}, T^{-\delta_3+\delta_2},T^{-\delta_4}\}\ll \frac{T^{\epsilon}T^{-\delta_2/2}}{T(\theta_1\theta_2)^{1/4}} \\
&\ll (\theta_2/\theta_1)^{1/6} T^{-2+\epsilon}.
\end{aligned}
\end{equation}
\\
Summarising the above, the contribution of this sub-case towards \eqref{uint} is 
\begin{equation}\notag
\ll (\theta_2/\theta_1)^{1/6}T^{-2+\epsilon}
\end{equation}if $c_1\ll T^{-a/4}, c_2\ll T^{-a/4}$ and is
\begin{equation}\notag
\ll (\theta_2/\theta_1)^{1/6}T^{-2+\epsilon}\left(T^{-(1-2(\beta+\gamma))/26+a/2}+T^{-a/8}\right)
\end{equation}otherwise. Equating the powers of $T$ in the last inequality, we choose $a=4(1-2(\beta+\gamma))/65$. Substituting, we finally obtain the contribution of this sub-case towards \eqref{uint} is 
\begin{equation}\label{c3}
\ll  (\theta_2/\theta_1)^{1/6}T^{-2+\epsilon}
\end{equation}if $\min\{c_1,c_2\}\ll T^{-(1-2(\beta+\gamma))/65}$, and is
\begin{equation}\label{c4}
\ll   (\theta_2/\theta_1)^{1/6} T^{-2-(1-2(\beta+\gamma))/130+\epsilon}
\end{equation}otherwise.\\
\\
\emph{Subcase 2.} 
\begin{equation}\label{sbsb1}
A\leq  T^{-\gamma}.
\end{equation}
In this case we use the trivial estimate \eqref{trivialbd} to see that its contribution towards \eqref{uint} is at most
\begin{equation}\label{c5}
\begin{aligned}
&\ll T^{-1}(\theta_1\theta_2)^{-1/4}T^{\delta_1/2+\delta_3/4+\delta_4/4-\delta_2/2}\min\{T^{-\delta_1}, T^{-\delta_3+\delta_2}, T^{-\delta_4}\}\\
&\ll   T^{-1+\delta_2}(\theta_1\theta_2)^{-1/4} T^{-\delta_2/2}\asymp T^{-1+\epsilon}(\theta_1\theta_2)^{-1/4} (\theta_2/\theta_1)^{1/12}= T^{-1+\epsilon}(\theta_2/\theta_1)^{1/6} (\theta_1^{1/6}\theta_2^{1/3})^{-1}\\
&\ll (\theta_2/\theta_1)^{1/6} T^{-2+\epsilon}.
\end{aligned}
\end{equation}
\subsubsection*{\bf{\textit{Case 2 :}}} $T^{-\beta}\min\{\alpha^{-1},T^{-1}\}\ll R_1$, $T^{-\beta}\min\{\beta^{-1},T^{-1}\}\ll R_2$ and either
\begin{equation}\label{1ste}
T^{-\delta_3/2}\ll \left(\frac{T}{\alpha(u)}+1\right)T^{-1/3+2\epsilon}
\end{equation}or
\begin{equation}\label{2nd}
T^{-\delta_4/2}\ll \left(\frac{T}{\beta(u)}+1\right)T^{-1/3+2\epsilon}.
\end{equation}Note that from the bounds in \eqref{eq1123} and \eqref{bd2}, we have
\begin{equation}\label{degbd1}
I_{R_1}^{\epsilon_1,\epsilon_2}(U,\alpha(u))\ll \frac{T^{\delta_3/4}}{\alpha(u)^{1/2}}+T^{-1/3}\asymp \frac{T^{\delta_1/4+\delta_3/4}}{\theta_1^{1/4}}+T^{-1/3}\ll T^{-1/2+\delta_1/4+\delta_3/4}+T^{-1/3}.
\end{equation}Similarly, we have
\begin{equation}\label{degbd2}
\tilde{I}_{R_2}^{\epsilon'_1,\epsilon'_2}(V,\beta(u))\ll T^{-1/2+\delta_1/4+\delta_4/4}+T^{-1/3}.
\end{equation}

Let $0<a<1/3$ be a quantity to be chosen in a moment.\\
\\
\emph{Subcase 1.} $\delta_1<1/3-a$.\\
\\
WLOG, suppose \eqref{1ste} holds. Then, using $\delta_1<1/3-a$, we obtain
\begin{equation}\label{delta3bd}
\begin{aligned}
T^{-\delta_3}\ll T^{-2/3+4\epsilon}\left(\frac{T}{\alpha(u)}+1\right)^2\ll T^{-2/3+4\epsilon+\delta_1}\left(\frac{T}{\theta_1^{1/2}}+1\right)^2\ll T^{-1/3-a+4\epsilon}.
\end{aligned}
\end{equation}Similar bounds holds in case \eqref{2nd} holds. Hence in any case, if $\delta_{max}=\max\{\delta_1,\delta_3,\delta_4\}$, then we have
\begin{equation}\label{maxsize}
T^{-\delta_{max}}\ll T^{-1/3-a+4\epsilon}.
\end{equation}
Using the above, the bounds \eqref{degbd1}, \eqref{degbd2} for the integrals and then trivially execute the $u$-integral in \eqref{os}  we get
\begin{equation}\label{uint2}
\begin{aligned}
&\frac{1}{T}\int_{0}^{1}\frac{G(u)}{u}I_{R_1}^{\epsilon_1,\epsilon_2}(U,\alpha(u)) \tilde{I}_{R_2}^{\epsilon'_1,\epsilon'_2}(V,\beta(u)) u^{3iT_2}\,du\\
&\ll T^{-1}(T^{-1/2+\delta_1/4+\delta_{max}/4}+T^{-1/3})^2\min\{T^{-\delta_1}, T^{-\delta_3+\delta_2},T^{- \delta_4}\}\\
&\ll T^{-1+\delta_2}(T^{-1+\delta_1/2+\delta_{max}/2}+T^{-2/3})T^{-\delta_{max}}
\end{aligned}
\end{equation}Using $\delta_1<1/3-a, \delta_2\leq \epsilon $ and $T^{-\delta_{max}}\ll T^{-1/3-a+4\epsilon}$ we obtain
\begin{equation}\notag
\frac{1}{T}\int_{0}^{1}\frac{G(u)}{u}I_{R_1}^{\epsilon_1,\epsilon_2}(U,\alpha(u)) \tilde{I}_{R_2}^{\epsilon'_1,\epsilon'_2}(V,\beta(u)) u^{3iT_2}\,du\ll T^{-2-a+5\epsilon}\ll (\theta_2/\theta_1)^{1/6} T^{-2-a+6\epsilon}.
 \end{equation} 
\\
\\
\emph{Subcase 2.} $1/3-a\leq \delta_1\leq 1/3$ and $\delta_2=0$.\\
\\
WLOG, suppose \eqref{1ste} holds. If $c^2_1T^{\delta_1}\gg 1$, then from the definition \eqref{defdelta34} it follows $T^{-\delta_3}\asymp c^2_1T^{\delta_1}$. Substituting this in \eqref{1ste} we obtain
\begin{equation}\notag
\begin{aligned}
c_1T^{\delta_1/2}\ll \left(\frac{T}{\alpha(u)}+1\right)T^{-1/3+2\epsilon} \ll T^{\delta_1/2-1/3+2\epsilon},
\end{aligned}
\end{equation}that is
\begin{equation}\label{b1}
c_1\ll T^{-1/3+2\epsilon}.
\end{equation}
On the other hand if $c^2_1T^{\delta_1}\ll 1$, then by the assumption $\delta \geq 1/3-a$ we get
\begin{equation}\label{b2}
c_1\ll T^{-1/6+a/2}.
\end{equation}So, in any case we have
\begin{equation}\label{c1bd}
c_1\ll  T^{-1/6+a/2}.
\end{equation}in this case. For the integral bound in this case, note that the arguments leading upto \eqref{maxsize} gives
\begin{equation}\notag
T^{-\delta_{max}}\ll T^{-1/3+4\epsilon}
\end{equation}in this case. Consequently, from the bound \eqref{uint2} we get
\begin{equation}\notag
\frac{1}{T}\int_{0}^{1}\frac{G(u)}{u}I_{R_1}^{\epsilon_1,\epsilon_2}(U,\alpha(u)) \tilde{I}_{R_2}^{\epsilon'_1,\epsilon'_2}(V,\beta(u)) u^{3iT_2}\,du\ll T^{-2+4\epsilon}\asymp (\theta_2/\theta_1)^{1/6}T^{-2+4\epsilon}
\end{equation}since $(\theta_2/\theta_1)^{1/6}\asymp T^{-\delta_2}=1$ in this sub-case.
For the case where \eqref{2nd} holds, it is clear that the only change is that the bound \eqref{c1bd} holds for $c_2$ instead of $c_1$. Summarising, we obtain that for $1/3-a\leq \delta_1$, either 
\begin{equation}\label{c1bd1}
c_1\ll  T^{-1/6+a/2},
\end{equation} or 
\begin{equation}\label{c2bd}
c_2\ll  T^{-1/6+a/2}
\end{equation}
 holds, and furthermore
\begin{equation}\label{uintbd}
\frac{1}{T}\int_{0}^{1}\frac{G(u)}{u}I_{R_1}^{\epsilon_1,\epsilon_2}(U,\alpha(u)) \tilde{I}_{R_2}^{\epsilon'_1,\epsilon'_2}(V,\beta(u)) u^{3iT_2}\,du\ll (\theta_2/\theta_1)^{1/6}T^{-2+4\epsilon}
\end{equation}We summarise the last two sub-cases above by choosing $a$ such that $-1/6+a/2=-a$, i.e., $a=1/9$ to get
\begin{equation}\label{c6}
\frac{1}{T}\int_{0}^{1}\frac{G(u)}{u}I_{R_1}^{\epsilon_1,\epsilon_2}(U,\alpha(u)) \tilde{I}_{R_2}^{\epsilon'_1,\epsilon'_2}(V,\beta(u)) u^{3iT_2}\,du\ll (\theta_2/\theta_1)^{1/6} T^{-2+4\epsilon},
\end{equation}if $\min\{c_1,c_2\}\ll T^{-1/9} $, and
\begin{equation}\label{c7}
\frac{1}{T}\int_{0}^{1}\frac{G(u)}{u}I_{R_1}^{\epsilon_1,\epsilon_2}(U,\alpha(u)) \tilde{I}_{R_2}^{\epsilon'_1,\epsilon'_2}(V,\beta(u)) u^{3iT_2}\,du\ll (\theta_2/\theta_1)^{1/6}T^{-2-1/9+6\epsilon}
\end{equation}otherwise.
\\
\\
\emph{Subcase 3.} $\delta_1>1/3$ and $\delta_2=0$.\\
\\
Following the above calculations, It is clear that the conclusion \eqref{c1bd1} and \eqref{c2bd} holds for this case with $a=0$, i.e., either
\begin{equation}\label{c8}
c_1\ll  T^{-1/6}\,\,\,\,\text{or}\,\,\,\,c_2\ll  T^{-1/6}
\end{equation}holds. For the integral bound, since $\delta_1>1/3$, we have $T^{-\delta_{max}}\ll T^{-1/3}$ and consequently the bound \eqref{uint2} gives
\begin{equation}\notag
\frac{1}{T}\int_{0}^{1}\frac{G(u)}{u}I_{R_1}^{\epsilon_1,\epsilon_2}(U,\alpha(u)) \tilde{I}_{R_2}^{\epsilon'_1,\epsilon'_2}(V,\beta(u)) u^{3iT_2}\,du\ll T^{-2+\epsilon}\asymp (\theta_2/\theta_1)^{1/6}T^{-2+\epsilon}
\end{equation}since $(\theta_2/\theta_1)^{1/6}\asymp T^{-\delta_2}=1$ in this sub-case. Hence this case gets absorbed into the previous one.
\begin{remark}
Here we outline the necessary modification required  when one of $U$ or $V$ is zero, i.e, one of $c_1$ or $c_2$ is zero. It enough to consider when $c_1=0$. Note that by definition \eqref{g}, $c_1=0$ implies $\delta_3=0$. Suppose first $\delta_1>2/3-100\epsilon$. Now, using the bounds \eqref{degbd1}, \eqref{degbd2} with $\delta_3=0$ and then trivially execute the $u$-integral we obtain
\begin{equation}\notag
\begin{aligned}
&\frac{1}{T}\int_{0}^{1}\frac{G(u)}{u}I_{R_1}^{\epsilon_1,\epsilon_2}(U,\alpha(u)) \tilde{I}_{R_2}^{\epsilon'_1,\epsilon'_2}(V,\beta(u)) u^{3iT_2}\,du\\
&\ll T^{-1}(T^{-1+\delta_1/2+\delta_4/4}+T^{-2/3})\min\{T^{-\delta_1}, T^{- \delta_4}\}\\
&\ll T^{-1}(T^{-1-\delta_1/4}+T^{-2/3-\delta_1})\ll T^{-2-1/6+100\epsilon}\asymp (\theta_2/\theta_1)^{1/6}T^{-2-1/6+100\epsilon},
\end{aligned}
\end{equation}which gets absorbed into previous upper bounds (\eqref{c7} for example). If $\delta_1\leq 2/3-100\epsilon$, then from \eqref{delta3bd} it is clear that \eqref{1ste} cannot hold since $\delta_3=0$. Hence only \eqref{2nd} can hold in this case and we will be left with only the condition \eqref{c2bd} and the second condition in  \eqref{c8} in the sub-cases above.
\end{remark}
\emph{\textbf{Case 3 :}} $T^{\beta}\min\{\alpha^{-1},T^{-1}\}\ll R_1$ , $T^{\beta}\min\{\beta^{-1},T^{-1}\}\ll R_2$,
\begin{equation}
T^{-\delta_3/2}\gg T^{-1/3+2\epsilon}\left(\frac{T}{\alpha(u)}+1\right),
\end{equation}
\begin{equation}\label{2ndcond}
T^{-\delta_4/2}\gg T^{-1/3+2\epsilon}\left(\frac{T}{\beta(u)}+1\right),
\end{equation}and $\min\{R_1,R_2\}\ll T^{-\beta}$. 
\\
\\
Suppose that $T^{\beta}\min\{\alpha^{-1},T^{-1}\}\ll R_1\ll T^{-\beta}$. Then from part \ref{negsmall} of the lemma we have the restriction
\begin{equation}\label{ures1}
\left|\left(\frac{\pi U}{\alpha(u)}\right)\left(\frac{T_1}{T_2}\right)-1\right|\ll T^{\epsilon}(R_1+(R_1\alpha(u)))^{-1})\ll T^{-\beta+\epsilon}
\end{equation}and the bound
\begin{equation}\label{R1}
I_{R_1}^{\epsilon_1,\epsilon_2}(U,\alpha(u)) \ll \alpha(u)^{-1/2}\asymp \frac{T^{\delta_1/4}}{\theta_1^{1/4}}.
\end{equation}Also, since $T^{-\delta_4}$ satisfies \eqref{2ndcond}, arguing as in \eqref{rootdiff}, we have the bound 
\begin{equation}\label{R2}
\tilde{I}_{R_2}^{\epsilon'_1,\epsilon'_2}(V,\beta(u))\ll \frac{T^{\delta_4/3}}{\beta(u)^{1/2}}\asymp \frac{T^{-\delta_2/2+\delta_1/4+\delta_4/4}}{\theta_2^{1/4}}.
\end{equation}Note that we still have the condition $T^{\delta}\ll T^{\epsilon}$ from \eqref{delta2small}. Using the bounds \eqref{R1},\eqref{R2} and trivially executing the $u$-integral with the restriction \eqref{ures1} we obtain
\begin{equation}\label{case3}
\begin{aligned}
&\frac{1}{T}\int_{0}^{1}\frac{G(u)}{u}I_{R_1}^{\epsilon_1,\epsilon_2}(U,\alpha(u)) \tilde{I}_{R_2}^{\epsilon'_1,\epsilon'_2}(V,\beta(u)) u^{3iT_2}\,du\\
&\ll \frac{T^{-\delta_2/2}}{T(\theta_1\theta_2)^{1/4}}T^{\delta_1/2+\delta_4/4}\min\{T^{-\delta_1}, T^{-\delta_3+\delta_2}, T^{-\delta_4}, T^{-\beta+\epsilon+\delta_2}\}\\
&\ll \frac{T^{-\delta_2/2}}{T(\theta_1\theta_2)^{1/4}}T^{-\beta/4+2\epsilon}\ll (\theta_2/\theta_1)^{1/6}T^{-2-\beta/4+3\epsilon}.
\end{aligned}
\end{equation}The same conclusion holds if we instead assume $T^{\beta}\min\{\beta^{-1},T^{-1}\}\ll R_1\ll T^{-\beta}$. \\
\\
\emph{\textbf{Case 4 :}} $R_1\ll T^{\beta}\min\{\alpha^{-1},T^{-1}\}$ or $R_2\ll T^{\beta}\min\{\beta^{-1},T^{-1}\}$.\\
\\
Suppose first that both $R_1\ll T^{\beta}\min\{\alpha^{-1},T^{-1}\}$ and $R_2\ll T^{\beta}\min\{\beta^{-1},T^{-1}\}$ holds. Then a trivial estimation gives
\begin{equation}\label{c41}
\begin{aligned}
&\frac{1}{T}\int_{0}^{1}\frac{G(u)}{u}I_{R_1}^{\epsilon_1,\epsilon_2}(U,\alpha(u)) \tilde{I}_{R_2}^{\epsilon'_1,\epsilon'_2}(V,\beta(u)) u^{3iT_2}\,du\ll \frac{R_1R_2}{T}T^{-\delta_1}\\
&\ll \frac{T^{2\beta-1}}{\alpha(u)\beta(u)}T^{-\delta_1}\asymp \frac{T^{-1+2\beta +\delta_1-\delta_2}}{\theta_1\theta_2}T^{-\delta_1}\ll T^{-\delta_2} T^{-3+\beta}\\
&\asymp  (\theta_2/\theta_1)^{1/6}T^{-3+2\beta}.
\end{aligned}
\end{equation}Now suppose $R_1\ll T^{\beta}\min\{\alpha^{-1},T^{-1}\}$ and  $R_2\gg T^{\beta}\min\{\beta^{-1},T^{-1}\}$. Then from the bounds in \eqref{eq1123} and \eqref{bd2}, we have
\begin{equation}\notag
\tilde{I}_{R_2}^{\epsilon'_1,\epsilon'_2}(V,\beta(u))\ll \frac{T^{\delta_4/4}}{\beta (u)^{1/2}}+T^{-1/3}\ll  T^{-1/2+\delta_1/4+\delta_4/4}+T^{-1/3}.
\end{equation}Using this and the trivial bound $R_1$ for $I_{R_1}^{\epsilon_1,\epsilon_2}(U,\alpha(u))$ we obtain
\begin{equation}\label{c42}
\begin{aligned}
&\frac{1}{T}\int_{0}^{1}\frac{G(u)}{u}I_{R_1}^{\epsilon_1,\epsilon_2}(U,\alpha(u)) \tilde{I}_{R_2}^{\epsilon'_1,\epsilon'_2}(V,\beta(u)) u^{3iT_2}\,du\\
&\ll \frac{R_1}{T}(T^{-1/2+\delta_1/4+\delta_4/4}+T^{-1/3})\min\{T^{-\delta_1}, T^{-\delta_4}\}\\
&\ll \frac{T^{-1+\beta}}{\alpha(u)}(T^{-1/2+\delta_1/4+\delta_4/4}+T^{-1/3})\min\{T^{-\delta_1}, T^{-\delta_4}\}\\
&\ll T^{-\delta_2+\epsilon}T^{-2+\beta}(T^{-1/2+3\delta_1/4+\delta_4/4}+T^{-1/3+\delta_1/2})\min\{T^{-\delta_1}, T^{-\delta_4}\}\\
&\ll T^{-\delta_2+\epsilon}T^{-2-1/3+\beta}\asymp (\theta_2/\theta_1)^{1/6}T^{-2-1/3+\beta+\epsilon},
\end{aligned}
\end{equation}where we have used the inequality $T/\alpha(u)\ll T^{\delta_1/2-\delta_2+\epsilon}$ from  \eqref{delta2aux}. Finally suppose $R_1\gg T^{\beta}\min\{\alpha^{-1},T^{-1}\}$ and  $R_2\ll T^{\beta}\min\{\beta^{-1},T^{-1}\}$, then proceeding similarly as above we get
\begin{equation}\label{c43}
\begin{aligned}
&\frac{1}{T}\int_{0}^{1}\frac{G(u)}{u}I_{R_1}^{\epsilon_1,\epsilon_2}(U,\alpha(u)) \tilde{I}_{R_2}^{\epsilon'_1,\epsilon'_2}(V,\beta(u)) u^{3iT_2}\,du\\
&\ll\frac{T^{-1+\beta}}{\beta(u)}(T^{-1/2+\delta_1/4+\delta_3/4}+T^{-1/3})\min\{T^{-\delta_1}, T^{-\delta_3+\delta_2}\}\\
&\ll T^{-2+\beta+\epsilon}T^{\delta_1/2}(T^{-1/2+\delta_1/4+\delta_3/4}+T^{-1/3})\min\{T^{-\delta_1}, T^{-\delta_3+\delta_2}\}\\
&\ll T^{-2+\beta+\epsilon}(T^{-1/2+3\delta_1/4+\delta_3/4}+T^{-1/3+\delta_1/2})\min\{T^{-\delta_1}, T^{-\delta_3+\epsilon}\}\\
&\ll T^{-2-1/3+\beta+2\epsilon}\ll (\theta_2/\theta_1)^{1/6}T^{-2-1/3+\beta+3\epsilon},
\end{aligned}
\end{equation}where we have used the inequalities $\beta(u)\asymp T^{-\delta_1/2}(T^{\delta_2}\theta_2^{1/2})\gg T^{-\delta_1/2}T^{1-\epsilon}$ and $T^{\delta_2}\ll T^{\epsilon}$ following from \eqref{delta2small}.

We compile the all the above cases to reach at our conclusion \eqref{mainest}. Let
\begin{equation}\notag
 \alpha=1-2(\beta+\gamma),\,\,A=\left|\frac{\pi |U|}{\theta_1^{1/2}}-1\right|.
 \end{equation} 
Combining the estimates from  \eqref{c3}, \eqref{c4}, \eqref{c5}, \eqref{c6}, \eqref{c7}, \eqref{c8}, \eqref{case3}, \eqref{c41}, \eqref{c42} ,\eqref{c43} and substituting in \eqref{K4}  we obtain
\begin{equation}\label{compile}
\begin{aligned}
T^{-100\epsilon}(\theta_1/\theta_2)^{1/6}\mathcal{K}_4(\theta_1,\theta_2, U,V)&\ll T^{3}\Big(T^{-2-\alpha/130}+T^{-2}\delta_{\min\{c_1,c_2\}\ll T^{-\alpha/65}}\\
&+T^{-2}\delta_{A\ll T^{-\gamma}}+T^{-2}\delta_{\min\{c_1,c_2\}\ll T^{-1/9}}\\
&+T^{-2-\beta/4}+T^{-2-1/3+\beta}+T^{-3+2\beta}\Big)
\end{aligned}
\end{equation}Comparing the exponents of the first, third and the fifth term, we choose $\beta,\gamma$ such that
\begin{equation}\notag
\frac{\alpha}{130}=\gamma=\frac{\beta}{4}.
\end{equation}Solving we obtain
\begin{equation}\notag
\gamma=\frac{\beta}{4}=\frac{1}{140}.
\end{equation}Note that with this choice of $\beta$ and $\gamma$, the first, third and the fifth term of \eqref{compile} dominates the others. This completes the proof of Theorem \ref{average}.
\section*{Appendix: An average of Bessel functions with non-linear twist}
For this section we fix a smooth partition of unity
\begin{equation}\notag
\sum_{R}F\left(\frac{r}{R}\right)=1\,\,\text{for}\,\,r\in (0,\infty)
\end{equation}consisting of a sequence of numbers $R\in \mathbb{R}_{> 0}$ and a smooth function $F$ supported on $[1,2]$ and satisfying $F^{(j)}(x)\ll_j 1$.
\begin{lemma}\label{11.1}
Let $\alpha >0, U\in \mathbb{R}$ and $\epsilon>0$ be small enough. Let $\mathcal{U}$ be a smooth weight function compactly supported on $\mathbb{R}_{>0}$ and satisfying $\mathcal{U}^{(j)}(x)\ll_j 1$. Define
\begin{equation}\label{mainI}
I(U,\alpha):=\int_{\mathbb{R}}\mathcal{U}(x)x^{iT_2}\mathcal{K}_{iT_1}(\alpha x)e(Ux)\,dx,
\end{equation}where $\mathcal{K}_{iT_1}$ is one of the Bessel function $\tilde{K}_{iT_1}$ or $J^{-}_{iT_1}$. Then there exists smooth weight functions $\mathcal{W}$ (depending on $\mathcal{U}$) compactly supported on $\mathbb{R}_{>0}$, such that
\begin{equation}\label{eq1121}
I(U,\alpha)\sim  \frac{e\left(\frac{-T_2}{2\pi}\right)}{T^{1/2}}\sum_{\epsilon_1,\epsilon_2=\pm 1}\,\,\sum_{R}I_{R}^{(\epsilon_1,\epsilon_2)}(U,\alpha),
\end{equation}where
\begin{equation}\label{Inepsilon}
I_{R}^{(\epsilon_1,\epsilon_2)}(U,\alpha)=\int_{\mathbb{R}}v^{-1}\omega_R(\epsilon_1\alpha,v) e\left(\frac{-T_2\log\phi_R(\epsilon_1\alpha,v)+\epsilon_2T_1\log v}{2\pi}\right)\,dv
\end{equation}where 
\begin{equation}\label{eq1122}
\phi_R(\alpha,v)=\frac{-2\pi}{T_2}\left(\frac{\alpha}{2\pi}\left(Rv\pm\frac{1}{Rv}\right)+U\right).
\end{equation}	and
\begin{equation}\label{11.24}
\omega_R(\alpha,v) =F(v-R^{-1})\mathcal{W}\left(\phi_R(\alpha,v)\right).
\end{equation}
The `$\pm$' in \eqref{eq1122} is $+$ for $\mathcal{K}_{iT_1}=\tilde{K}_{iT_1}$ and is $-$ for $\mathcal{K}_{iT_1}=J^{-}_{iT_1}$.\\
For $R\gg T^{\epsilon}\min\{\alpha^{-1}, T^{-1}\}$, $I_R^{(\epsilon_1,\epsilon_2)}$ is negligibly small unless
\begin{equation}\label{Rsize}
\begin{aligned}
(R+1)\alpha\asymp T,\,\,\,\text{if}\,\,\,\mathcal{K}_{iT_1}=\tilde{K}_{iT_1},\\
R\alpha\asymp T,\,\,\,\text{if}\,\,\,\mathcal{K}_{iT_1}=J^{-}_{iT_1}.
\end{aligned}
\end{equation}Let $z_1(\alpha),z_2(\alpha)$ be the the two roots of the corresponding phase function derivative
\begin{equation}
\frac{\partial}{\partial v}\left(-T_2\log(\phi_R(\alpha,v))+\epsilon_2T_1\log v\right)=0.
\end{equation}Suppose
\begin{equation}
z_1(\epsilon_1\alpha)-z_2(\epsilon_1\alpha)\asymp R^{-1}T^{-\delta/2},
\end{equation}for some $\delta\in\mathbb{R}$.
We have the following asymptotics for $I_R$  :
\begin{enumerate}
\item\label{asymptotic} If
\begin{equation}\label{alphasize}
R\gg T^{3\epsilon}\min\{\alpha^{-1}, T^{-1}\}
\end{equation}
and
\begin{equation}\label{eq11.27}
z_1(\epsilon_1\alpha)-z_2(\epsilon_1\alpha)\asymp R^{-1}T^{-\delta/2}\gg T^{2\epsilon}\left(\frac{R^3\alpha}{(R+1)^2}\right)^{-1/3}
\end{equation}then $I_{R}^{(\epsilon_1,\epsilon_2)}(U,\alpha)$ is negligibly small unless the roots $z_1(\epsilon_1\alpha),z_2(\epsilon_1\alpha)$ are real, in which case,
\begin{equation}\label{eq1123}
\begin{aligned}
&I_{R}^{(\epsilon_1,\epsilon_2)}(U,\alpha)\\
&\sim \frac{T^{\delta/4}}{|\alpha|^{1/2}}\sum_{k=1,2}\frac{\omega_R(\epsilon_1\alpha,z_k(\epsilon_1\alpha))}{(RT^{\delta/2}|z_1(\epsilon_1\alpha)-z_2(\epsilon_1\alpha)|)^{1/2}} \left(\phi_{R}(\epsilon_1\alpha,z_k(\epsilon_1\alpha))\right)^{-iT_2}\left(z_k(\epsilon_1\alpha)\right)^{i\epsilon_2 T_1}.
\end{aligned}
\end{equation}The weight function satisfy
\begin{equation}\label{eq1126}
|\alpha|^j\frac{d^j }{d\alpha^j}\frac{\omega_R(\epsilon_1\alpha,z_k(\epsilon_1\alpha))}{(RT^{\delta/2}|z_1(\epsilon_1\alpha)-z_2(\epsilon_1\alpha)|)^{1/2}} \ll_{j,R,\epsilon} A^j,\,\,j\geq 0,
\end{equation}where
\begin{equation}\notag
A:= (1+R^{-1}T^{\delta/2})(1+T^{\delta}).
\end{equation}
\item\label{thirdderbd}
If $R\gg T^{3\epsilon}\min\{\alpha^{-1}, T^{-1}\}$ and
\begin{equation}
z_1(\epsilon_1\alpha)-z_2(\epsilon_1\alpha)\asymp R^{-1}T^{-\delta/2}\ll T^{2\epsilon}\left(\frac{R^3\alpha}{(R+1)^2}\right)^{-1/3},
\end{equation}then
\begin{equation}\label{bd2}
I_{R}^{(\epsilon_1,\epsilon_2)}(U,\alpha)\ll ((R+1)\alpha)^{-1/3}.
\end{equation}
\item\label{negsmall} If $T^{\epsilon}\min\{\alpha^{-1}, T^{-1}\}\ll R\ll T^{-\epsilon}$, then $I_{R}^{(\epsilon_1,\epsilon_2)}(U,\alpha)$ is negligibly small unless
\begin{equation}
\begin{aligned}
\left|\left(\frac{\pi U}{\alpha}\right)\left(\frac{T_1}{T_2}\right)-1\right|\ll T^{\epsilon}(R+(R\alpha))^{-1}),\,\,\text{if}\,\,\mathcal{K}_{iT_1}=\tilde{K}_{iT_1},\\
\left|\left(\frac{\pi U}{\alpha}\right)-1\right|\ll T^{\epsilon}(R+(R\alpha))^{-1}),\,\,\,\text{if}\,\,\,\mathcal{K}_{iT_1}=J^{-}_{iT_1}.
\end{aligned}
\end{equation}in which case
\begin{equation}\label{bd3}
I_{R}^{(\epsilon_1,\epsilon_2)}(U,\alpha)\ll \alpha^{-1/2}.
\end{equation}
\end{enumerate}
\end{lemma}
\begin{remark}
The $``\sim"$ symbol in \eqref{eq1123} means equal upto lower order terms. The lower order terms can be calculated explicitly by following the proof and can be handled similarly as the main term. These have smaller contributions and their details are suppressed for the sake of simplicity.
\end{remark}
\begin{proof}
Let us begin with the proof of $\eqref{eq1121}$ for $\mathcal{K}_{iT_1}=\tilde{K}_{iT_1}$. The modifications required for $J^{-}_{iT_1}$ will be pointed out at appropriate places in the course of the proof below. Consider the integral representation
\begin{equation}\notag
\begin{aligned}
\tilde{K}_{iT_1}(\alpha x)&=\frac{1}{2}\sum_{\pm}\int_{-\infty}^{\infty}e\left(\frac{\pm 2\alpha x\sinh u+T_1u}{2\pi}\right)du\\
&=\frac{1}{2}\sum_{\epsilon_1,\epsilon_2\in\{\pm 1\}^2}\int_{0}^{\infty}e\left(\frac{2\epsilon_1\alpha x\sinh u+\epsilon_2T_1u}{2\pi}\right)du,
\end{aligned}
\end{equation}for $\alpha, x>0$. Changing variable $e^u\mapsto v$, we obtain
\begin{equation}\notag
\tilde{K}_{iT_1}(\alpha x)=\frac{1}{2}\sum_{\epsilon_1,\epsilon_2\in\{\pm 1\}^2}\int_{1}^{\infty}e\left(\frac{\epsilon_1\alpha x(v-v^{-1})+\epsilon_2T_1\log v}{2\pi}\right)\frac{dv}{v}.
\end{equation} Inserting the dyadic partition $U((v-1)/R), R>0$ of the interval $(1,\infty)$ into the last integral we obtain
\begin{equation}\notag
\tilde{K}_{iT_1}(\alpha x)=\frac{1}{2}\sum_{\epsilon_1,\epsilon_2\in\{\pm 1\}^2}\sum_{R>0}\int_{\mathbb{R}}U((v-1)/R) e\left(\frac{\epsilon_1\alpha x(v-v^{-1})+\epsilon_2T_1\log v}{2\pi}\right)\frac{dv}{v}.
\end{equation}Changing variable $v\mapsto Rv$ we obtain
\begin{equation}\label{besselsim}
\begin{aligned}
&\tilde{K}_{iT_1}(\alpha x)\\
&=\frac{1}{2}\sum_{\epsilon_1,\epsilon_2\in\{\pm 1\}^2}\sum_{R>0}e(\epsilon_2T_1\log(R))\int_{\mathbb{R}}U(v-R^{-1})e\left(\frac{\epsilon_1\alpha x (Rv-R^{-1}v^{-1})+\epsilon_2T_1\log v}{2\pi}\right)\frac{dv}{v}.
\end{aligned}
\end{equation}Note that $U(v-1/R)$  has support inside $[1+R^{-1},2+R^{-1}]$ and satisfies  $\frac{\partial^j}{\partial v^j}U(v-R^{-1})\ll_j 1$. Note that the phase function in \eqref{besselsim} has the first derivative
\begin{equation}\label{fd}
(2\pi)^{-1}\left(R\epsilon_1\alpha x(1+(Rv)^{-2})+\epsilon_2T_1v^{-1}\right)
\end{equation}For $v\in \text{supp}\,U(v-R^{-1})$ we have 
\begin{equation}\notag
R\epsilon_1\alpha x(1+(Rv)^{-2})\asymp R\alpha\,\,\,\,\text{and}\,\,\,\,\epsilon_2T_1v^{-1}\asymp T(1+R^{-1})^{-1}.
\end{equation}So if we assume $RT\gg T^{\epsilon}$, it follows that the integral \eqref{besselsim} is negligibly small unless
\begin{equation}\label{nsize}
(R+1)\alpha\asymp T.
\end{equation}This completes the proof of the claim \eqref{Rsize}.
\begin{remark}
In the case $\mathcal{K}_{iT_1}=J^{-}_{iT_1}$, \eqref{fd} becomes $\left(R\epsilon_1\alpha x(1-(Rv)^{-2})+\epsilon_2T_1v^{-1}\right)$. So we must have $ R^2\alpha/(R+1)\asymp R\epsilon_1\alpha x(1-(Rv)^{-2})\asymp  \epsilon_2T_1v^{-1}\asymp RT/(R+1)$, that is, $R\alpha\asymp T$, for non-negligible contribution.
\end{remark}
Substituting \eqref{besselsim} in \eqref{mainI}, we obtain
\begin{equation}\label{xint}
\begin{aligned}
I(U, \alpha)=\frac{1}{2}\sum_{\epsilon_1,\epsilon_2\in\{\pm 1\}^2}&\sum_{Rw}e(\epsilon_2T_1\log(R))\int_{\mathbb{R}}U(v-R^{-1})e\left(\frac{\epsilon_2T_1\log v}{2\pi}\right)\\
&\,\,\,\,\,\,\,\,\int_{\mathbb{R}}\mathcal{U}(x)x^{iT_2}e\left(\frac{\epsilon_1\alpha (Rv-R^{-1}v^{-1})x}{2\pi}+Ux\right)\,dx\,\,\frac{dv}{v}.
\end{aligned}
\end{equation} Let 
\begin{equation}\label{defphiR}
\phi_R(\alpha,v):=-2\pi T_2^{-1}\left(\frac{\alpha(Rv-R^{-1}v^{-1})}{2\pi}+U\right),
\end{equation}then by a simple stationary phase analysis, the last $x$-integral is essentially
\begin{equation}\notag
\begin{aligned}
&\int_{\mathbb{R}}\mathcal{U}(x)x^{iT_2}e\left(-\frac{T_2}{2\pi}\cdot\phi_R(\epsilon_1\alpha,v)x\right)\,dx\\
&\sim\frac{(2\pi)^{1/2}}{|T_2|^{1/2}}\cdot (\phi_R(\epsilon_1\alpha,v))^{-1}\mathcal{U}((\phi_R(\epsilon_1\alpha,v))^{-1})\cdot e\left(\frac{-T_2\log(\phi_R(\epsilon_1\alpha,v))-T_2}{2\pi}\right).
\end{aligned}
\end{equation}Substituting the last approximation of the $x$ integral into \eqref{xint}, we obtain
\begin{equation}\notag
 I(U,\alpha)\sim  \frac{e\left(\frac{-T_2}{2\pi}\right)}{T^{1/2}}\sum_{\epsilon_1,\epsilon_2=\pm 1}\sum_{R}e(\epsilon_2T_1\log(R))I_{R}^{(\epsilon_1,\epsilon_2)}(U,\alpha),
\end{equation}where
\begin{equation}\notag
I_{R}^{(\epsilon_1,\epsilon_2)}(U,\alpha)=\int_{\mathbb{R}}v^{-1}\omega_R(\epsilon_1\alpha,v) e\left(\frac{-T_2\log\phi_R(\epsilon_1\alpha,v)+\epsilon_2T_1\log v}{2\pi}\right)\,dv,
\end{equation}where
\begin{equation}\label{defomegaR}
\omega_R(\alpha,v)=U(v-R^{-1})\mathcal{W}\left(\phi_R(\alpha,v)\right),\,\,\,\,\mathcal{W}(y)=\frac{2\pi T^{1/2}}{|T_2|^{1/2}}\cdot y^{-1}\mathcal{U}(y^{-1}).
\end{equation}To complete the proof of part \eqref{eq1121}, it remains to prove $U\ll T$. Note that from the definition \eqref{defphiR}
\begin{equation}\notag
U/T\ll |\phi_r(\alpha,v)|+R\alpha/T\ll |\phi_r(\alpha,v)|+1\ll 1,
\end{equation}for $\phi_r(\alpha,v)\in \text{supp}\,\mathcal{W}$. This completes the proof upto \eqref{eq1121}
\\

Let us proceed for the proof of \ref{asymptotic} \eqref{eq1123}. We have
\begin{equation*}
\begin{aligned}
&\frac{\partial}{\partial v}\left(-T_2\log(\phi_R(\epsilon_1\alpha,v))+\epsilon_2T_1\log v\right)\\
&=-\frac{T_2\epsilon_1\alpha\left(v^2+R^{-2}\right)}{v(\epsilon_1\alpha\left(v^2-R^{-2}\right)+R^{-1}2\pi U v)}+\frac{\epsilon_2T_1}{v}\\
&= \frac{\epsilon_1\alpha(\epsilon_2T_1-T_2)v^2+R^{-1}(2\pi U)(\epsilon_2T_1)v-\epsilon_1\alpha(u)R^{-2}(\epsilon_2T_1+T_1)}{v(\epsilon_1\alpha\left(v^2-R^{-2}\right)+R^{-1}2\pi U v)}
\end{aligned}
\end{equation*}Hence $z_1(\epsilon_1\alpha)$ and $ z_2(\epsilon_1\alpha)$ are the roots of the quadratic equation
\begin{equation}\label{quadratic}
v^2+R^{-1}\left(\frac{\pi U}{\epsilon_1\alpha}\right)\left(\frac{2\epsilon_2T_1}{(\epsilon_2T_1-T_2)}\right)v-R^{-2}\left(\frac{\epsilon_2T_1+T_2}{\epsilon_2T_1-T_2}\right),
\end{equation}and we have
\begin{equation}\label{psider}
\begin{aligned}
\frac{\partial}{\partial v}\left(-T_2\log(\phi_R(\epsilon_1\alpha,v))+\epsilon_2T_1\log v\right)&=\frac{R\alpha(\epsilon_2T_1-T_2)(v-z_1(\epsilon_1\alpha))(v-z_2(\epsilon_1\alpha))}{T_2v^2\phi_R(\epsilon_1\alpha,v))}
\end{aligned}
\end{equation}Denote
\begin{equation}\notag
A_0:=R^3\alpha/(R+1)^2\gg T^{\epsilon}.
\end{equation}Note that by our assumption \eqref{eq11.27}
\begin{equation}\label{rootgap}
z_1-z_2\gg T^{2\epsilon}A_0^{-1/3}.
\end{equation}
Let $\psi(v):=-T_2\log(\phi_R(\epsilon_1\alpha,v))+\epsilon_2T_1\log v$. Suppose the roots $z_1,z_2$ are not real. Let $w$ be a smooth function such that $w(t)=1, t\in [-1/2,1/2]$ and $\text{supp}(w)\subseteq [-1,1]$. Denote $A:=|z_1(\epsilon_1\alpha)-z_2(\epsilon_1\alpha)|\gg T^{2\epsilon }A_0^{-1/3} $. We can write
\begin{equation}\label{imaginary}
I_{R}^{(\epsilon_1,\epsilon_2)}(U,\alpha)=\int_{\mathbb{R}}G_1(v)e(\psi(v)/2\pi)\,dv+\int_{\mathbb{R}}G_2(v)e(\psi(v)/2\pi)\,dv
\end{equation}where $$G_1(v):=w(A^{-1}(v-\Re(z_1)))v^{-1}\omega^{(\epsilon_1,\epsilon_2)}_R(\alpha,v),\,\,\,\,G_2(v)=(1-w(A^{-1}(v-\Re(z_1))))v^{-1}\omega^{(\epsilon_1,\epsilon_2)}_R(\alpha,v).$$ Consider the integral with $G_1$. For $v$ in support of $G_1$, from \eqref{psider} we obtain
\begin{equation}\notag
\psi'(v)\gg A_0\left(|v-\Re(z_1)|^2+|z_1-z_2|^2\right)\gg A_0A^2.
\end{equation} Furthermore, using the simple inequalities
\begin{equation}\label{phider}
\frac{\partial^j }{\partial v^j}\phi_R(\alpha, v)\ll_j \frac{\alpha}{T}\Bigg(\frac{R^j}{(R+1)^{j-1}}\Bigg)\asymp (R/(R+1))^j,\,j\geq 1,\,\,\frac{\partial^j }{\partial v^j}(1/v^2)\ll (R/(R+1))^{j+2},\,j\geq 0
\end{equation} it is easy to show that for $v$ in support of $\omega^{(\epsilon_1,\epsilon_2)}_R(\alpha,v)$,
\begin{equation}\label{denoder}
\frac{\partial ^j}{\partial v^j}(1/v^2\phi_R(v,\alpha))\ll _j (R/(R+1))^{j+2},\,j\geq 0,
\end{equation}so that using the above and \eqref{psider}, we obtain
\begin{equation}\notag
\psi''(v)\ll A_0(A+A^2),\,\,\psi^{(j)}(v)\ll_j A_0(1+A+A^2),\,j\geq 3.
\end{equation}On the other hand, using \eqref{phider} and \eqref{defomegaR}, we obtain that
 \begin{equation}\notag
G_1^{(j)}(v)\ll_{j} (1+A^{-1})^j.
 \end{equation}Applying Lemma \ref{s1} with the parameters $(X,U^{-1})=(1, 1+A^{-1})$, $W=A_0A^2, Y=\min\{1,A\}^2A_0(A+A^2), Q=\min\{1,A\}$, we obtain
 \begin{equation}\label{G1}
I_{n}^{(\epsilon_1,\epsilon_2)}(U,\alpha,\delta)\ll_{A} \left(\frac{A_0 A^2}{(A_0(A+A^2))^{1/2}}\right)^{-A}+\left(\frac{A_0 A^2}{1+A^{-1}}\right)^{-A}.
 \end{equation}Using the fact that $A\gg T^{2\epsilon}A_0^{-1/3}$ and $A_0\gg T^{\epsilon}$, it can be easily verified that each term inside the parenthesis above is $\gg T^{\epsilon}$. Hence we conclude that the first integral in \eqref{imaginary} is negligibly small. For the second integral in \eqref{imaginary}, we further introduce dyadic partition and insert localising factors $F(T^{\delta}(v-\Re(z))$, where $T^{-\delta}\gg A$. Each dyadic part is of the form
 \begin{equation}\label{G2dyadic}
\int_{\mathbb{R}}F(T^{\delta}(v-\Re(z))G_2(v)e(\psi(v)/2\pi)\,dv. 
 \end{equation}Then for $v$ in the support of $F(T^{\delta}(v-\Re(z))G_2(v)$, we have
 \begin{equation}\notag
\psi'(v)\gg  A_0\left(|v-\Re(z_1)|^2+|z_1-z_2|^2\right)\gg A_0 T^{-2\delta},
 \end{equation}and arguing as earlier we have
 \begin{equation}\notag
\psi''(v)\ll A_0(T^{-\delta}+T^{-2\delta}),\,\,\psi^{(j)}(v)\ll_j A_0(1+T^{-\delta}+T^{-2\delta}),\,j\geq 3.
 \end{equation}Now arguing similarly as in \eqref{G1} with $T^{-\delta}$ in place of $A$, and using the fact that $T^{-\delta}\gg A$, we can conclude \eqref{G2dyadic} is negligibly small.\\
 
Hence, for rest of the calculation we assume that the roots $z_1,z_2$ are real. 
Fix a smooth function $w$ such that $w(t)=1, t\in [-1/2,1/2]$ and $\text{supp}(w)\subseteq [-1,1]$. We divide the range of integration of  $I_{n}^{(\epsilon_1,\epsilon_2)}(U,\alpha)$ into two pieces (localising around the two roots)
\begin{equation}\notag
I_{R}^{(\epsilon_1,\epsilon_2)}(U,\alpha)=\int_{v\in\mathbb{R}}v^{-1}\omega^{(\epsilon_1,\epsilon_2)}_R(\alpha,v)e(\psi(v)/2\pi)\,dv=I_1+I_2,
\end{equation}where
\begin{equation}\label{I1def}
I_1:=\int_{\mathbb{R}}(1-w(T^{-\epsilon}A_0^{1/3}(v-z_1)))v^{-1}\omega^{(\epsilon_1,\epsilon_2)}_R(\alpha,v)e(\psi(v)/2\pi) \,dv,
\end{equation}and
\begin{equation}\notag
I_2:=\int_{\mathbb{R}}w(T^{-\epsilon}A_0^{1/3}(v-z_1))v^{-1}\omega^{(\epsilon_1,\epsilon_2)}_R(\alpha,v)e(\psi(v)/2\pi)\,dv.
\end{equation}Consider the $I_1$ part first. Note that $v\in \text{supp}(1-w(T^{-\epsilon}A_0^{1/3}(v-z_1))) $ implies $v-z_1\gg T^{\epsilon}A_0^{-1/3}$. We fix the size of $v-z_1$ by introducing dyadic partition of unity and inserting localising factors $F(T^{\delta}(v-z_1))$, where
\begin{equation}
 T^{\epsilon}A_0^{-1/3}\ll T^{-\delta}
\end{equation}Note that since $\text{supp}\,\omega^{(\epsilon_1,\epsilon_2)}_R(\alpha,v)\subseteq [1+R^{-1}, 1+2R^{-1}]$ and $|z_1|\ll 1+R^{-1}$ which follows from \eqref{quadratic}, we can also assume
\begin{equation}\label{deltasize}
T^{\epsilon}A_0^{-1/3}\ll T^{-\delta}\ll 1+R^{-1} .
\end{equation}
 Each dyadic part is of the form
\begin{equation}\notag
\begin{aligned}
I_1(\delta):=\int_{v\in\mathbb{R}}F(T^{\delta}(v-z_1))&(1-w(T^{-\epsilon}A_0^{1/3}(v-z_1)))v^{-1}\omega^{(\epsilon_1,\epsilon_2)}_R(\alpha,v) \\
&e\left(\frac{-T_2\log\phi_R(\epsilon_1\alpha,v)+\epsilon_2T_1\log v}{2\pi}\right)\,dv.
\end{aligned}
\end{equation}
 Changing variable $T^{\delta}(v-z_1)\mapsto v_1$, we obtain
\begin{equation}\label{I1dyadic}
I_1(\delta)=T^{-\delta}\int_{v_1}G(v_1)e(\phi(v_1)/2\pi)\,dv_1,
\end{equation}where 
\begin{equation}\notag
G(v_1)=F(v_1)(1-w(T^{-\epsilon}A_0^{1/3}T^{-\delta}v_1))(T^{-\delta}v_1+z_1)^{-1}\omega^{(\epsilon_1,\epsilon_2)}_R(\alpha, T^{-\delta}v_1+z_1) 
\end{equation}and $\phi(v_1):=\psi(T^{-\delta}v_1+z_1)$ which using \eqref{psider} satisfy
\begin{equation}\label{I1der}
\begin{aligned}
\phi'(v_1)=\frac{R\alpha(\epsilon_2T_1-T_2)T^{-3\delta}v_1(v_1+T^{\delta}(z_1-z_2))}{T_2(T^{-\delta}v_1+z_1)^2\phi_n(\epsilon_1\alpha, T^{-\delta}v_1+z_1)}
\end{aligned}
\end{equation}
\begin{lemma}\label{•}
$I_1(\delta)$ is negligibly small unless $T^{\delta}(z_1-z_2)\ll 1$ for some large enough implied constant.
\end{lemma}	
\begin{proof}
Suppose $T^{\delta}(z_1-z_2)\gg 1$ for some large enough implied constant. Then from \eqref{I1der}, for $v_1$ in the support of $G(v_1)$,
\begin{equation}\notag
\phi'(v_1)\gg A_0 T^{-3\delta}(T^{\delta}(z_1-z_2)),
\end{equation}
\begin{equation}\notag
\phi^{(j)}(v_1)\ll_j A_0 T^{-3\delta}(T^{\delta}(z_1-z_2)),\,j\geq 2.
\end{equation}where we have used \eqref{denoder} and \eqref{deltasize}. Furthermore,
\begin{equation}\notag
G^{(j)}(v_1)\ll_j  (1+T^{-\delta}+A_0^{1/3}T^{-\delta})^{j}.
\end{equation}Applying Lemma \ref{s1} with $(X, U^{-1})= (1, 1+T^{-\delta}+T^{-\delta}A_0^{1/3}), (Y,Q^{-1})=(A_0T^{-3\delta}(T^{\delta}(z_1-z_2)), 1)$ and $W=A_0T^{-3\delta}(T^{\delta}(z_1-z_2))$, we obtain
\begin{equation}\label{rootbded}
\begin{aligned}
I_1(\delta)&\ll_A \left(\frac{A_0 T^{-3\delta}(T^{\delta}(z_1-z_2))}{(A_0 T^{-3\delta}(T^{\delta}(z_1-z_2)))^{1/2}}\right)^{-A}+\left(\frac{A_0 T^{-3\delta}(T^{\delta}(z_1-z_2))}{1+T^{-\delta}+T^{-\delta}A_0^{1/3}}\right)^{-A}\\
&\ll (A_0T^{-3\delta})^{-A/2}+\left(\frac{A_0 T^{-3\delta}}{1+T^{-\delta}+T^{-\delta}A_0^{1/3}}\right)^{-A}
\end{aligned}
\end{equation}Using $T^{-\delta}\gg T^{\epsilon}A_0^{-1/3}$ and $A_0\gg T^{\epsilon}$, it can be easily shown that the term inside each parenthesis above is $\gg T^{\epsilon}$. The lemma follows.
\end{proof}

The next lemma localises the integral $I_1(\delta)$ around the root $-T^{\delta}(z_1-z_2)$ of its phase function.
\begin{lemma}\label{I1local}
Suppose $w$ is a smooth function such that $w(t)=1, t\in [-1/2,1/2]$ and $\text{supp}(w)\subseteq [-1,1]$. Let $A:= A_0^{1/2}T^{-3\delta/2-\epsilon} (\gg T^{\epsilon/2})$ and
\begin{equation}\notag
H(v_1):=(1-w(A(v_1+T^{\delta}(z_1-z_2))))G(v_1).
\end{equation}Then
\begin{equation}\notag
I:=\int_{v_1}H(v_1)e(\phi(v_1)/2\pi)\,dv_1\ll_{A} T^{-A}.
\end{equation}
\end{lemma}	
\begin{proof}
Using \eqref{denoder} ,\eqref{deltasize} and \eqref{I1der}, it is easy to see that for $v_1$ in the support of $1-w(A(v+T^{\delta}(z_1-z_2)))$, we have
\begin{equation}\notag
\phi'(v_1)\gg T^{\epsilon}A_0^{1/2} T^{-3\delta/2},\,\phi^{(j)}(v_1)\ll A_0 T^{-3\delta},\,\,j\geq 2,
\end{equation}and
\begin{equation}\notag
H^{(j)}(v_1)\ll_j (A+T^{-\delta}+T^{-\delta}A_0^{1/3})^j,\,\,j\geq 1.
\end{equation}Applying Lemma \ref{s1} with the parameters $(X, U^{-1})=(1, A+T^{-\delta}+T^{-\delta}A_0^{1/3})$, $W=T^{\epsilon}A_0 T^{-3\delta/2}$, $(Y,Q^{-1})=(A_0T^{-3\delta}, 1)$, we obtain
\begin{equation}\label{Ineg}
I\ll_{A} T^{-A\epsilon}\left(\frac{A_0^{1/2}T^{-3\delta/2}}{A+T^{-\delta}+T^{-\delta}A_0^{1/3}}\right)^{-A}+T^{-A\epsilon/2}.
\end{equation}Using $A= A_0^{1/2}T^{-3\delta/2-\epsilon}, T^{-\delta}\gg T^{\epsilon}A_0^{-1/3}$ and $A_0\gg T^{\epsilon}$, the term inside the parenthesis is easily seen to be $\gg 1$ and the claim follows.
\end{proof}
It follows from \eqref{I1dyadic} and Lemma \ref{I1local} that
\begin{equation}\label{I1localised}
I_1(\delta)\approx T^{-\delta}\int_{v_1}w(A(v_1+T^{\delta}(z_1-z_2)))G(v_1)e(\phi(v_1)/2\pi)\,dv_1.
\end{equation}Note that using \eqref{I1der}, \eqref{denoder} and \eqref{deltasize}, for $v_1$ in the support of $w(A(v_1+T^{\delta}(z_1-z_2)))G(v_1)$, we get
\begin{equation}\notag
\phi''(v_1)\gg A_0 T^{-3\delta}\left(1+O(A^{-1})\right)\asymp A_0 T^{-3\delta}
\end{equation}It is also clear that for $v_1$ in the support of $w(A(v_1+T^{\delta}(z_1-z_2)))G(v_1)$,
\begin{equation}\label{od}
\phi'(v_1)\ll T^{\epsilon}A_0^{1/2} T^{-3\delta/2},\,\,\,\phi^{(j)}(v_1)\ll_j  A_0 T^{-3\delta},\,\,j\geq 3.
\end{equation}Let $Y_0=A_0 T^{-3\delta}, Q_0=Y_0^{1/2}T^{-\epsilon}$. Since $Q_0\gg T^{3\epsilon}T^{-\epsilon}=T^{2\epsilon}$, the above inequalities implies that
\begin{equation}\notag
\phi^{(j)}(v_1)\ll_j Y_0 Q_0^{j-2}
\end{equation}holds for all $j\geq 1$. We also have
\begin{equation}\notag
\frac{\partial^j}{\partial v_1^j}w(A(v_1+T^{\delta}(z_1-z_2)))G(v_1)\ll_j (A+T^{-\delta}+T^{-\delta}A_0^{1/3})^j,
\end{equation}As earlier, plugging in $A= A_0^{1/2}T^{-3\delta/2-\epsilon}, T^{-\delta}\gg T^{\epsilon}A_0^{-1/3}$ and $A_0\gg T^{\epsilon}$, we can easily verify
\begin{equation}\notag
A+T^{-\delta}+T^{-\delta}A_0^{1/3}\ll T^{-\epsilon/2} Y_0^{1/2}
\end{equation}We are now ready to apply \ref{s2} to the integral $I_1(\delta)$ in  \eqref{I1localised} with the parameters, $ (X,V^{-1})=(1, A+T^{-\delta}+T^{-\delta}A_0^{1/3}), (Y, Q^{-1})= (Y_0Q_0^{-2}, Q_0)$, with the unique root $t_0=-T^{\delta}(z_1-z_2)$, to get the main term
\begin{equation}\notag
I_1(\delta)\sim T^{-\delta} \frac{e(\phi(t_0)/2\pi)}{\sqrt{\phi''(t_0)}}G(t_0).
\end{equation}Evaluating from \eqref{I1der}, we have
\begin{equation}\notag
\phi''(t_0)=\phi''(-T^{\delta}(z_1-z_2))=\left(-\frac{(\epsilon_2T_1-T_2)}{T_2\phi_R(\epsilon_1\alpha, z_2)}\right)\cdot \frac{R\alpha T^{-2\delta}(z_1-z_2)}{z_2^2}.
\end{equation}From their definitions, we also have
\begin{equation}\notag
\phi(t_0)=\phi (-T^{\delta}(z_1-z_2))=\psi(z_2)=-T_2\log\phi_R(\epsilon_1\alpha, z_2)+\epsilon_2T_1\log z_2,
\end{equation}and
\begin{equation}\notag
G(t_0)=G(-T^{\delta}(z_1-z_2))=F(T^{\delta}(z_2-z_1))z_2^{-1}\omega^{(\epsilon_1,\epsilon_2)}_R(\alpha, z_2).
\end{equation}
Substituting these expressions we obtain
\begin{equation}\notag
I_1(\delta)\sim \frac{1}{\alpha^{1/2}}\left(-\frac{(\epsilon_2T_1-T_2)}{T_2\phi_R(\epsilon_1\alpha, z_2)}\right)^{-1/2}\frac{F(T^{\delta}(z_2-z_1))\omega^{(\epsilon_1,\epsilon_2)}_R(\alpha, z_2)}{(R(z_1-z_2))^{1/2}}e(\psi(z_2)/2\pi).
\end{equation}Finally, summing over all the dyadic parts and substituting in \eqref{I1def}, we obtain
\begin{equation}\label{I1asymp}
I_1\sim \frac{1}{\alpha^{1/2}}\left(-\frac{(\epsilon_2T_1-T_2)}{T_2\phi_R(\epsilon_1\alpha, z_2)}\right)^{-1/2}\frac{\omega^{(\epsilon_1,\epsilon_2)}_R(\alpha, z_2)}{(R(z_1-z_2))^{1/2}}\left(\phi_{R}(\epsilon_1\alpha,z_2(\epsilon_1\alpha))\right)^{-iT_2}\left(z_2(\epsilon_1\alpha)\right)^{i\epsilon_2 T_1}.
\end{equation}This complete the analysis for the $I_1$ part.

Next consider the $I_2$ integral 
\begin{equation}\label{I2def2}
I_2=\int_{\mathbb{R}}w(T^{-\epsilon}A_0^{1/3}(v-z_1))v^{-1}\omega^{(\epsilon_1,\epsilon_2)}_R(\alpha,v)e(\psi(v)/2\pi) \,dv.
\end{equation}
Using \eqref{psider} and  \eqref{denoder}, for $v$ in the support of $w(T^{-\epsilon}A_0^{1/3}(v-z_1))\omega_R(\epsilon_1\alpha,v)$, we have
\begin{equation}\notag
\psi''(v)\asymp A_0 \Big(|z_1-z_2|+O(T^{\epsilon }A_0^{1/3})+O((R/(R+1))|z_1-z_2|T^{\epsilon }A_0^{-1/3}))\Big)
\end{equation}Since $|z_1-z_2|\gg T^{2\epsilon}A_0$ from \eqref{rootgap}, and $R/(R+1)A_0^{-1/3}T^{\epsilon}= T^{\epsilon}((R+1)\alpha)^{-1/3}\asymp T^{-1/3+\epsilon} $, it follows that
\begin{equation}\notag
\psi''(v)\asymp A_0|z_1-z_2|.
\end{equation}Furthermore, we have
\begin{equation}\notag
\psi'(v)\ll A_0 |z_1-z_2|T^{\epsilon}A^{-1/3}_0\ll A_0 |z_1-z_2|
\end{equation}since $A_0\gg T^{3\epsilon} $, and using \eqref{denoder} we have
\begin{equation}\label{3term}
\psi^{(j)}(v)\ll_j  A_0 |z_1-z_2|,\,\,j\geq 3.
\end{equation}For the weight function, we have
\begin{equation}\notag
\frac{\partial ^j}{\partial v^j}w(T^{-\epsilon}A_0^{1/3}(v-z_1))v^{-1}\omega_R(\epsilon_1\alpha,v)\ll_j (1+T^{-\epsilon}A_0^{1/3})^j,\,\,j\geq 1.
\end{equation}Since $|z_1-z_2|\gg T^{2\epsilon}A_0$ and $A_0\gg T^{3\epsilon}$, it follows that
\begin{equation}\notag
1+T^{-\epsilon}A_0^{1/3}\ll T^{-\epsilon}(A_0 |z_1-z_2|)^{1/2}.
\end{equation}Hence, we can apply Lemma \ref{s2} to the $I_2$ integral in \eqref{I2def2} with the parameters $ (X, V^{-1})= (1, 1+T^{-\epsilon}A_0^{1/3}), (Y, Q^{-1})=(A_0 |z_1-z_2|, 1)$ and with the unique root $z_1$, to get 
\begin{equation}\notag
I_2\sim \frac{e(\psi(z_1)/2\pi)}{\sqrt{\psi''(z_1)}}z_1^{-1}\omega_R(\alpha, z_1).
\end{equation}From \eqref{psider}, we get
\begin{equation}\notag
\psi''(z_1)= \left(\frac{(\epsilon_2T_1-T_2)}{T_2\phi_R(\epsilon_1\alpha, z_1)}\right)\frac{R\alpha (z_1-z_2)}{z_1^2}
\end{equation}Substituting the above and $\psi(z_1)=-T_2\log\phi_R(\epsilon_1\alpha, z_1)+\epsilon_2T_1\log z_1$, we obtain
\begin{equation}\label{I2asymp}
I_2\sim \frac{1}{\alpha^{1/2}}\left(\frac{(\epsilon_2T_1-T_2)}{T_2\phi_R(\epsilon_1\alpha, z_1)}\right)^{-1/2} \frac{\omega_R(\alpha, z_1)}{(R(z_1-z_2))^{1/2}}\left(\phi_{R}(\epsilon_1\alpha,z_1(\epsilon_1\alpha))\right)^{-iT_2}\left(z_1(\epsilon_1\alpha)\right)^{i\epsilon_2 T_1}
\end{equation}Proof upto \eqref{eq1123} of Lemma \ref{11.1} is now complete after combining \eqref{I1asymp} and \eqref{I2asymp}. It remains to prove \eqref{eq1126} to complete the proof of part \ref{asymptotic}.\\

Note that from \eqref{quadratic}, we have
\begin{equation}\label{roots}
z_1(\epsilon_1\alpha)=R^{-1}c_1\left(\frac{\pi U}{\alpha}\right)+\sqrt{f(\alpha)}, \,\,\,\,z_2(\epsilon_1\alpha)=R^{-1}c_1\left(\frac{\pi U}{\alpha}\right)-\sqrt{f(\alpha)}.
\end{equation}where \[f(\alpha)=R^{-2}\left(\frac{\pi U}{\alpha}\right)^2+R^{-2}c_0\] and  $c_1=\frac{\epsilon_2T_1}{\epsilon_1(\epsilon_2T_1-T_2)}\asymp 1$, $c_0=\frac{\epsilon_2T_1+T_2}{\epsilon_2T_1-T_1}\asymp 1$. It is easy to see that
\begin{equation}\label{fder}
\frac{\partial^j}{\partial\alpha^j}f(\alpha)\ll_j |\alpha|^{-j}R^{-2}\left(\frac{ U}{|\alpha|}\right)^2,\,\,j\geq 1.
\end{equation}Recall the hypothesis $f(\alpha)\asymp R^{-2}T^{-\delta}$. Using these informations, one deduces from the Fa\'a di Bruno's formula \ref{bruno}, that
\begin{equation}\notag
\frac{\partial^j}{\partial\alpha^j}\sqrt{f(\alpha)}\ll_j \sum_{\substack{j_1,j_2,\cdots,j_n\\}} (R^{-2}T^{-\delta})^{1/2-\sum j_i}\prod_{i=1}^{n} \left(|\alpha|^{-i}R^{-2}\left(\frac{ U}{|\alpha|}\right)^2\right)^{j_i},
\end{equation}where the sum is over all $n$-tuples of non-negative integers $(j_1,j_2,\cdots,j_n)$ satisfying
\begin{equation}\notag
1\cdot j_1+2\cdot j_2+3\cdot j_3+\cdots+n\cdot j_n=j.
\end{equation}Simplifying we obtain
\begin{equation}\notag
|\alpha|^j\frac{\partial^j}{\partial\alpha^j}\sqrt{f(\alpha)}\ll_j (R^{-2}T^{-\delta})^{1/2}\sum \left(T^{\delta}\left(\frac{U}{|\alpha|}\right)^2\right)^{\sum j_i}.
\end{equation}Since $\sum j_i\leq j$, we conclude
\begin{equation}\label{imgrootder}
|\alpha|^j\frac{\partial^j}{\partial\alpha^j}\sqrt{f(\alpha)}\ll_j (R^{-2}T^{-\delta})^{1/2} \left(1+T^{\delta}\left(\frac{U}{|\alpha|}\right)^2\right)^j.
\end{equation}On the other hand we have
\begin{equation}\label{realrootder}
|\alpha|^j\frac{\partial^j}{\partial\alpha^j}\left(R^{-1}\frac{\pi U}{\alpha}\right)\ll_j R^{-1}\frac{U}{|\alpha|}.
\end{equation}By the simple AM-GM inequality, the right hand side of \eqref{realrootder} gets     absorbed into \eqref{imgrootder}. Substituting in \eqref{roots} we obtain
\begin{equation}\label{rootder}
|\alpha|^j\frac{\partial^j}{\partial\alpha^j}z_i(\epsilon_1\alpha)\ll_j R^{-1}T^{-\delta/2} \left(1+T^{\delta}\left(\frac{U}{|\alpha|}\right)^2\right)^j,\,\,\,j\geq 1.
\end{equation}One can similarly show that
\begin{equation}\label{invrootder}
|\alpha|^j\frac{\partial^j}{\partial\alpha^j}z_i^{-1}(\epsilon_1\alpha)\ll_j R^{-1}T^{-\delta/2} \left(1+T^{\delta}\left(\frac{U}{|\alpha|}\right)^2\right)^j,\,\,\,j\geq 1.
\end{equation}Now
\begin{equation}\notag
\phi_R(\alpha, z_i(\epsilon_1\alpha))=-\frac{R\alpha}{T_2}\left(z_i(\epsilon_1\alpha)-\frac{1}{R^2z_i(\epsilon_1\alpha)}\right)-\frac{2\pi U}{T_2},
\end{equation}so that
\begin{equation}\notag
\begin{aligned}
\frac{\partial^j}{\partial\alpha^j}\phi_R(\alpha, z_i(\epsilon_1\alpha))=-\frac{R\alpha}{T_2}\frac{\partial^j}{\partial\alpha^j}&\left(z_i(\epsilon_1\alpha)-\frac{1}{R^2z_i(\epsilon_1\alpha)}\right)\\
&-\frac{R}{T_2}\frac{\partial^{j-1}}{\partial\alpha^{j-1}}\left(z_i(\epsilon_1\alpha)-\frac{1}{R^2z_i(\epsilon_1\alpha)}\right).
\end{aligned}
\end{equation}Substituting the bounds \eqref{rootder} and \eqref{invrootder} into above, we obtain
\begin{equation}\notag
\begin{aligned}
\frac{\partial^j}{\partial\alpha^j}\phi_R(\alpha, z_i(\epsilon_1\alpha))\ll_j T^{-\delta/2}(\alpha/T)|\alpha|^{-j}&(1+T^{\delta}(U/\alpha)^2)^j\\
&+T^{-\delta/2}T^{-1}|\alpha|^{-j+1}(1+T^{\delta}(U/\alpha)^2)^{j-1}.
\end{aligned}
\end{equation}Since $1+T^{\delta}(U/\alpha)^2>1$, we conclude
\begin{equation}\label{phiRder}
|\alpha|^j\frac{\partial^j}{\partial\alpha^j}\phi_R(\alpha, z_i(\epsilon_1\alpha))\ll_j T^{-\delta/2}(\alpha/T)(1+T^{\delta}(U/\alpha)^2)^j,\,\,j\geq 1.
\end{equation}Using the last inequality,\eqref{rootder} and the Fa\'a di Bruno's formula \ref{bruno}, we conclude that for \[\omega_R(\epsilon_1\alpha,z_i(\epsilon_1\alpha))=F(z_i(\epsilon_1\alpha))U(z_i(\epsilon_1\alpha)-R^{-1})\mathcal{W}\left(\phi_R(\epsilon_1\alpha, z_i(\epsilon_1\alpha))\right),\]
\begin{equation}\label{weightder}
|\alpha|^j\frac{\partial^j}{\partial\alpha^j}\omega_R(\epsilon_1\alpha,z_i(\epsilon_1\alpha))\ll_j \Big((1+T^{-\delta/2}(R^{-1}+\alpha/T))(1+T^{\delta}(U/\alpha)^2)\Big)^j,\,\,j\geq 1.
\end{equation}Finally,  since $(z_1(\epsilon_1\alpha)-z_2(\epsilon_1\alpha))=\sqrt{f(\alpha)}\asymp R^{-1}T^{-\delta/2}$, using the derivative bound \eqref{fder}, and the Fa\'a di Bruno's formula \ref{bruno}, we deduce
\begin{equation}\label{rootdiffder}
\begin{aligned}
&\frac{\partial^j}{\partial\alpha^j} (RT^{\delta/2}(z_1(\epsilon_1\alpha)-z_2(\epsilon_1\alpha)))^{-1/2}\ll_j \sum_{\substack{j_1,j_2,\cdots,j_n\\\sum kj_k=j}} \left(RT^{\delta/2}|\alpha|^{-k}R^{-2}(U/|\alpha|)^2\right)^{j_k}\\
&=|\alpha|^{-j}\sum_{\substack{j_1,j_2,\cdots,j_n\\\sum kj_k=j}}(R^{-1}T^{\delta/2}(U/|\alpha|)^2)^{\sum j_k}\ll |\alpha|^{-j}\Big(1+R^{-1}T^{\delta/2}(U/|\alpha|)^2\Big)^j
\end{aligned}
\end{equation}Combining \eqref{weightder} and \eqref{rootdiffder} we obtain
\begin{equation}\notag
|\alpha|^j\frac{\partial^j}{\partial\alpha^j}\left(\frac{\omega_R(\epsilon_1\alpha,z_i(\epsilon_1\alpha))}{RT^{\delta/2}(z_1(\epsilon_1\alpha)-z_2(\epsilon_1\alpha)))^{1/2}}\right)\ll_j \Big((1+T^{-\delta/2}(R^{-1}+\alpha/T))(1+T^{\delta}(U/\alpha)^2)\Big)^j.
\end{equation}The proof of \eqref{eq1126} is complete after observing that $\alpha/T\asymp (R+1)^{-1}$ and 
\begin{equation}\notag
T^{\delta}(U/\alpha)^2\ll 1+T^{\delta}
\end{equation}which is trivially true if $U/\alpha\ll 1$, and follows the definition $T^{-\delta}\asymp (\pi U/\alpha)^2-c_0$ if $U/\alpha\gg 1$.
\\

Let us proceed to prove part \ref{thirdderbd} \eqref{bd2}, which is the third derivative bounds for the integral $I_{R}^{(\epsilon_1,\epsilon_2)}(U,\alpha)$. Recall that
\begin{equation}\notag
I_{R}^{(\epsilon_1,\epsilon_2)}(U,\alpha)=\int_{v\in\mathbb{R}}v^{-1}\omega_R(\epsilon_1\alpha,v)e(\psi(v)/2\pi)\,dv
\end{equation}and from \eqref{psider}
\begin{equation}\label{psidernew}
\psi'(v)=\frac{R\alpha(\epsilon_2T_1-T_2)(v-z_1(\epsilon_1\alpha))(v-z_2(\epsilon_1\alpha))}{T_2v^2\phi_R(\epsilon_1\alpha,v))}.
\end{equation}The standing  assumptions for the case under consideration are
\begin{equation}\notag
R\gg T^{3\epsilon}\min\{\alpha^{-1}, T^{-1}\}\,\,\,\,\text{and}\,\,\,\,\,z_1(\epsilon_1\alpha)-z_2(\epsilon_1\alpha)\ll T^{2\epsilon}A_0^{-1/3},
\end{equation}where $A_0=R^3\alpha/(R+1)^2$. It is enough to work around one of root $z_1$(say) and furthermore, by introducing dyadic partition of unity, it is enough to prove the claim for each dyadic piece
\begin{equation}\notag
I(\delta):=\int_{v\in\mathbb{R}}F(T^{\delta}(v-\Re(z_1)))v^{-1}\omega_R(\epsilon_1\alpha,v)e(\psi(v)/2\pi)\,dv,
\end{equation}where we can assume $T^{-\delta}\ll 1+R^{-1}$ following from the earlier remark \eqref{deltasize}.
\begin{lemma}\label{•}
$I(\delta)$ is negligibly small unless 
\begin{equation}\notag
T^{-\delta}\ll T^{2\epsilon}A^{-1/3}_0,
\end{equation}for some large enough implied constant.
\end{lemma}	
\begin{proof}
Suppose that for some large enough implied constant,
\begin{equation}\notag
T^{-\delta}\gg T^{2\epsilon}A^{-1/3}_0.
\end{equation}Then from \eqref{psidernew}, for $v$ in the support of $F(T^{\delta}(v-\Re(z_1)))\omega^{(\epsilon_1,\epsilon_2)}_R(\alpha,v)$,
\begin{equation}\notag
\psi'(v)\gg A_0 T^{-2\delta}.
\end{equation}Furthermore, using \eqref{denoder} and $T^{-\delta}\ll 1+R^{-1}$, we get
\begin{equation}\notag
\begin{aligned}
\psi''(v)&\ll A_0 \Big(T^{-\delta}+T^{-2\delta}(R/(R+1))\Big)\ll A_0T^{-\delta}\\
\psi^{(j)}(v)&\ll_j A_0 \Big(T^{-2\delta}(R/(R+1))^{j-1}+(R/(R+1))^{j-2}T^{-\delta}+(R/(R+1))^{j-3}\Big)\\
&\ll A_0 ,\,\,j\geq 3.
\end{aligned}
\end{equation} We also have
\begin{equation}\notag
\frac{\partial ^j}{\partial v^j}F(T^{\delta}(v-\Re(z_1)))v^{-1}\omega_R(\epsilon_2\alpha,v)\ll_j (1+T^{\delta})^j,\,j\geq 1.
\end{equation}Applying Lemma \ref{s1} with $(X,U^{-1})=(1,1+T^{\delta}), W=A_0T^{-2\delta}$ and \\
$(Y, Q^{-1})=(A_0T^{-\delta} \min\{1, T^{-2\delta}\}, \max\{1, T^{\delta}\})$, we obtain
\begin{equation}\notag
I(\delta)\ll_A \left(\frac{A_0T^{-2\delta}}{1+T^{\delta}}\right)^{-A}+\left(\frac{A_0T^{-2\delta}}{A_0^{1/2}T^{-\delta/2}}\right)^{-A}
\end{equation}As earlier, using the $T^{-\delta}\gg T^{2\epsilon}A_0^{-1/3}$ and the basic inequality $A_0\gg T^{\epsilon}$, one can verify that both the term inside the parenthesis above is $\gg T^{\epsilon}$. The lemma follows.
\end{proof} Assuming $T^{-\delta}\ll T^{2\epsilon}A_0^{-1/3}$, trivially executing the $v$ integral one obtains
\begin{equation}\notag
I(\delta)\ll T^{-\delta}\sup {v^{-1}\omega_R(\alpha,v)}\ll T^{2\epsilon}A_0^{-1/3}(R/(R+1))=T^{2\epsilon} ((R+1)\alpha)^{-1/3}
\end{equation}This completes the proof of part \ref{thirdderbd}.\\

It remains to prove part \ref{negsmall}. Recall
\begin{equation}\notag
I_{R}^{(\epsilon_1,\epsilon_2)}(U,\alpha)=\int_{\mathbb{R}}v^{-1}\omega_R(\epsilon_1\alpha,v) e\left(\frac{\psi(v)}{2\pi}\right)\,dv,
\end{equation}where
\begin{equation}\notag
\psi(v)=-T_2\log\phi_R(\epsilon_1\alpha,v)+\epsilon_2T_1\log v,\,\,\phi_R(\alpha,v):=-2\pi T_2^{-1}\left(\frac{\alpha(Rv-R^{-1}v^{-1})}{2\pi}+U\right),
\end{equation}and $\omega_R(\epsilon_1\alpha,v)=U(v-R^{-1})\mathcal{W}\left(\phi(\epsilon_1\alpha,v)\right)$ for some compactly supported smooth functions $U,\mathcal{W}$ with bounded derivatives. Now,
\begin{equation}\label{sider}
\begin{aligned}
\psi'(v)&=\frac{\epsilon_1\alpha(R+(Rv^2)^{-1})}{\phi_R(\alpha,v)}+\frac{\epsilon_2T_1}{v}\\
&=\frac{\epsilon_1\alpha(Rv+(Rv)^{-1})-(\epsilon_2T_1/T_2)(2\pi U)-(\epsilon_2T_1/T_2)\alpha(Rv-(Rv)^{-1})}{v\phi_R(\alpha, v)}.
\end{aligned}
\end{equation}Since $R\ll T^{-\epsilon}, Rv=1+O(R)$, it follows
\begin{equation}\label{npsider}
\psi'(v)=\frac{2\epsilon_1\alpha-(\epsilon_2T_1/T_2)(2\pi U)+O(R\alpha)}{v\phi_R(\alpha, v)}
\end{equation}Also, using the fact that  $\frac{\partial^j}{\partial v^j}(1/v\phi_R(\alpha,v))\ll_j (R/(R+1))^{j+1}$, which can be proved in a similar fashion as \eqref{denoder}, it is easy to see that
\begin{equation}\notag
\psi^{(j)}(v)\ll_j R^2\alpha,\,\,j\geq 2.
\end{equation}So if $|\epsilon_1\alpha -(\epsilon_2T_1/T_2)(2\pi U)|\gg T^{\epsilon}(R\alpha+R^{-1})$, then
\begin{equation}\notag
\psi'(v)\gg T^{\epsilon}(1+R^2\alpha).
\end{equation}Appealing to Lemma \ref{s1} with the parameters $(X,U^{-1})=(1,1), W=T^{\epsilon}(1+R^2\alpha)$ and $ (Y,Q^{-1})=(R^2\alpha, 1)$, it follows that $I_{R}^{(\epsilon_1,\epsilon_2)}(U,\alpha)$ is negligibly small.

Now in the case $|\epsilon_1\alpha -(\epsilon_2T_1/T_2)(2\pi U)|\ll T^{\epsilon}(R\alpha+R^{-1})$, from \eqref{npsider} we get
\begin{equation}\notag
\psi'(v)=\frac{\epsilon_1\alpha(R+(Rv^2)^{-1})}{\phi_R(\alpha,v)}+\frac{\epsilon_2T_1}{v}\ll T^{\epsilon}(1+R^2\alpha),
\end{equation}and from \eqref{sider} we have
\begin{equation}\notag
\begin{aligned}
\psi''(v)&=\frac{(\alpha^2/T_2)(R+(Rv^2)^{-1})^2}{\phi_R^2(\epsilon_1\alpha, v)}-\frac{(2\epsilon_1\alpha)/(Rv^3)}{\phi_R(\epsilon_1\alpha,v)}-\frac{\epsilon_2T_1}{v^2}\\
&=\frac{(\alpha^2/T_2)(R+(Rv^2)^{-1})^2}{\phi_R^2(\epsilon_1\alpha, v)}-v^{-1}\left(\frac{\epsilon_1\alpha(R+(Rv^2)^{-1})}{\phi_R(\alpha,v)}+\frac{\epsilon_2T_1}{v}+O(R^2\alpha)\right)\\
&\gg \left|\frac{R^2\alpha^2}{T}-T^{\epsilon}(R+R^3\alpha)\right|\gg R^2\alpha^2/T\asymp R^2\alpha,
\end{aligned}
\end{equation}since we are in the range $\min\{T^{-1},\alpha^{-1}\}\ll R\ll T^{-\epsilon}$ and $T/\alpha\asymp 1+R\asymp 1$. Finally, observing that $v^{-1}\omega_R(\epsilon_1\alpha,v)\ll R, \frac{\partial }{\partial v}(v^{-1}\omega_R(\epsilon_1\alpha,v))\ll R $, and using the second derivative bound we obtain
\begin{equation}\notag
I_{R}^{(\epsilon_1,\epsilon_2)}(U,\alpha)\ll R (R^2\alpha)^{-1/2}=\alpha^{-1/2}.
\end{equation}

\end{proof}

\section*{Acknowledgements} The author would like to thank Prof. Ritabrata Munshi for his encouragement, suggestions and support throughout the work.

\bibliographystyle{abbrv}

\end{document}